\documentclass[a4paper,11pt]{article}
\usepackage{amsmath}
\usepackage{amssymb} 
\usepackage{theorem}
\usepackage{fullpage}
\usepackage[utf8x]{inputenc}

\usepackage{amssymb, latexsym}
\usepackage{graphicx}
\newtheorem{theo}{Theorem}
\newtheorem{coro}{Corollary}
\newtheorem{lemm}{Lemma}
\newtheorem{prop}{Proposition}

\theorembodyfont{\rm} 
\newtheorem{defi}{Definition}
\newtheorem{rema}{Remark}
\newtheorem{exem}{Example}

\newenvironment{proof}{\noindent {\it Proof}.}{\hfill$\Box$}

\numberwithin{equation}{section}
\usepackage{enumerate}
\makeatletter
\renewcommand\theenumi{\@alph\c@enumi}
\renewcommand\theenumii{\@alph\c@enumii}
\renewcommand\theenumiii{\@alph\c@enumiii}
\renewcommand\theenumiv{\@alph\c@enumiv}

\let\@listlla\list
\def\list#1#2{\@listlla{#1}{#2%
      \itemsep=2pt\parsep=0pt\topsep=3pt plus 1pt minus 1 pt}}
\makeatother
\newenvironment{case}[2][Case]{%
  \trivlist \item[\hskip\labelsep{\bfseries #1 #2.}]\begin{em}}{
  \end{em}\endtrivlist}

\newcommand{\superparagraph}[1]{\trivlist
  \item[\hskip\labelsep{\bfseries\boldmath$\bullet$%
     \ #1.}]\endtrivlist
}
\newcommand{\subsubR}{_{_{\mbox{\scriptsize$\IR$}}}}
\newcommand{\Clos}[1]{\ensuremath{\operatorname{Clos}\left(#1\right)}}
\DeclareMathOperator{\Int}{Int}
\DeclareMathOperator{\Id}{Id}
\DeclareMathOperator{\Card}{Card}
\DeclareMathOperator{\Per}{Per}
\newcommand{\PerR}{\ensuremath{\Per\subsubR}}
\DeclareMathOperator{\Orb}{Orb}
\DeclareMathOperator{\B}{B}
\DeclareMathOperator{\Fol}{Fol}
\DeclareMathOperator{\Const}{Const}
\DeclareMathOperator{\Rot}{Rot}
\newcommand{\RotR}{\ensuremath{\Rot\subsubR}}
\newcommand{\set}[2]{\ensuremath{\{#1 \,\colon #2\}}}
\makeatletter
\def\@map#1#2[#3]{\mbox{$#1 \colon #2 \longrightarrow #3$}}
\def\map#1#2{\@ifnextchar [{\@map{#1}{#2}}{\@map{#1}{#2}[#2]}}
\def\@chull#1[#2]{\ensuremath{\langle #1 \rangle_{_{#2}}}}
\def\chull#1{%
  \@ifnextchar [{\@chull{#1}}{\ensuremath{\langle #1 \rangle}}%
}
\makeatother
\newcommand{\evalat}[1]{\bigr\rvert_{#1}}
\newcommand{\IN}{\ensuremath{\mathbb{N}}}
\newcommand{\IZ}{\ensuremath{\mathbb{Z}}}
\newcommand{\IQ}{\ensuremath{\mathbb{Q}}}
\newcommand{\IR}{\ensuremath{\mathbb{R}}}
\newcommand{\IC}{\ensuremath{\mathbb{C}}}

\newcommand{\eps}{\varepsilon}
\newcommand{\CA}{\mathcal{A}}
\newcommand{\CB}{\mathcal{B}}
\newcommand{\CC}{\mathcal{C}}
\newcommand{\CD}{\mathcal{D}}
\newcommand{\CE}{\mathcal{E}}
\newcommand{\CG}{\mathcal{G}}
\newcommand{\CL}{\mathcal{L}}
\newcommand{\CP}{\mathcal{P}}
\newcommand{\modi}{\ensuremath{\kern -0.55em\pmod{1}}}
\newcommand{\rhos}[1]{\rho_{_{#1}}}
\newcommand{\orhos}[1][F]{\overline{\rho}_{_{#1}}}
\newcommand{\urhos}[1][F]{\underline{\rho}_{_{#1}}}
\newcommand{\Fcrbas}[2]{#1^r_{#2}}
\newcommand{\Fncr}{\Fcrbas{F}{n}}
\newcommand{\Fcr}{\Fcrbas{F}{1}}
\DeclareMathOperator{\InfX}{\mbox{\upshape\sffamily\bfseries T}}
\newcommand{\InfG}{\InfX^{\circ}}
\newcommand{\Li}[1][T]{\ensuremath{\mathcal{C}_1(#1)}}
\newcommand{\labelarrow}[1]{\stackrel{#1}{\longrightarrow}}
\newcommand{\pluscover}[1]{\xrightarrow[#1]{%
   \hspace*{.25em}+\hspace*{.1em}}
}
\title{Rotation sets for graph maps of degree~1
\footnotetext{{\it Math. classification:} 37E45, 37E25, 54H20, 37E15.}
\footnotetext{  }
\footnotetext{Annales de l'Institut Fourier, {\bf 58}, No. 4, 1233-1294, 2008.}
}
\author{Llu\'{\i}s Alsed\`a and Sylvie Ruette}

\date{}
\begin{document}
\maketitle
\begin{abstract}
For a continuous map on a topological graph containing a loop $S$ it
is possible to define the degree (with respect to the loop $S$) and,
for a map of degree $1$, rotation numbers. We study the rotation set
of these maps and the periods of periodic points having a given
rotation number. We show that, if the graph has a single loop $S$ then
the set of rotation numbers of points in $S$ has some properties
similar to the rotation set of a circle map; in particular it is a
compact interval and for every rational $\alpha$ in this interval
there exists a periodic point of rotation number $\alpha$.

For a special class of maps called combed maps, the rotation set
displays the same nice properties as the continuous degree one circle
maps.
\end{abstract}

\section*{Introduction}

One of the basic problems in combinatorial and topological dynamics is
the characterisation of the sets of periods in dimension one. This
problem has its roots and motivation in the striking Sharkovskii
Theorem \cite{Sar,Shartrans}. Since then, a lot of effort has been
spent in finding characterisations of the set of periods for more
general one dimensional spaces.

One of the lines of generalisation of Sharkovskii Theorem consists on
characterising the possible sets of periods of continuous self maps on
trees. The first remarkable results in this line after \cite{Sar} are
due to Alsed\`a, Llibre and Misiurewicz \cite{ALM2} and Baldwin
\cite{Bal}. In \cite{ALM2} it is obtained the characterisation of the
set of periods of the continuous self maps of a 3-star with the
branching point fixed in terms of three linear orderings, whereas in
\cite{Bal} the characterisation of the set of periods of all
continuous self maps of $n$-stars is given (an $n$-star is a tree
composed of $n$ intervals with a common endpoint). Further extensions
of Sharkovskii Theorem are due to:
\begin{list}{$\bullet$}{\setlength{\leftmargin}{\parindent}}
\item Baldwin and Llibre \cite{BL} to continuous maps on trees such
that all the branching points are fixed,
\item Bernhardt \cite{Bern} to continuous maps on trees such that all
the branching points are periodic,
\item Alsed\`a, Juher and Mumbr\'u \cite{AJM1, AJM2, AJM3, AJM4} to
the general case of continuous tree maps.
\end{list}

Another line of generalisation of Sharkovskii Theorem is to consider
spaces that are not contractile to a point. In particular topological
graphs which are not trees, the circle being the simplest one. This
case displays a new feature: While the sets of periods of continuous
maps on trees can be characterised using only a finite number of
orderings, the sets of periods of continuous circle maps of degree one
contain the set of all denominators of all rationals (not necessarily
written in irreducible form) in the interior of an interval of the
real line. As a consequence, these sets of periods cannot be expressed
in terms of a finite collection of orderings. The result which
characterises the sets of periods of continuous circle maps of degree
one is due to Misiurewicz \cite{Mis} and uses as a key tool the
rotation theory. Indeed, the sets of periods are obtained from the
rotation interval of the map.

The characterisation of the sets of periods for circle maps of degree
different from one is simpler than the one for the case of degree one.
It is due to Block, Guckenheimer, Misiurewicz and Young \cite{BGMY}.

Finding a generalisation of the Sharkovskii Theorem for self maps of a
topological graph which is not the circle is a big challenge and in
general it is not known what the sets of periods may look like.
However, in this setting, one expects to find at least sets of periods
of all possible types appearing for tree and circle maps.

Two motivating results that give some insight on the kind of sets of
periods that one can find in this setting are \cite{LLl} and
\cite{LPR}. The first of them deals with continuous self maps on a
graph $\sigma$ consisting on a circuit and an interval attached at a
unique branching point $b$ such that the maps fix $b$. The second one
studies the continuous self maps of the 2-foil (that is, the graph
consisting on two circles attached at a single point).

Our aim is to go forward in the generalisation of \cite{LLl} by using
the ideas and techniques of \cite{Mis, BGMY}. To this end we need to
develop a rotation theory for continuous self maps of degree one of
topological graphs having a unique circuit and, afterwards, we need to
apply this theory to the characterisation of the sets of periods of
such maps.

In this paper we propose a rotation theory for the above class of maps
and we study the relation between the rotation numbers and the
periodic orbits. The use of this theory in the characterisation of the
sets of periods of such maps will be the goal of a future project.

A rotation theory is usually developed in the universal covering space
by using the liftings of the maps under consideration. It turns out
that the rotation theory on the universal covering of a graph with a
unique circuit can be easily extended to a wider family of spaces.
These spaces are defined in detail in Subsection~\ref{subsec:LS} and
called \emph{lifted spaces}. Each lifted space $T$ has a subset
$\widehat{T}$ homeomorphic to the real line $\IR$ that corresponds to
an ``unwinding'' of a distinguished circuit of the original space.

In the rest of this section (and in fact in the whole paper) we will
abuse notation and denote the set $\widehat{T}$ by $\IR$ for
simplicity.

Given a lifted space $T$ and a map $F$ from $T$ to itself of degree
one, there is no difficulty to extend the definition of rotation
number to this setting in such a way that every periodic point still
has a rational rotation number as in the circle case. However, the
obtained rotation set $\Rot(F)$ may not be connected and we do not
know yet whether it is closed. Despite of this fact, the set
$\RotR(F)$ corresponding to the rotation numbers of all points
belonging to $\IR$, has properties which are similar to (although
weaker than) those of the rotation interval for a circle map of degree
one.

Also, there is a special class of degree one continuous maps on lifted
spaces that we call \emph{combed maps}, whose rotation set displays
the same nice properties as the continuous degree one circle maps.

The paper is organised as follows. Section~\ref{sec:DEP} is devoted to
fixing the notation, to defining the notion of rotation number and
rotation set in this setting, and to studying the basic properties of
this set. In Section~\ref{sec:pos-cov} we introduce the technical
notion of \emph{positive covering} and, by means of its use, we prove
a result that will be used throughout the paper.

Section~\ref{sec:rotation-set} is devoted to studying the
basic properties of the rotation set. It is divided into two
subsections. In the first one (Subsection~\ref{subsec:concompRS}) we
study the connectedness and compactness of the rotation set whereas in
the second one (Subsection~\ref{sec:periods-X}) we describe the
information on the periodic orbits of the map which is carried out by
the rotation set.

In Section~\ref{sec:circular} we define the combed maps and, for this
class of maps, we study the special features of the rotation set and
its relation with the set of periods.

Section~\ref{sec:graph} specialises the results obtained previously in
the particular case when the lifted space is a graph. Finally,
Section~\ref{sec:examples} is devoted to showing some examples and
counterexamples to illustrate some previous comments and results.

We thank an anonymous referee for detailed and clever comments that
helped us improving the writing of a previous version of the paper,
and Bill Allombert for Lemma~\ref{lem:BA}.

\section{Definitions and elementary properties}\label{sec:DEP}
\subsection{Lifted spaces and retractions}\label{subsec:LS}

The aim of this subsection is to define in detail the class of
\emph{lifted spaces} where we will develop the rotation theory. They
are obtained from a metric space by unwinding one of its loops.
This gives a new space that contains a subset homeomorphic to the real
line and that is ``invariant by a translation''. This construction
mimics the process of considering the universal covering space of a
compact connected topological graph that has a unique loop.

Before defining lifted spaces we will informally discuss a couple of
examples to fix the ideas. Consider the topological graph $G$
represented in Figure~\ref{fig:hatG}. The unwinding of $G$ with
respect to the loop $S$ is the infinite graph $\widehat{G}$ which is
made up of infinitely many subspaces $(\widehat{G}_n)_{n\in\IZ}$ that
are all homeomorphic by a translation $\tau$. Moreover, there is a
continuous projection $\map{\pi}{\widehat{G}}[G]$ such that
$\pi\evalat{\Int(\widehat{G}_n)}$ is a homeomorphism onto
$G\setminus\{x_0\}$ for each $n \in \IZ$, and $\pi(\tau (y))=\pi(y)$
for all $y\in \widehat{G}$. The set $\pi^{-1}(S)$ is homeomorphic to
the real line. If we imagine that the loop $S$ has length $1$ and that
$x_0$ is the origin, then it is natural to consider a homeomorphism
$\map{h}{\IR}[\pi^{-1}(S)]$ such that $\pi^{-1}(x_0)=h(\IZ)$. In this
setting, $\tau(h(x))=h(x+1)$ for all $x\in\IR$.

Note that, since $G$ has more than one loop, $\widehat{G}$ is not the
universal covering of $G$.

\begin{figure}[htb]
\centerline{\includegraphics[width=25pc]{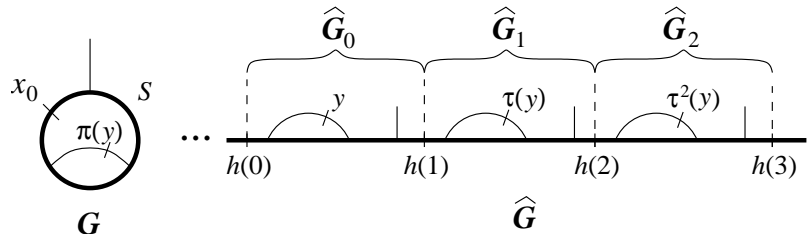}}
\caption{The graph $G$, on the left, is unwound with respect to the
bold loop $S$, on the right. In $\widehat{G}$, the bold line is
$\pi^{-1}(S)=h(\IR)$ and $h(\IZ) = \pi^{-1}(x_0)$. $\widehat{G}$ in
this example will not be considered as a lifted space since it cannot
be retracted to $h(\IR)$.\label{fig:hatG}}
\end{figure}

In a similar way, we can unwind any connected compact metric $X$
with a loop, as in Figure~\ref{fig:InfX}. These two examples have a
main difference: the space $\widehat{X}$ shown in
Figure~\ref{fig:InfX} can be ``retracted'' to $h(\IR)$ because the
closure of any connected component of $\widehat{X}\setminus h(\IR)$
meets $h(\IR)$ at a single point, whereas this property does not hold
for $\widehat{G}$ in Figure~\ref{fig:hatG}. Notice that the unwinding
of a graph with a single loop always has this property.  In this
paper, we deal with spaces $\widehat{X}$ of the type shown in
Figure~\ref{fig:InfX}.

\begin{figure}[htb]
\centerline{\includegraphics[width=25pc]{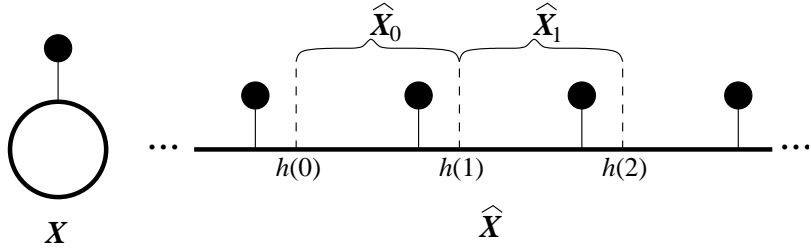}}
\caption{The unwinding of a connected compact metric space $X$ with a
loop. In this example $\widehat{X}$ can be retracted to
$h(\IR)$.\label{fig:InfX}}
\end{figure}

Now we formalise the definition of this class of spaces.

\begin{defi}\label{def:liftedspace}
Let $T$ be a connected metric space. We say that $T$ is a
\emph{lifted space} if there exists a homeomorphism $h$ from $\IR$
into $T$, and a homeomorphism $\map{\tau}{T}$ such that
\begin{enumerate}[(i)]
\item $\tau(h(x)) = h(x+1)$ for all $x\in\IR$,
\item the closure of each connected component of  $T \setminus h(\IR)$
is a compact set that intersects $h(\IR)$ at a single point, and
\item the number of connected components $C$ of $T \setminus h(\IR)$
such that $\Clos{C} \cap h([0,1]) \neq \emptyset$ is finite.
\end{enumerate}
The class of all lifted spaces will be denoted by $\InfX$.
\end{defi}

\begin{rema}\label{rem:tauplusone}
By replacing $h(x)$ by $h(x+a)$ for some appropriate $a$, if
necessary, we may assume that $h(\IZ)$ does not intersect the closure
of any connected component of $T \setminus h(\IR)$. In this situation,
for every $n \in \IZ$, let $T_n$ denote the closure of the connected
component of $T\setminus\{h(n),h(n+1)\}$ intersecting $h((n,n+1))$.
Then  $\map{\tau\evalat{T_n}}{T_n}[T_{n+1}]$ is a homeomorphism.
\end{rema}

To simplify the notation, in the rest of the paper we will identify
$h(\IR)$ with $\IR$ itself. In particular, we are implicitly extending
the usual ordering, the arithmetic and the notion of intervals from
$\IR$ to $h(\IR)$.

Observe that, in the above setting,
Definition~\ref{def:liftedspace}(i) gives $\tau(x) = x+1$ for all
$x\in\IR$. Taking this and Remark~\ref{rem:tauplusone} into account,
it is natural to visualise the homeomorphism $\tau$ as a ``translation
by 1'' in the whole space $T$ (despite of the fact that such an
arithmetic operation need not be defined). Thus, in what follows, to
simplify the formulae we will abuse notation and write $x+1$ to denote
$\tau(x)$ for all $x \in T$. Then the fact that $T$ is homeomorphic to
itself by $\tau$ can be rewritten in this notation as: $T + 1 = T$.
Note also that, since $\tau$ is a homeomorphism, this notation can be
extended by denoting $\tau^m(x)$ by $x+m$ for all $m \in \IZ$. In what
follows, if $A\subset T$ is a set and $m \in \IZ$ then $A + m$ will
denote $\set{x+m}{x\in A}$.

\begin{exem}
To better understand the simplifications introduced above consider the
following paradigmatic particular case (see Figure~\ref{fig:inftrees}
for an example): The lifted space $T$ is embedded in $\IR^n$ and the
map $\tau(\overrightarrow{x})$ is defined as $\overrightarrow{x} +
\overrightarrow{e_1}$, where $\overrightarrow{e_1} = (1,0,\dots,0)$
denotes the first vector in the canonical base. Then $T$ must contain
the line $t\overrightarrow{e_1}$ for $t \in \IR,$ and the map $h$ from
Definition~\ref{def:liftedspace} is defined by $h(t) =
t\overrightarrow{e_1}$.
\end{exem}

\begin{figure}[htb]
\centerline{\includegraphics{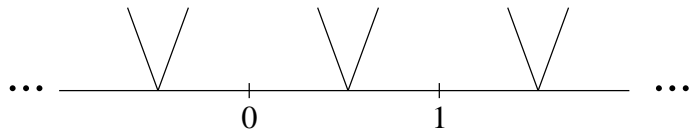}}
\caption{An example of a lifted tree that can be embedded in
$\IR^2$.\label{fig:inftrees}}
\end{figure}

Next we introduce a tool that will play a crucial role in the rest of
the paper. It is the retraction from $T$ to $\IR$. It will be used as
a measuring tool of displacements to the left or to the right and also
to identify the place where the image of a point lies in $T \setminus
\IR$.

\begin{defi}
Given $T\in \InfX$ there is a natural retraction from $T$ to $\IR$
that in the rest of the paper will be denoted by $r$. When $x\in\IR$,
then clearly $r(x) = x$. When $x \notin \IR$, by definition, there
exists a connected component $C$ of $T \setminus \IR$ such that $x\in
C$ and $\Clos{C}$ intersects $\IR$ at a single point $z$. Then $r(x)$
is defined to be, precisely, the point $z$. In particular, $r$ is
constant on $\Clos{C}$.

A point $x \in \IR$ such that $r^{-1}(x) \neq \{x\}$ will be called a
\emph{branching point of $T$}. The set of all branching points of $T$
will be denoted by $\B(T)$. It is a subset of $\IR$ by definition.
\end{defi}

The next lemma recalls the basic properties of the natural retraction.
Its proof is a simple exercise and thus it will be omitted.

\begin{lemm}\label{lem:const-r}
For each $T\in\InfX$ the following statements hold:
\begin{enumerate}
\item If $x\notin \IR$, then  there exists a neighbourhood $U$ of $x$
such that $r$ is constant in $U$.
\item The map $\map{r}{T}[\IR]$ is continuous and verifies $r(x+1) =
r(x)+1$ for all $x\in T$.
\end{enumerate}
\end{lemm}

\subsection{Maps and orbits on lifted spaces}\label{subsec:MaO}
The aim of this subsection is to study which is the object that
corresponds to orbits at the level of lifted spaces. We start by
generalising the notion of lifting and degree to this setting.

Suppose that $X$ is a metric space with a loop $S$ and that the
unwinding of $S$ gives a lifted space $T \in \InfX$. Then, there
exists a continuous map $\map{\pi}{T}[X]$, called the \emph{standard
projection from $T$ to $X$}, such that $\pi([0,1]) = S$ and $\pi(x+1)
= \pi(x)$ for all $x \in T$.

Let $\map{f}{X}$ be continuous. By using standard techniques (see for
instance \cite{Wa}) it is possible to construct a (non-unique)
continuous map $\map{F}{T}$ such that $f\circ\pi=\pi\circ F$. Each of
these maps will be called a \emph{lifting of $f$}.

Observe that $f\circ\pi=\pi\circ F$ implies that $F(1) - F(0) \in \IZ$
and, as the next lemma states, this number is independent of the
choice of the lifting. It is called indistinctly the \emph{degree of
$f$} or the \emph{degree of $F$} and denoted by $\deg(f)$ and
$\deg(F)$.

The next lemma, whose proof is straightforward (see for instance
\cite[Section~3.1]{ALM}), summarises the basic properties of lifting
maps.

\begin{lemm}\label{lem:liftings-general}
Let $\map{f}{X}$ be continuous. If the continuous map $\map{F}{T}$ is
a lifting of $f$ then $F(x+1) = F(x) + \deg(f)$  for every $x \in T$.
On the other hand, if $\map{F'}{T}$ is continuous, then $F'$ is a
lifting of $f$ if and only if $F = F' + k$ for some $k \in \IZ$.
Moreover, the following statements hold for all $x\in\IR,$ $k\in\IZ$
and $n \ge 0$:
\begin{enumerate}
\item $F^n(x+k) = F^n(x) + k\deg(f)^n$, and
\item $(F+k)^n(x) = F^n(x) + k(1+d+d^2+\dots+d^{n-1})$, with $d =
\deg(f)$.
\end{enumerate}
If $g$ is another continuous map from $X$ into itself, then $\deg(g
\circ f) = \deg(g)\cdot \deg(f)$.
\end{lemm}

Next, as we have said, we want to describe how periodic points and
periodic orbits of $f$ are seen at the lifting level.

Let $F$ be any lifting of $f$. A point $x \in T$ is called
\emph{periodic \modi} if there exists $n \in \IN$ such that $F^n(x)
\in x+\IZ$. The \emph{period \modi} of $x$ is the least positive
integer $n$ satisfying this property; that is, $F^n(x) \in x+\IZ$ and
$F^i(x) \notin x+\IZ$ for all $1\leq i\leq n-1$. Observe that $x$ is
periodic {\modi} for $F$ if and only if $\pi(x)$ is periodic for $f$.
Moreover, the $F$-period {\modi} of $x$ and the $f$-period of $\pi(x)$
coincide.

In a similar way, the set
\[
\set{F^{n}(x)+m}{ n \ge 0 \text{ and } m\in \IZ},
\]
will be called the \emph{orbit {\modi} of $x$}, and denoted by
$\Orb_1(x,F)$. Clearly,
\[
\Orb_1(x,F) = \pi^{-1}(\set{f^{n}(\pi(x))}{n \ge 0})
            = \pi^{-1}(\Orb(\pi(x),f)).
\]

When $x$ is periodic {\modi} then the orbit {\modi} of $x$,
$\Orb_1(x,F)$, is also called \emph{periodic \modi}. In this case it
is not difficult to see that $\Card\left(\Orb_1(x,F) \cap
r^{-1}([n,n+1))\right)$ coincides with the $f$-period of $x$ for all
$n \in \IZ$.

A standard approach to study the periodic points and orbits of $f$ is
to work at the lifting level with the periodic {\modi} points and
orbits instead of the original map and space. This is the
approach we will follow in this paper. The results on $F$ can
obviously be pulled back to $f$ and $X$.

As it has been said in the introduction, the aim of this paper is to
develop the rotation theory for liftings in lifted spaces and study
the relation between rotation numbers and periodic {\modi} orbits. As
it is usual, this theory can only be developed for maps of degree one,
that is, for maps verifying $F(x+1)=F(x)+1$ for all $x \in T$. So, in
the rest of the paper, we will only consider the class $\Li$ of all
continuous maps of degree 1 from $T \in \InfX$ into itself.

The following lemma is a specialisation of
Lemma~\ref{lem:liftings-general} to maps of $\Li$. Its last statement
follows from the previous one and Lemma~\ref{lem:const-r}(b).

\begin{lemm}\label{lem:degree1-Fn}
The following statements hold for $T\in\InfX$, $F\in\Li$, $n\in\IN$,
$k\in\IZ$ and $x\in T$:
\begin{enumerate}
\item $F^n(x+k)=F^n(x)+k$,
\item $(F+k)^n(x)=F^n(x)+kn$.
\item If $G$ is another map from $\Li$ then $F \circ G \in \Li$.
In particular $F^n \in \Li$.
\item The map $\map{r \circ F^n}{T}[\IR]$ is continuous and verifies
\[
r(F^n(x+1)) = r(F^n(x)) + 1
\]
for all $x \in T$.
\end{enumerate}
\end{lemm}

\subsection{Maps of degree \boldmath $1$ and rotation
numbers}\label{subsec:Fn}
The aim of this subsection is to introduce the notion of rotation
number for our setting and to study its basic properties. We define
three types of rotation numbers.

\begin{defi}
Let $T\in\InfX$, $F\in \Li$ and $x\in T$. We set
\[
\urhos(x) :=  \liminf_{n\to+\infty} \frac{r\circ F^n (x)-r(x)}{n}
              \ \text{and}\
\orhos(x) := \limsup_{n\to+\infty} \frac{r\circ F^n (x)-r(x)}{n}.
\]
When $\urhos(x) = \orhos(x)$ then this number will be denoted by
$\rhos{F}(x)$ and called the \emph{rotation number of $x$}. The
numbers $\urhos(x)$ and $\orhos(x)$ are called the \emph{lower
rotation number of $x$} and \emph{upper rotation number of $x$},
respectively.
\end{defi}

\begin{rema}
If $T$ is embedded in a normed vector field (e.g. $T \subset
\IR^n$), then one can easily see that the composition with the
retraction $r$ can be removed from the above formula without any
change and the rotation numbers can be defined simply by using
\[
\frac{F^n(x)-x}{n}.
\]
The only reason to consider $r\circ F^n$ instead of $F^n$ in the
general case is to ``project'' the point $F^n(x)$ to $\IR$ where we
have arithmetic, to be able to measure the distance between $F^n(x)$
and $x$.
\end{rema}

We now give some elementary properties of rotation numbers.

\begin{lemm}\label{lem:first-properties}
Let $T\in\InfX$, $F\in\Li$, $x\in T$, $k\in\IZ$ and $n\in\IN$.
\begin{enumerate}
\item $\orhos(x+k)=\orhos(x)$.
\item $\orhos[(F+k)](x)=\orhos(x)+k$.
\item $\orhos[F^n](x)=n\orhos(x)$.
\end{enumerate}
The same statements hold with $\underline{\rho}$ instead of
$\overline{\rho}$.
\end{lemm}

\begin{proof}
The Statements~(a) and (b) follow from Lemma~\ref{lem:degree1-Fn}(a)
and (b) respectively. The proof of (c) is similar to
\cite[Lemma~3.7.1(b)]{ALM}.
\end{proof}

\medskip
An important object that synthesises all the information about
rotation numbers is the \emph{rotation set} (i.e., the set of
\emph{all} rotation numbers). Since we have three types of rotation
numbers, we have three kinds of rotation sets.
\begin{defi}
For $T \in \InfX$ and $F \in \Li$ we define the following
\emph{rotation sets}:
\begin{align*}
\Rot^+(F) &= \set{\orhos(x)}{x\in T},\\
\Rot^-(F) &= \set{\urhos(x)}{x\in T},\\
\Rot(F)   &= \set{\rhos{F}(x)}{x\in T
                  \text{ and $\rhos{F}(x)$ exists}},
\end{align*}
Similarly we define $\RotR^+(F)$, $\RotR^-(F)$ and $\RotR(F)$
by replacing $x\in T$ by $x\in \IR$ in the above three definitions.
\end{defi}

The next simple example helps in better understanding the basic
features of rotation numbers and sets. In particular it will show
that the rotation set in this framework does not display the nice
properties of the rotation sets for continuous degree one circle
maps and will justify the study of the sets $\RotR^+$,
$\RotR^-$ and $\RotR$.

\begin{exem}\label{ex:R-not-connected}
Let $T$ be the lifted space shown  in
Figure~\ref{fig:R-not-connected}. This lifted space has two branches
$A,\ B$ between $0$ and $1$ outside $\IR$, joined at a common
branching point $e$. We denote by $a$ and $b$ the endpoints of $A$ and
$B$, respectively.

Observe that $T$ is uniquely arcwise connected. So, given two points
$x$ and $y$, the \emph{convex hull of $\{x,y\}$ in $T$} which is by
definition the smallest closed connected subset of $T$ containing $x
$ and $y$ coincides with the image of any injective path in $T$
joining $x$ and $y$. It will be denoted by $\chull{x,y}$.

Let $\map{F}{T}$ be the continuous map of degree $1$ defined by
\begin{enumerate}[(i)]
\item $F|_{\IR}=\Id$,
\item $F(A)=\chull{e,a-1}$ and $F|_A$ is injective,
\item $F(B)=\chull{e,b+1}$ and $F|_B$ is injective.
\end{enumerate}

\begin{figure}[htb]
\centerline{\includegraphics{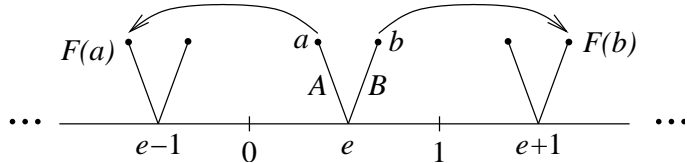}}
\caption{The set $\Rot(F)=\{-1,0,1\}$ is not connected and is not
equal to $\RotR(F)=\{0\}$.\label{fig:R-not-connected}}
\end{figure}

Obviously, $\RotR(F)=\{0\}$, $\rhos{F}(a)=-1$ and $\rhos{F}(b)=1$. Let
$x\in A$. If there exists $k\geq 1$ such that $F^k(x)\in\IR$, then
$F^n(x)=F^k(x)$ for all $n\geq k$ and $\rhos{F}(x)=0$. Otherwise,
$F^n(x)\in A-n$ for all $n\geq 1$, and $\rhos{F}(x)=-1$. Similarly, if
$x\in B$ then $\rhos{F}(x)$ equals $0$ or $1$.  Hence
$\Rot(F)=\{-1,0,1\}$, which is not a connected set. Consequently,
$\Rot(F) \neq \RotR(F)$ despite of the fact that the set
$\bigcup_{n\geq 0}F^{-n}(\IR)$ coincides with
$T\setminus(\{a,b\}+\IZ)$, which is dense in $T$.

In a similar way one can construct examples of lifted spaces and maps
$F$ such that $\Rot(F)$ has $n$ connected components for any finite,
arbitrarily large $n$, even when there is a single branch outside
$\IR$. Or connected components outside $\RotR(F)$ which are non
degenerate intervals (e.g., $F|_{\IR}=\Id$ and $F(A)\supset (A+1)\cup
(A+2)$ in the above example). Generally, when the dynamics of parts of
the branches has no relation with the dynamics of $\IR$,
disconnectedness of the rotation set is likely to occur.
\end{exem}

To study the sets $\Rot(F)$ and $\RotR(F)$ and their relation with the
periodic {\modi} points and orbits of the map $F$ we introduce the
following notation. For a continuous map $F \in \Li$ and $n \in \IN$
we set
\[
\Fncr := \map{r \circ F^n\evalat{\IR}}{\IR}.
\]
From Lemma~\ref{lem:degree1-Fn}(d) it follows that the map $\Fncr$ is
a lifting of a circle map of degree $1$, and thus the results on
rotation sets for circle maps apply to it straightforwardly.

We also generalise the notion of a \emph{twist orbit} from the context
of degree one circle maps to this setting.

\begin{defi}
Let $T \in \InfX$ and $F \in \Li$. An orbit {\modi} $P \subset \IR$ of
$F$ will be called \emph{twist} if $F\evalat{P}$ is strictly
increasing.
\end{defi}

\begin{rema}\label{rem:rotorbits}
The following statements are easy to check.
\begin{enumerate}[(i)]
\item Two points in the same orbit {\modi} have the same rotation
number.

\item If $F^q(x)=x+p$ with $q\in\IN$ and $p\in\IZ$, then
$\rhos{F}(x)=p/q$. Therefore all periodic {\modi} points have rational
rotation numbers.

\item Let $x$ be a periodic {\modi} point of period $q$ and $p\in \IZ$
such that $F^q(x)=x+p$. If $\Orb_1(x,F)$ is a twist orbit, then it
follows from \cite[Corollary 3.7.6]{ALM} that $(p,q)=1$.
\end{enumerate}
\end{rema}

The following theorem describes the relation between the
sets $\Rot(F^n)$ and $\Rot(\Fncr)$.

\begin{theo}\label{theo:Fn}
Let $F \in \Li$ and let $n\geq 1$. Assume that $x \in \IR$ is such
that $\Orb_1(x,F^n) \subset \IR$. Then, $\rhos{\Fncr}(x) =
\rhos{F^n}(x)  = n\rhos{F}(x)$. Conversely, for each $\alpha \in
\Rot(\Fncr)$ there exists $x \in \IR$ such that $\alpha=
\rhos{\Fncr}(x) = \rhos{F^n}(x)  = n\rhos{F}(x)$, $\Orb_1(x,F^n) =
\Orb_1(x,\Fncr) \subset \IR$ and $\Orb_1(x,F^n)$  is twist. Moreover,
if $\alpha \in \IQ$ then $x$ can be chosen to be periodic {\modi} of
$F.$ In particular, for each $n \in \IN$, $\tfrac{1}{n} \Rot(\Fncr)
\subset \RotR(F)$.
\end{theo}

To prove Theorem~\ref{theo:Fn} we introduce the notion of $\Const(F)$
and study its basic properties.

Given a continuous map $\map{g}{X}[Y]$, we will denote by $\Const(g)$
the set of points $x \in X$ such that $g$ is constant in a
neighbourhood of $x$. Clearly, $\Const(g)$ is open and $g|_{\Clos{C}}$
is constant for each connected component $C$ of $\Const(g)$.

\begin{lemm}\label{lem:Const}
Let $T\in\InfX$ and let  $F \in \Li$. If $x\notin \Const(r\circ F)$
then $r\circ F(x) = F(x)$. Consequently, for each $n \in \IN$, $x
\notin \Const(\Fncr)$ implies $\Fncr(x) = F^n(x)$.
\end{lemm}

\begin{proof}
Suppose that $r\circ F(x) \neq F(x)$. Then, $F(x) \notin \IR$. By
the
continuity of $F$ and Lemma~\ref{lem:const-r}(a), there exists an open
neighbourhood $U$ of $x$ in $T$ such that $r(F(U)) = r(F(x))$. This
shows that $x \in \Const(r\circ F)$, which is a contradiction. The
second statement of the lemma follows trivially from the first one.
\end{proof}

\medskip
Now we are ready to prove Theorem~\ref{theo:Fn}.

\medskip
\noindent{\it Proof of Theorem~\ref{theo:Fn}.}
The first statement of the theorem follows from
Lemma~\ref{lem:first-properties}(c).
If $\alpha \in \Rot(\Fncr)$, then by \cite[Theorem~3.7.20]{ALM} there
exists a  point $x \in \IR$ such that   $\rhos{\Fncr}(x) = \alpha$,
$\Orb_1(x,\Fncr) \subset \IR \setminus \Const(\Fncr)$ and
$\Orb_1(x,\Fncr)$ is twist. Moreover, if  $\alpha \in \IQ$ then $x$
can be chosen to be a periodic {\modi}  point of $\Fncr$. Then the
theorem follows from Lemma~\ref{lem:Const}.
{\hfill$\Box$}

\medskip
From Theorem~\ref{theo:Fn} we can derive the following consequences.

\begin{coro}\label{cory:incl}
Let $F \in \Li$ and let $n \in \IN$. Then, $\Rot(\Fcr) \subset
\tfrac1n \Rot(\Fncr)$. Moreover, for each $n \in \IN$, $\Rot(\Fncr)$
is a nonempty compact interval. Consequently, the set $\bigcup_{n\geq
1} \tfrac{1}{n} \Rot(\Fncr)$ is a nonempty interval contained in
$\RotR(F)$.
\end{coro}

\begin{proof}
For each $\alpha \in \Rot(\Fcr)$ let $x$ be the point given by the
second statement of Theorem~\ref{theo:Fn} for $\Fcr$. In particular,
$\rhos{F}(x) = \alpha$ and $\Orb_1(x,F) = \Orb_1(x,\Fcr) \subset \IR$.
Consequently, for all $n \in \IN$, $\Orb_1(x,F^n) \subset \IR$ and, by
the first statement of Theorem~\ref{theo:Fn},
\[
\rhos{\Fncr}(x) = \rhos{F^n}(x)  = n\rhos{F}(x) = n\alpha.
\]
This ends the proof of the first statement of the corollary. The fact
that, for each $n \in \IN$, $\Rot(\Fncr)$ is a nonempty compact
interval follows for instance from \cite[Theorem~3.7.20]{ALM}. Then,
the last statement of the corollary follows immediately.
\end{proof}

\begin{rema}
In general, the interval $\bigcup_{n\geq 1}\tfrac{1}{n}\Rot(\Fncr)$
need not be closed: see Example~\ref{ex:gaps-in-Per(p/q)}.
\end{rema}

\section{Positive covering}\label{sec:pos-cov}

To find periodic points in one-dimensional spaces, the notion of
\emph{covering} (introduced in \cite{Bloc3}) is often used. If $I,J$
are two compact intervals, $I$ $F$-covers $J$ if there exists a
subinterval $I_0\subset I$ such that $F(I_0)=J$. It is well known that
if $I$ $F$-covers $I$ then there exists a point $x\in I$ such that
$F(x)=x$. If $I$ $F$-covers $I$ then $F(I)\supset I$ but the latter
condition does not ensure the existence of a fixed point (see e.g.
\cite{ZMGG}).

In this section we are going to introduce a variant of the notion of
covering, that we call \emph{positive covering}. Roughly speaking, $I$
positively $F$-covers $J$ if $F(I)\supset J$ and this inclusion is
``globally increasing''. Positive covering does not imply covering but
we will see that if $I$ positively $F$-covers $I$ then $F$ has a fixed
point in $I$ (Proposition~\ref{prop:+cover-periodic}). This will be a
main tool in the rest of the paper.

\begin{defi}
Let $T\in\InfX$, let $\map{F}{T}$ be a continuous map and let $I,\ J$
be two compact subintervals of $\IR$. We say that $I$ \emph{positively
$F$-covers} $J$ and we write $I\pluscover{F} J$ if there exist $x,y
\in I$ such that $ x \le y$, $r\circ F(x) \le \min J$ and $r\circ
F(y)\geq \max J$. We remark that $I$ positively $F$-covers $J$ if and
only if $I$ positively $r \circ F$-covers $J$. So, we will
indistinctly write $I\pluscover{F} J$ or $I\pluscover{r \circ F} J$.
\end{defi}

In the next lemma we state some basic properties of positive covering.
We say that an interval $I\subset\IR$ is \emph{non degenerate} if it
is neither empty nor reduced to a point.

\begin{lemm}\label{lem:+cover-basic}
Let $T\in\InfX$, let $\map{F, G}{T}$ be two continuous maps and let
$I,\ J$ and $K$ be three compact non degenerate subintervals of $\IR$.
\begin{enumerate}
\item Suppose that $I\pluscover{F} J$. If  $K\subset J$ then
$I\pluscover{F} K$. If $K\supset I$ then $K\pluscover{F} J$.

\item If $I\pluscover{F} J$ and $a,b\in J$ with $a< b$, then there
exist $x_0,y_0 \in I$ such that $x_0\leq y_0$, $F(x_0)=a$, $F(y_0)=b$
and $r\circ F(t)\in (a,b)$ for all $t\in (x_0,y_0)$.

\item If $I\pluscover{F}J$ and $J\pluscover{G}K$ then
$I\pluscover{G\circ F} K$.

\item Suppose that $F$ is of degree $1$. If $I\pluscover{F} J$ then
$(I+n)\pluscover{F-k}(J+n-k)$ for all $n,k\in\IZ$.
\end{enumerate}
\end{lemm}

\begin{proof}
Statements (a) and  (d) follow easily from the definitions.

To prove (b), suppose that $I\pluscover{F}J$, that is, there exist
$x_1\leq y_1$ in $I$ such that $r\circ F(x_1)\leq\min J$ and $r\circ
F(y_1)\geq \max J$. Since $r\circ F$ is continuous, $r\circ
F(I)\supset J$. Let $a,b\in J$ with $a<b$ and set $x_0= \max\set{t \in
[x_1,y_1]}{r\circ F(t)=a}$. Then, $x_0<y_1$ because $a< \max J$.
Lemma~\ref{lem:const-r}(a) implies then that $F(x_0)\in\IR$, and thus
$F(x_0)=a$. Similarly, let $y_0=\min\set{t\in [x_0,y_1]}{r\circ
F(t)=b}$. The point $F(y_0)$ is in $\IR$, and thus $F(y_0)=b$. The
choice of $x_0,y_0$ implies that if $t\in (x_0,y_0)$, then $r\circ
F(t)\in (a,b)$. This proves (b).

Now we prove (c). Suppose that $I\pluscover{F}J\pluscover{G}K$. Let
$a,b \in J$ such that $a < b$, $r\circ G(a)\leq\min K$ and $r\circ
G(b)\geq \max K$. According to (b) there exist  $x_0,y_0 \in I$ such
that $x_0\leq y_0$, $F(x_0)=a$ and $F(y_0)=b$.  Then $r\circ G\circ
F(x_0)\leq\min K$ and $r\circ G\circ F(y_0)\geq\max K$; that is,
$I\pluscover{G\circ F}K$. This shows (c).
\end{proof}

\medskip
The next proposition will be a key tool to find periodic {\modi}
points.

\begin{prop}\label{prop:+cover-periodic}
Let $T\in\InfX$, $F\in \Li$ and let $I_0,\ldots, I_{k-1}$ be compact
non degenerate intervals in $\IR$ such that
\[
 I_0\pluscover{(\Fcrbas{F}{n_1})^{q_1} - p_1}
 I_1\pluscover{(\Fcrbas{F}{n_2})^ {q_2} -p_2} \cdots
 I_{k-1}\pluscover{(\Fcrbas{F}{n_k})^{q_k} - p_k}
 I_0,
\]
where the numbers $n_i$ and $q_i$ are positive integers and
$p_i\in\IZ$. For every $i \in \{1,2,\dots, k\}$ set
$m_i := \sum_{j=1}^{i} q_j n_j$
and
$\widehat{p}_i := \sum_{j=1}^{i} p_j.$
Then, there exists $x_0 \in I_0$ such that
$F^{m_k}(x_0) = x_0 + \widehat{p}_k$ and
$F^{m_i}(x_0) \in I_i + \widehat{p}_i$ for all
$i = 1,2, \dots, k-1$.
\end{prop}

To prove the above proposition we need three technical lemmas.

\begin{lemm}\label{lem:fixed-point-non-constant}
Let $a,b\in\IR$ with $a<b$ and let $\map{g}{[a,b]}[\IR]$ be a
continuous map such that $g(a)\leq a$ and $g(b)\geq b$. Then there
exists $x\in [a,b]$ such that $g(x)=x$ and $x \notin \Const(g)$.
\end{lemm}

\begin{proof}
Let $b_0=\min\set{x\in [a,b]}{g(x)=b}$. Observe that $b_0$ cannot
belong to $\Const(g)$, the map $g$ is continuous and $g(a)\leq a < b_0
\leq g(b_0)=b$. Thus there exists $x\in [a,b_0]$ such that $g(x)=x$.
Define $x_0= \max\set{x\in [a,b_0]}{g(x)=x}$. We will show by absurd
that $x_0\notin \Const(g)$.

Suppose that $x_0\in\Const(g)$ and call $J$ the connected component of
$\Const(g)$ containing $x_0$ and $a_0=\sup J$. Then, the interval $J$
is relatively open in $[a,b_0]$ and $b_0\notin\Const(g)$. This implies
that $a_0 \notin J$ and hence, $x_0<a_0$.  Since
$g(a_0)=g(x_0)=x_0<a_0$ and $g(b_0) \ge b_0$, there exists a fixed
point of $g$ in $[a_0,b_0]$ which contradicts the choice of $x_0$.
Consequently, $x_0\notin\Const(g)$.
\end{proof}

\medskip
The next lemma is easy to prove.

\begin{lemm}\label{lem:CompOfConstIsConst}
Let $F,H$ be continuous maps from $\IR$ into itself. Then,
$\Const(F) \subset \Const(H \circ F).$
\end{lemm}

\begin{lemm}\label{lem:composition-Const}
Let $T\in\InfX$, $F\in\Li$, $x_0\in\IR$ and let $n_1, n_2, \dots, n_k$
be positive integers. For every $i \in\{1,2,\dots, k\}$ set
$
G_i := \Fcrbas{F}{n_i} \circ\cdots\circ
       \Fcrbas{F}{n_2}\circ\Fcrbas{F}{n_1}
$
and $m_i := n_1 + n_2 + \dots + n_i$.
Assume that $x_0 \in \IR$ is such that $x_0 \notin \Const(G_k)$. Then,
$G_i(x_0) = F^{m_i}(x_0)$ for all $i = 1,2, \dots, k$.
\end{lemm}

\begin{proof}
To prove the lemma assume that on the contrary there exists $i \in
\{1,2, \dots, k\}$ such that $G_i(x_0) \ne F^{m_i}(x_0)$ but $G_j(x_0)
= F^{m_j}(x_0) \in \IR$ for all $j = 1,2, \dots, i-1$. To deal with
the case $i = 1$ we set $m_0 = 0$, $G_0 = \Id$ and, to simplify the
notation, $z = F^{m_{i-1}}(x_0) = G_{i-i}(x_0)$. Then we have
\[
  \Fcrbas{F}{n_i}(z) = \Fcrbas{F}{n_i} \bigl(G_{i-1}(x_0)\bigr) =
  G_i(x_0) \ne F^{m_i}(x_0) = F^{n_i}(z).
\]
Therefore, from Lemma~\ref{lem:Const}, it follows that $z \in
\Const(\Fcrbas{F}{n_i})$. Since $z = G_{i-i}(x_0)$, by the continuity
of $G_{i-i}$ it follows that $x_0 \in \Const(\Fcrbas{F}{n_i} \circ
G_{i-i}) = \Const(G_i)$. When $i < k$ we obtain that $x_0 \in
\Const(G_k)$ by Lemma~\ref{lem:CompOfConstIsConst}. Thus, in all cases
we have shown that $x_0 \in \Const(G_k)$; a contradiction.
\end{proof}

\begin{proof}[Proof of Proposition~\ref{prop:+cover-periodic}]
For every $i \in \{1,2,\dots, k\}$ set
$
G_i:=(\Fcrbas{F}{n_i})^{q_i} \circ\cdots\circ (\Fcrbas{F}{n_1})^{q_1}.
$
Then, in view of Lemma~\ref{lem:+cover-basic}(c,d), we have
$I_0 \pluscover{G_k - \widehat{p}_k} I_0$.
Moreover, applying inductively Lemma~\ref{lem:+cover-basic}(b), we get
that there exist $x,y \in I_0$ such that $x < y$,
$G_i([x,y]) \subset I_i + \widehat{p}_i$ for $i=1,2, \dots, k-1$,
$(G_k - \widehat{p}_k)(x)=\min I_0$ and
$(G_k - \widehat{p}_k)(y)=\max I_0$.
Moreover, by Lemma~\ref{lem:fixed-point-non-constant} applied to the
map
$\map{(G_k - \widehat{p}_k)\evalat{[x,y]}}{[x,y]}[\IR]$
there exists a point $x_0 \in [x,y]$ such that
$G_k(x_0) = x_0 + \widehat{p}_k$ and $x_0 \notin \Const(G_k)$.
By Lemma~\ref{lem:composition-Const}, $G_i(x_0)= F^{m_i} (x_0)$ for
all $i = 1,2, \dots, k$. Therefore, by the definition of $[x,y]$, the
point $F^{m_i}(x_0)$ belongs to $I_i + \widehat{p}_i$ for all $i =
1,2, \dots, k-1$, and
$F^{m_k}(x_0) = G_k(x_0) = x_0 + \widehat{p}_k$.
\end{proof}

\medskip
To be able to use Proposition~\ref{prop:+cover-periodic} in an easy
way we introduce the following notation. Let
\begin{align*}
\CP & \colon
I_0 \pluscover{F^{n_1}-p_1} I_1 \cdots \pluscover{F^{n_k}-p_k} I_k,
\text{ and}\\
\CP' & \colon
I_k \pluscover{F^{m_1}-q_1} J_1 \cdots \pluscover{F^{m_l}-q_l} J_l
\end{align*}
be two sequences of positive coverings. Then we will denote by
$\CP\CP'$ the concatenation of $\CP$ and $\CP'$. That is, $\CP\CP'$
denotes the sequence:
\[
 I_0 \pluscover{F^{n_1}-p_1} I_1 \cdots \pluscover{F^{n_k}-p_k}
 I_k \pluscover{F^{m_1}-q_1} J_1 \cdots \pluscover{F^{m_l}-q_l} J_l.
\]
In the particular case when $\CP$ is a \emph{loop}, that is $I_0 =
I_k$, then we will denote by $\CP^n$ the sequence $\CP$ concatenated
with itself $n-1$ times:
\[
\overbrace{\CP\cdots\CP\cdots\CP}^{n\text{ times}}.
\]

Finally, $\Fol(\CP)$ will denote the set of points that ``follow''
$\CP$. That is,
\[
\Fol(\CP) := \set{x\in I_0}{
   \bigl(F^{n_1+\ldots+n_i} - (p_1 + \ldots + p_i)\bigr)(x)
   \in I_i \text{ for all } 1\leq i\leq k
}.
\]
Clearly, $\Fol(\CP)$ is a compact set and, in view of
Proposition~\ref{prop:+cover-periodic}, it is non-empty.

\section{The rotation set}\label{sec:rotation-set}

In this section we deepen the study about the rotation set of the maps
from $F\in\Li$, with $T\in\InfX$. It is divided into two subsections.
In the first one we study the connectedness and compactness of the
rotation set together with its relation with periodic {\modi} orbits.
In Subsection~\ref{sec:periods-X} we describe the information on the
periodic {\modi} orbits of the map which is carried out by the
rotation set.

\subsection{On the connectedness and compactness of the rotation
set}\label{subsec:concompRS}

The rotation set $\Rot(F)$ may not be connected (see
Example~\ref{ex:R-not-connected}) and in general we do not know
whether it is closed. However, the main result of this subsection
(Theorem~\ref{theo:RR-closed-interval}) shows that the set of rotation
numbers of points $x \in \IR$ is a non empty compact interval which
coincides with $\RotR^+(F)$ and $\RotR^-(F)$. Its proof is inspired by
\cite[Lemma 3]{Ito}.

\begin{theo}\label{theo:RR-closed-interval}
Let $T\in\InfX$ and $F\in\Li$. Then $\RotR(F)$ is a non empty compact
interval and $\RotR(F)=\RotR^+(F)=\RotR^-(F)=\Clos{\bigcup_{ n\geq 1}
\tfrac{1}{n}\Rot(\Fncr)}$. Moreover, if $\alpha\in \RotR(F)$, then
there exists a point $x\in\IR$ such that $\rhos{F}(x)=\alpha$ and
$F^n(x)\in\IR$ for infinitely many $n$. If $p/q\in\Int(\RotR(F))$,
then there exists a periodic {\modi} point $x\in\IR$ with
$\rhos{F}(x)=p/q$.
\end{theo}

To prove Theorem~\ref{theo:RR-closed-interval} we will use the
next lemma which is not difficult to prove.

\begin{lemm}\label{lem:escape}
Let $T\in\InfX$, $F\in\Li$, $x\in\IR$ and $A$ a constant. If
$\orhos(x)<\alpha$, then there exists a positive integer $N$ such
that, for all $n\geq N$, $r\circ F^n(x)\leq x+n\alpha-A$. If
$\orhos(x)>\alpha$, then there exists an increasing sequence of
positive integers $\{n_k\}_{k\geq 0}$ such that, for all $k\geq 0$,
$r\circ F^{n_k}(x)\geq x+n_k\alpha+A$.

Similar statements with the inequalities reversed hold for
$\urhos(x)$.
\end{lemm}

\begin{proof}[Proof of Theorem~\ref{theo:RR-closed-interval}]
We are going to show that $\RotR^+(F)$ is a non empty compact interval
equal to $\RotR(F)$, the case with $\RotR^-(F)$ being similar.

By definition, $\RotR^+(F) \supset\RotR(F)$ and by
Corollary~\ref{cory:incl}  $\RotR(F)$ contains the non empty interval
$\bigcup_{n\geq 1}\tfrac{1}{n} \Rot(\Fncr)$. If $\RotR^+(F)$ is
reduced to a single point, then
\[
 \RotR^+(F) = \RotR(F) = \bigcup_{n\geq1} \frac{1}{n}\Rot(\Fncr)
  = \Clos{\bigcup_{n\geq1} \frac{1}{n}\Rot(\Fncr)}.
\]
Moreover, again by Corollary~\ref{cory:incl}, $\Rot(\Fcr) = \RotR(F)$.
So, the theorem follows in this case by Theorem~\ref{theo:Fn}.

In the rest of the proof we assume that $\RotR^+(F)$ contains at
least two points. This set is bounded by $\max\set{|r\circ
F(x)-r(x)|}{x\in T}$ and hence there exist $a=\inf \RotR^+(F)$ and
$b=\sup\RotR^+(F)$. Fix $\alpha\in [a,b]$. Since $a<b$, there exist
sequences of integers $p_n\in\IZ$, $q_n\in\IN$, such that, for every
$n \in \IN$, we have $\tfrac{p_n}{q_n}\in (a,b)$,
\[
 \left|\frac{p_{n}}{q_{n}}-\alpha\right| \leq
 \left|\frac{p_1}{q_1}-\alpha\right|
   \qquad\text{and}\qquad
 \lim_{n\to+\infty}\frac{p_n}{q_n}=\alpha.
\]
By the choice of $a$, $b$ and $\tfrac{p_n}{q_n}$, for all $n\geq 1$,
there exist $x_n,y_n\in\IR$ such that $\orhos(x_n) < \tfrac{p_n}{q_n}$
and $\orhos(y_n)>\tfrac{p_n}{q_n}$. Moreover, by
Lemma~\ref{lem:first-properties}(a), the points $x_n$ and $y_n$ can be
chosen so that $x_n \in [0,1]$ and $y_n\in [x_n,x_n+1]$. Set
$I=[0,2]$. By the choice of $x_n$ and $y_n$ we have $[x_n,y_n]\subset
I \subset [x_n-1,y_n+2]$.

Applying Lemma~\ref{lem:escape} to $F^{q_n}$  we see that there exist
two positive integers $N$ and $k_n > N$ such that
$r\circ F^{kq_n}(x_n) \leq x_n + kp_n - 1$ for all $k\geq N$, and
$r\circ F^{k_nq_n}(y_n)\geq y_n + k_np_n + 2$. Then
$[x_n,y_n] \pluscover{F^{k_nq_n} - k_np_n}[x_n-1,y_n+2]$
and, hence, $I \pluscover{F^{k_nq_n} - k_np_n} I$
by Lemma~\ref{lem:+cover-basic}(a).

Let $\{i_n\}_{n\geq 1}$ be a sequence of positive integers that will
be specified later and let $\CP_n := I \pluscover{F^{k_nq_n} - k_np_n}
I$. We set
\[
X_n = \Fol\left( (\CP_1)^{i_1}(\CP_2)^{i_2}\cdots(\CP_n)^{i_n}\right)
  \qquad\text{and}\qquad
X = \bigcap_{n\geq 1} X_n.
\]
As it has been noticed before, $X_n$ is a non empty compact set and,
clearly, $X_{n+1}\subset X_n$. Hence $X$ is not empty. Moreover, if
$x\in X$, then $F^n(x)\in\IR$ for infinitely many $n$.

We will show that if the sequence $\{i_n\}_{n\geq 1}$ increases
sufficiently fast then, $\rhos{F}(x)=\alpha$ for all $x\in X$. To do
it write $N_n=i_nk_nq_n$. Now we set $i_1=1$ and, if $i_1,\ldots,
i_{n-1}$ are already fixed, we choose $i_n$ such that
\begin{enumerate}[(i)]
\item $\tfrac{N_1+\cdots+N_{n-1}}{i_nk_nq_n}\leq\tfrac{1}{n}$,
\item $\tfrac{k_{n+1}q_{n+1}}{i_nk_nq_n}\leq\tfrac{1}{n}$.
\end{enumerate}
For any $k\in \IN$ there exists an integer $n$ such that
\[
 N_1+\cdots+N_{n-1}\leq k<N_1+\cdots+N_{n-1}+N_n.
\]
Therefore, there exist $0\leq i<i_n$ and $0\leq s<k_nq_n$ so that $k$
can be written as $k = \widetilde{N} + s$ where for simplicity we have
set
$\widetilde{N} := N_1+\cdots+N_{n-1}+ik_nq_n$.
On the other hand, recall that the map $y \mapsto r\circ F(y)-r(y)$ is
$1$-periodic on $T$. Thus, $L=\max\set{|r\circ F(z)-r(z)|}{z\in T}$
exists. Consequently, for $x \in X$ and $k$ large enough we have,
\[
|r\circ F^k(x)-r\circ F^{\widetilde{N}}(x)| \leq sL.
\]
Thus,
\begin{equation}\label{eq:Fk(x)-x}
\left|\frac{r\circ F^k(x)-x-k\alpha}{k}\right| \leq
 \frac{s}{k}L + \frac{s}{k}|\alpha| +
 \left|
  \frac{F^{\widetilde{N}}(x) - x - \widetilde{N}\alpha}{k}
 \right|.
\end{equation}

Since $x \in X$ we have that $x \in I$, and $F^{\widetilde{N}}(x) =
z+m$ with $z \in I$ and
\[
 m = \sum_{j=1}^{n-1}i_jk_jp_j+ik_np_n
   = \sum_{j=1}^{n-1}N_j\frac{p_j}{q_j}+ik_nq_n\frac{p_n}{q_n} .
\]
Therefore, since $I$ has length 2,
\begin{eqnarray*}
\lefteqn{
  \left|F^{\widetilde{N}}(x) - x - \widetilde{N}\alpha\right|
\leq \left|z-x|+|m-\widetilde{N}\alpha\right|}\\
 & \leq & 2 + \sum_{j=1}^{n-1} N_j\left|\frac{p_j}{q_j} -\alpha\right|
            + ik_nq_n \left|\frac{p_n}{q_n}-\alpha\right|\\
 & \leq & 2 + \sum_{j=1}^{n-2} N_j\left|\frac{p_1}{q_1}-\alpha\right|
            + N_{n-1}\left|\frac{p_{n-1}}{q_{n-1}}-\alpha\right|
            + ik_nq_n\left|\frac{p_n}{q_n}-\alpha\right|
\end{eqnarray*}
(where in the last inequality we have used that
$\left|\tfrac{p_j}{q_j}-\alpha\right| \leq
\left|\tfrac{p_1}{q_1}-\alpha\right|$ for all $j$).

Now, observe that
\begin{itemize}
\item from Condition~(i) we see that,
\[
 \frac{1}{k} \sum_{j=1}^{n-2} N_j \leq\frac{1}{N_{n-1}}
\sum_{j=1}^{n-2} N_j \leq \frac{1}{n-1},
\]

\item Condition~(ii) gives $\tfrac{s}{k} < \tfrac{q_nk_n}{N_{n-1}}
\leq\tfrac{1}{n-1}$, and

\item $\tfrac{N_{n-1}}{k} \le 1$ and $\tfrac{ik_nq_n}{k} \le 1$
because $k \ge \widetilde{N} \ge N_{n-1}+ik_nq_n$.
\end{itemize}
Consequently, by replacing all the above in
Equation~(\ref{eq:Fk(x)-x}), we obtain
\begin{multline*}
\left|\frac{r\circ F^k(x)-x}{k}-\alpha\right|
   <  \frac{L+|\alpha|}{n-1} + \frac{2}{k}
  + \frac{1}{n-1}\left|\frac{p_1}{q_1}-\alpha\right|\\
     + \left|\frac{p_{n-1}}{q_{n-1}}-\alpha\right|
     + \left|\frac{p_n}{q_n}-\alpha\right|.
\end{multline*}
Since $n$ goes to infinity when so does $k$ and $\lim_{n\to+\infty}
\tfrac{p_n}{q_n}=\alpha$, we get that the right hand side of the above
inequality converges to zero. Hence,
\[
 \rhos{F}(x) = \lim_{k\to+\infty}\frac{r\circ F^k(x)-x}{k}=\alpha.
\]
This proves that $\RotR^+(F) \subset [a,b] \subset \RotR(F)$; that is,
\[
\RotR(F)=\RotR^+(F)=[a,b].
\]

When $\alpha = \tfrac{p}{q}\in (a,b)$, the proof is simpler and gives
a periodic {\modi} point with rotation number $p/q$. Indeed, by taking
$p_1=p$ and $q_1=q$, the sequence $\CP_1$ gives
$I \pluscover{F^{k_1q}-kp}I$.
Thus, by Proposition~\ref{prop:+cover-periodic}, there exists a point
$x\in I$ such that $F^{k_1q}(x)=x+k_1p$. Hence $x$ is periodic {\modi}
and $\rhos{F^{k_1q}}(x) = k_1p$. By Lemma~\ref{lem:first-properties}
$\rhos{F}(x) = p/q$ and
$p/q\in \tfrac{1}{k_1q}\Rot(F_{k_1q}^r)$.
Moreover, by Theorem~\ref{theo:Fn},
$\tfrac{1}{k_1q}\Rot(F_{k_1q}^r)\subset \RotR(F)$.
Thus the density of the rational numbers in $[a,b]$ implies that
\[
\RotR(F) = \Clos{\bigcup_{n\geq 1}\tfrac{1}{n}\Rot(\Fncr)}.
\]
\end{proof}

\begin{rema}
The last statement of Theorem~\ref{theo:RR-closed-interval} is
weaker than Theorem~\ref{theo:set-of-periods}. We nevertheless state
it here because it is a byproduct of the proof.
\end{rema}

Generally $\RotR(F)$ is a proper subset of $\Rot(F)$. The next
proposition gives an immediate sufficient condition to have $\RotR(F)=
\Rot(F)$. We will see later other sufficient conditions (which include
the transitive case) when the lifted space $T$ is an infinite graph
(Theorem~\ref{theo:RR(F)=R(F)}).

\begin{prop}\label{prop:RR(F)=R(F)}
Let $T\in\InfX$ and $F\in\Li$. If $\displaystyle
\bigcup_{n\in\IZ}F^n(\IR)=T$ then
\[
 \RotR(F)=\Rot(F)= \Rot^+(F)=\Rot^-(F).
\]
\end{prop}

\begin{proof}
Let $y\in T$. If $y\in F^n(\IR)$ with $n\geq 0$, let $x\in\IR$ such
that $y=F^n(x)$. If $y\in F^{-n}(\IR)$ with $n\geq 0$, let
$x=F^n(y)\in\IR$.  In both cases, $\orhos(y)=\orhos(x)$ and
$\urhos(y)=\urhos(x)$. Thus, $\RotR^+(F) = \Rot^+(F)$, $\RotR^-(F)
= \Rot^-(F)$ and $\RotR(F) = \Rot(F)$. On the other hand, by
Theorem~\ref{theo:RR-closed-interval}, we get that
$\RotR(F) = \RotR^+(F) = \RotR^-(F)$;
which ends the proof of the proposition.
\end{proof}

\subsection{Relation between the rotation set and the set of
periods}\label{sec:periods-X}
In this subsection, we study the set of periods of periodic {\modi}
points with a given (rational) rotation number. To be more precise we
need to introduce the appropriate notation.

\begin{defi}
Let $T \in \InfX$ and $F \in \Li$.
The set of periods of all periodic {\modi} points of $F$ in $T$ will
be denoted by $\Per(F)$. Also, given $\alpha \in \IR$,
$\Per(\alpha,F)$ will denote the set of periods of all periodic
{\modi} points of $F$ in $T$ whose $F$-rotation number is $\alpha$.
Similarly, we denote by $\PerR(F)$ and $\PerR(\alpha,F)$ the same
sets as before with the additional restriction that the periodic
{\modi} points under consideration must belong to $\IR$  (we do
not require that the whole periodic {\modi} orbits belong to $\IR$).
\end{defi}

The main results of this section state that, for every
$p/q\in\Int(\RotR(F))$, the set $\Per(p/q,F)$
contains $\set{nq}{\text{for all $n\in\IN$ large enough}}$.
Moreover, if $\RotR(F)$ is not reduced to a
single point, then $\IN \setminus \Per(F)$ is finite.

The next proposition clarifies the relation between $\Per(F)$ and
$\Per(\alpha,F)$. It improves Remark~\ref{rem:rotorbits}(ii).

\begin{prop}\label{prop:Per-subset-qN}
Assume that $F\in\Li$. Then,
\[
  \Per(F) = \bigcup_{\alpha \in \Rot(F) \cap \IQ} \Per(\alpha,F).
\]
On the other hand, if $p,q$ are coprime and $p/q \in \Rot(F)$, then
$\Per(p/q,F)\subset q\IN$.
\end{prop}

\begin{proof}
The first statement of the proposition follows directly from
Remark~\ref{rem:rotorbits}(ii).

Now assume that $p \in \IZ$ and $q \in \IN$ are
coprime and let $n \in \Per(p/q,F)$. Assume that $x$ is a
periodic {\modi} point of $F$ of period $n$ such that
$\rhos{F}(x)=p/q$. There exists $k\in\IZ$ such that
$F^n(x)=x+k$. By what precedes, $\rhos{F}(x)=k/n=p/q$.
Then, since $p,q$ are coprime there exists $d \geq 1$ such that
$k=dp$ and $n=dq$. That is, $n \in q\IN$.
\end{proof}

\medskip
The next proposition gives a sufficient condition to have periodic
points of all large enough periods. It is a key tool for
Theorem~\ref{theo:set-of-periods}.

\begin{defi}\label{def:chi}
Let {\map{\chi}{\IR^+}[\IN]} be the map defined by
\[
\chi(t) = \begin{cases}
     \max\{\lceil t\rceil^2, 51\lceil t \rceil\}
          & \text{if $t > 1$,}\\
     1    & \text{when $0 \le t \le 1$,}
   \end{cases}
\]
where $\lceil \cdot \rceil$ denotes the ceiling function.
\end{defi}

\begin{prop}\label{prop:arithmetic-period}
Let $T\in\InfX$, $F\in\Li$, and let $I,\ J$ be two disjoint compact
non degenerate subintervals of $\IR$. Assume that there exists a
constant $t>0$ such that for all integers $n\geq t$ both $I$ and $J$
positively $F^n$-cover $I$ and $J$.  Then, for every positive integer
$m\geq \chi(t)$, there exists a point $x\in I$ such that $F^m(x)=x$
and $F^i(x) \neq x$ for all $1\leq i\leq m-1$.
\end{prop}

The proof of the proposition entirely relies on the following
arithmetical lemma.

\begin{lemm}\label{lem:arithmetic}
Let $N \in \IN$. Then, for every $m \geq \chi(N)$, there exist
$n_1,\ldots,n_{k_0}$ such that
\begin{enumerate}
\item $n_1+n_2+\cdots+n_{k_0}=m$,

\item $n_i\geq N$ for all $1\leq i\leq k_0$,

\item if $d$ divides $m$, $d \neq m$, then there exists $1\leq i\leq
k_0-1$ such that $d$ divides $n_1+\cdots+n_i$.
\end{enumerate}
\end{lemm}

\begin{proof}
If $N=1$, then the result is obvious by taking $k_0=m$ and $n_i=1$ for
all $1\leq i\leq m$, because $\chi(N) = 1$.

Let $m\geq N>1$. We write
\[
m=p_1^{\alpha_1}\ldots p_k^{\alpha_k}
\]
with $\alpha_i\geq 1$ and $p_1>p_2>\cdots>p_k$ the prime factors of
$m$. We define $d_i=\tfrac{m}{p_i}$ for all $1\leq i\leq k$. If $d$
divides $m$, $d \neq m$, then $d$ divides $d_i$ for some $1\leq i\leq
k$.  Consequently, it is sufficient to prove the lemma for the
divisors $d_1,\ldots, d_k$ instead of for any $d$ dividing $m$ and $d
\neq m$. The numbers $d_i$ are ordered as follows:
\[
d_1<d_2<\cdots<d_k.
\]
The idea of the proof is the following. A small $d_i$ corresponds to a
large prime factor $p_i$, and thus most of the $d_i$'s are ``large''.
It will be possible to write these large divisors as a sum
$n_1+\cdots+n_i$ with $n_j\geq N$. It will remain to deal with a small
number of small $d_i$'s. For computational reasons, we fix the
boundary between ``large'' and ``small'' $d_i$'s at
$\tfrac{\sqrt{m}}{\sqrt{N}}$.

Assume that $m\geq N^2$, which is equivalent to
$\left(\tfrac{\sqrt{m}}{\sqrt{N}}\right)^4\geq m$. This implies that
$m$ has at most three prime factors $p_i>\tfrac{\sqrt{m}}{\sqrt{N}}$,
which are $\{p_i\}_{1\leq i\leq \eps}$ for some $0\leq \eps\leq 3$
($\eps$ may be zero).

We first deal with $\{d_i\}_{\eps+1\leq i\leq k}$ (the ``large''
divisors --- note that this set is empty when $\eps=k$). For $i
\in \{\eps+1,\dots, k\},$ we have $d_i\geq \sqrt{m}\sqrt{N} \geq N$
because $p_i\leq \tfrac{\sqrt{m}}{\sqrt{N}}$. Moreover, for all
$i \in \{\eps+1,\dots, k\},$
\[
d_{i+1}-d_i=\frac{m(p_i-p_{i+1})}{p_ip_{i+1}}\geq \frac{m}{p_i^2}\geq
N.
\]
We define $n_1=d_{\eps+1}$ and $n_{i+1}=d_{\eps+i+1}-d_{\eps+i}$ for
all
$1\leq i\leq k-\eps-1$.  In this way, $n_i\geq N$ and
$n_1+\cdots+n_i=d_{\eps+i}$ for all $1\leq i\leq k-\eps$.

Now we deal with $\{d_i\}_{1\leq i\leq \eps}$ (the ``small''
divisors). For all $1\leq i\leq \eps$, we define $n_{k-\eps+i}$ such
that $d_{k+i}$ divides $n_1+\cdots+n_{k-\eps+i}$ and $N\leq
n_{k-\eps+i}\leq N+d_{k+i}$.

Finally, we define $k_0=k+1$ and $n_{k_0}=m-(n_1+\cdots+n_{k_0-1})$.
It remains to show that $n_{k_0}\geq N$ when $m$ is large enough. To
prove it, observe that $n_1+\cdots+n_{k-\eps}=\tfrac{m}{p_k}\leq
\tfrac{m}{2}$ and, for all $1\leq i\leq\eps$,
$p_i>\tfrac{\sqrt{m}}{\sqrt{N}}$. Thus $d_i<\sqrt{m}{\sqrt{N}}$. This
implies that
\[
n_{k_0}\geq \frac{m}{2}-3\sqrt{m}\sqrt{N} -3N.
\]
Suppose that $m\geq \alpha^2N$, $\alpha>0$. Then $n_{k_0}\geq
\left(\tfrac{\alpha^2}{2}-3\alpha-3\right)N$. To have $n_{k_0}\geq N$,
it  is sufficient to have $\tfrac{\alpha^2}{2}-3\alpha-3\geq 1$, that
is, $\alpha\geq 3+\sqrt{17}$. Since $(3+\sqrt{17})^2<51$, it follows
that when $N>1$ then it is sufficient to have $m$ larger than or equal
to $\max\{N^2, 51N\}$. This completes the proof of the lemma.
\end{proof}

\begin{rema}
The values of the function $\chi$ specified in
Definition~\ref{def:chi} are not optimal, but this is not important.
We only need that there exist positive integers $\chi(N)$ verifying
Lemma~\ref{lem:arithmetic}, and that $\chi(t)=1$ if $0\leq t\leq 1$.
\end{rema}

\begin{proof}[Proof of Proposition~\ref{prop:arithmetic-period}]
Take $m\geq \chi(t)$ and write $m = n_1+\cdots+n_k$ with
$n_1,\ldots,n_k$ satisfying Lemma~\ref{lem:arithmetic} for $N=\lceil
t\rceil $. We consider
\[
I\pluscover{F^{n_1}}J\pluscover{F^{n_2}}J\pluscover{F^{n_3}} \cdots
J\pluscover{F^{n_k}} I.
\]

By Proposition~\ref{prop:+cover-periodic} (with $q_i=1$ and $p_i=0$),
there exists $x$ in $I$ such that $F^m(x)=x$ and
$F^{n_1+\cdots+n_i}(x)\in J$ for all $1\leq i\leq k-1$. We have to
prove that $F^i(x) \neq x$ for all $1\leq i\leq m-1$. Let $d$ be the
minimal positive integer such that $F^d(x)=x$. Clearly, $d$ divides
$m$. Suppose that $d<m$. Then, in view of
Lemma~\ref{lem:arithmetic}(c) there exists $1\leq i\leq k-1$ such that
$d$ divides $n_1+\cdots+n_i$, which implies that
$F^{n_1+\cdots+n_i}(x)=x$. On the other hand,
$F^{n_1+\cdots+n_i}(x)\in J$ and $I\cap J=\emptyset$, which leads to a
contradiction. Thus, the period of $x$ is $d=m$.
\end{proof}

\medskip
In the rest of this subsection we use
Proposition~\ref{prop:arithmetic-period} to study the sets
$\PerR(p/q,F)$ and $\PerR(F)$. Obviously, these sets depend on
$\RotR(F)$ which, by Theorem~\ref{theo:RR-closed-interval} is a
non-empty compact interval of the real line. The next result is the
analogue in our setting (although it is somewhat weaker) of
\cite[Lemma~3.9.1]{ALM} that, for circle maps of degree one, says that
if $p/q\in\Int(\Rot(F))$ with $p$ and $q$ coprime, then
$\Per(p/q,F)=q\IN$.

\begin{theo}\label{theo:set-of-periods}
Let $T\in\InfX$, $F\in\Li$ and $\alpha,\beta\in\Int(\RotR(F))$,
$\alpha\leq\beta$. There exists a positive integer $N$ (depending on
$\alpha,\beta$) such that, if $\tfrac{p}{q}\in [\alpha,\beta]$ with
$p,q$ coprime, then
\[
\PerR(p/q,F)\supset \set{mq}{m\geq \chi(N/q)}.
\]
In particular, if $q\geq N$ then $\PerR(p/q,F)=q\IN$.
\end{theo}

\begin{proof}
According to Lemma~\ref{lem:escape}, there exist a positive integer
$N$ and two points $x_0,x_1\in\IR$ such that $\orhos(x_0) < \alpha$,
$\orhos(x_1) > \beta$ and, for all $n\geq N$, $r\circ F^n(x_0)\leq
x_0+n\alpha-1$ and $r\circ F^n(x_1)\geq x_1+n\beta+1$ By
Lemma~\ref{lem:degree1-Fn}(a) we may translate $x_1$ by an integer
such that $x_0 < x_1 < x_0+1$. Set $I=[x_0,x_1]$. Clearly, for every
$n \ge N$ and $j \in \{n\alpha-1, \dots , n\beta+1\} \cap \IN$, we
have
\[
I\pluscover{F^n} I+j.
\]
In particular, if $nq \geq N$ and $i \in \{nq\alpha-np-1, \dots ,
nq\beta-np+1\} \cap \IN$,
\[
I\pluscover{F^{nq}-np} I+i.
\]
Thus $I$ positively $(F^q-p)^n$-covers $I-1$, $I$ and $I+1$ (notice
that $nq\alpha-np\leq 0\leq nq\beta-np$ because $p/q\in
[\alpha,\beta]$).

Set $J=I+1$. Then $I\cap J=\emptyset$ and both $I$ and $J$ positively
$(F^q-p)^n$-cover $I$ and $J$ for all $n\geq N/q$. According to
Proposition~\ref{prop:arithmetic-period}, we get that, for all $m\geq
\chi(N/q)$, there exists a periodic point $x$ of period $m$ for the
map  $F^q-p$. Hence $F^{qm}(x)=x+mp$ and $\rhos{F}(x)=\tfrac{p}{q}$.
To end the proof of the first statement of the theorem we have to show
that $F^{i}(x)-x \notin \IZ$ for $i=1,2,\dots,mp-1$. Assume that, on
the contrary, there exists $1 \le d = \tfrac{mq}{l}$ with $l \in \IN$,
$l > 1$ such that $F^d(x) = x + a$ for some $a \in \IZ$. Then, in view
of Lemma~\ref{lem:degree1-Fn}(a),
\[
 x + mp = F^{mq}(x) = F^{ld}(x) = x + la = x + \frac{mq}{d} a.
\]
Consequently, $a = d\tfrac{p}{q}$ with $d\tfrac{p}{q}\in\IZ$. Thus $d$
must be a multiple of $q$ because $p,q$ are coprime. Write $d=bq$.
Then $F^{bq}(x)=x + bp$, which implies that $b=m$ which, in turn,
implies $d = mq$. In other words, $x$ is periodic {\modi} of
period $mq$ for $F$. Therefore, $\PerR(p/q,F)\supset \set{mq}{m\geq
\chi(N/q)}$.

The second statement of the theorem follows from the first one
and the fact that $\chi(t)=1$ whenever $t\leq 1$.
\end{proof}

\begin{rema}
In view of Example~\ref{ex:gaps-in-Per(p/q)}, the positive integer $N$
of Theorem~\ref{theo:set-of-periods} cannot be taken uniform for the
whole interval $\Int(\RotR(F))$.

On the other hand, Theorem~\ref{theo:set-of-periods} does not imply
that $\PerR(p/q,F)$ is equal to $\set{n\in\IN}{n\geq N}$ for some
positive integer $N$ (see Example~\ref{ex:Per0not=N}).
\end{rema}

In Corollary~\ref{cor:N-Per(F)} we deduce from
Theorem~\ref{theo:set-of-periods} that $\Per(F)$ contains all but
finitely integers, provided $\RotR(F)$ is non-degenerate. Its proof
relies on the next arithmetical lemma.

\begin{lemm}\label{lem:BA}
Let $N$ be a positive integer and $\alpha,\beta\in\IR$,
$\alpha<\beta$. There exists a positive integer $N_0$ such that, for
all $n\geq N_0$, there exists $\tfrac{p}{q}\in [\alpha,\beta]$ with
$p,q$ coprime, such that $q\geq N$ and $q$ divides $n$.
\end{lemm}

\begin{proof}
We fix a rational $\tfrac{a}{b}\in [\alpha,\beta)$ with $a,b$ coprime
and $b>0$, and $M$ a positive integer such that
$\tfrac{a}{b}+\tfrac{1}{M}\in [\alpha,\beta]$. Let $n\geq M$. There
exists $r\in\{1,\ldots,b\}$ such that $b$ divides $na+r$. Then
$\tfrac{a}{b}+\tfrac{r}{bn} = \tfrac{na+r}{bn}$ belongs to
$[\alpha,\beta]$ because $\tfrac{r}{bn} \le \tfrac{1}{M}$. Since
$(na+r)b-(bn)a = br$, B\'ezout's theorem implies that $\gcd(na+r,bn)$
divides $br\neq 0$. Thus we can write $\tfrac{na+r}{bn} =
\tfrac{p}{q}$ with $p,q$ coprime and
\[
q = \frac{bn}{\gcd(na+r, bn)} \ge \frac{bn}{br} \ge \frac{n}{b}.
\]
Moreover, $\tfrac{na+r}{bn} = \tfrac{(na+r)/b}{n}$ because $b$ divides
$na+r$, and hence $q$ divides $n$. Consequenly, the lemma holds by
taking $N_0=\max (M,bN)$.
\end{proof}

\begin{coro}\label{cor:N-Per(F)}
Let $T\in\InfX$ and $F\in\Li$. If $\RotR(F)$ is not degenerate to a
point, then the set $\IN\setminus \PerR(F)$ is finite.
\end{coro}

\begin{proof}
Let $N$ be the positive integer given by
Theorem~\ref{theo:set-of-periods} for some $\alpha,\beta\in
\Int(\RotR(F))$, $\alpha<\beta$. By Lemma~\ref{lem:BA}, there exists
an integer $N_0$ such that, for all $n\geq N_0$, there exists
$\frac{p}{q}\in [\alpha,\beta]$ with $p,q$ coprime, such that $q\geq
N$ and $q$ divides $n$. According to
Theorem~\ref{theo:set-of-periods}, $\PerR(p/q,F) = q\IN\ni n$.
Hence $\PerR(F)$ contains all integers $n\geq N_0$.
\end{proof}

\section{Combed maps}\label{sec:circular}

The aim of this section is to show that the rotation set of all maps
from a special subclass of $\Li$ (with $T\in\InfX$), called
\emph{combed maps}, has nice properties analogous to the ones
displayed by the continuous circle maps. To do this we will extend the
notions of ``lower'' and ``upper'' lifting and ``water functions'' in
the spirit of \cite[Section~3.7]{ALM} to this setting.

In the rest of this section $T$ will denote a space from $\InfX$.

\subsection{General definitions for combed maps}
We start our task with the simple observation that, for each $x,y \in
T$, the relation $r(x) \le r(y)$ defines a linear pre-ordering on $T$
which, in what follows, will be denoted by $x \preccurlyeq y$ (we
recall that a pre-ordering is a reflexive, transitive relation). We
will also use the notation $x \prec y$ to denote $r(x) < r(y)$.

\begin{defi}
A map $F \in \Li$ such that $F(x) \preccurlyeq F(y)$ whenever $x
\preccurlyeq y$ will be called \emph{non-decreasing}. Also, given $F,G
\in \Li$ we write $F \preccurlyeq G$ to denote that $F(x) \preccurlyeq
G(x)$ for each $x \in T$.
\end{defi}

\begin{rema}\label{non-decreasing}
When $F$ is non-decreasing and $r(x) = r(y)$, then it easily follows
that $r(F(x)) = r(F(y))$. Notice also that the map $r \in \Li$ is
non-decreasing.
\end{rema}

The following simple lemma follows in a similar way to
\cite[Lemma~3.7.19]{ALM} (and hence we omit its proof).

\begin{lemm}\label{rem:order-iterated}
Assume that $F,G \in \Li$, $F \preccurlyeq G$ and either $F$ or $G$ is
non-decreasing. Then $F^n \preccurlyeq G^n$ for all $n \in \IN$.
\end{lemm}

Next we define the \emph{upper} and \emph{lower} maps that, as in the
circle case, will play a key role in the study of the rotation
interval of maps from $\Li$. Given $F \in \Li$, we define
$\map{F_l,F_u}{\IR}$ by
\begin{align*}
F_u(x) & :=  \sup \set{r(F(y))}{y \preccurlyeq x} ,\\
F_l(x) & :=  \inf \set{r(F(y))}{y \succcurlyeq x}.
\end{align*}

\begin{rema}\label{rem:up-low-finint}
The following equivalent definitions for the maps $F_u$ and $F_l$
hold:
\begin{align*}
F_u(x) & =  \max \set{r(F(y))}{x-1 \preccurlyeq y \preccurlyeq x} ,\\
F_l(x) & =  \min \set{r(F(y))}{x+1 \succcurlyeq  y \succcurlyeq x}.
\end{align*}
To prove the above equalities we have to show that
\[
\sup \set{r(F(y))}{y \preccurlyeq x} = M:=
\max \set{r(F(y))}{x-1 \preccurlyeq y \preccurlyeq x}
\]
(we only prove the statement for $F_u$; the other one follows
analogously). Since the map $r \circ F$ is continuous,
\[
 \sup \set{r(F(y))}{y \preccurlyeq x} = \max\Bigl\{
    \sup \set{r(F(y))}{y \preccurlyeq x-1}, M \Bigr\}.
\]
Thus, it is enough to see that $r(F(y)) \le M$ for all $y \preccurlyeq
x-1$. If, on the contrary, there exists $z \preccurlyeq x-1$ such that
$r(F(z)) > M$, then there exists $k \in \IN$ such that $x-1
\preccurlyeq z+k \preccurlyeq x$ and, by Lemma~\ref{lem:degree1-Fn}(d)
\[
 r(F(z+k)) = r(F(z)) + k > M + k  >
     \max \set{r(F(y))}{x-1 \preccurlyeq y \preccurlyeq x};
\]
a contradiction.
\end{rema}

Now we introduce the notions of combed maps.

\begin{defi}
A map $F \in \Li$ will be called \emph{left-combed} (respectively
\emph{right-combed}) \emph{at $x \in \IR$} if $r \circ F
(\set{y\in\IR}{y \le x}) \supset r \circ F (r^{-1}(x))$ (respectively
$r \circ F (\set{y \in\IR}{y \ge x}) \supset r \circ F (r^{-1}(x))$).
If $F$ is both left-combed and right-combed at $x$ then it will be
simply called \emph{combed at $x$} (see Figure~\ref{fig:combex} for an
example). The map $F$ will be called \emph{combed} if it is combed at
every point $x\in \IR$.
\end{defi}

\begin{rema}
If $x \notin \B(T)$ (recall that $\B(T)$ denotes the set of all
branching points of $T$), then $r^{-1}(x) = \{x\}$. Therefore,  $F$ is
combed at $x$.
\end{rema}

\begin{figure}[htb]
\centerline{\includegraphics[width=25pc]{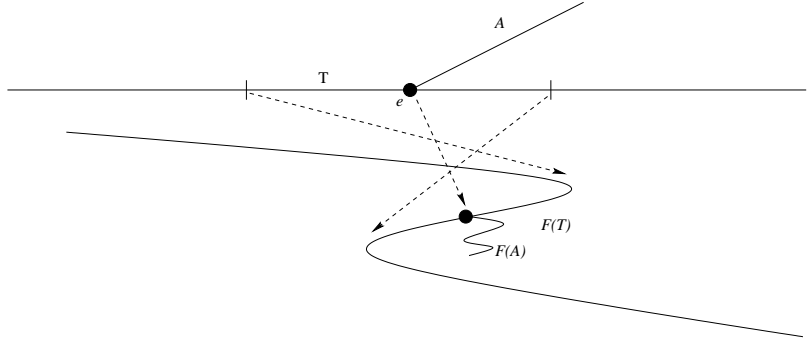}}
\caption{The image of the branch $A$ gets ``hidden'' inside $F(\IR)$
and thus $F$ is combed at $e$ (actually, $F(\IR)$ is in $T$, and the
figure shows how it folds up). An observer looking at $F(T)$ from
above or below does not distinguish this map from a ``pure circle
map''.\label{fig:combex}}
\end{figure}

\subsection{A characterisation of the upper and lower map for
combed maps}

The following technical lemma gives a nice characterisation of the
maps $F_u$ and $F_l$ for combed maps.

\begin{lemm}\label{lem:combed}
For any map $F \in \Li$ and $x \in \IR$ the following statements hold:
\begin{enumerate}
\item If $F$ is left-combed at all $y \in \IR$ such that $y \le x$,
then
\[
F_u(x) =  \sup \set{r(F(y))}{y \in \IR \text{ and } y \le x}.
\]

\item If $F$ is right-combed at all $y \in \IR$ such that $y \ge x$,
then
\[
F_l(x) =  \inf \set{r(F(y))}{y \in \IR \text{ and } y \ge x}.
\]
\end{enumerate}
\end{lemm}

\begin{proof}
We will only prove statement (a). The proof of (b) is analogous.

Clearly,
\[
 \set{y \in T}{y \preccurlyeq x} = \set{y \in \IR}{y \le x} \cup
    \left(\hspace*{-1.3em}\bigcup_{\hspace*{1.5em}\begin{subarray}{l}
      z \in \B(T) \\ z \le x
    \end{subarray}} r^{-1}(z)\right).
\]
Then, since $F$ is left-combed at all $y \in \IR$ such that $y \le
x$, we get
\[
 r \circ F (r^{-1}(z)) \subset r \circ F (\set{y\in\IR}{y \le z})
                       \subset r \circ F (\set{y\in\IR}{y \le x})
\]
for all $z \in \B(T)$, $z \le x$.
Consequently,
\[
 F_u(x) = \sup \set{r(F(y))}{y \preccurlyeq x}
        = \sup \set{r(F(y))}{y \in \IR \text{ and } y \le x}.
\]
\end{proof}

\begin{rema}
As in Remark~\ref{rem:up-low-finint} it follows that if $F$ is
left-combed at all $y \in \IR$ such that $y \le x$, then
\[
 F_u(x) =  \max \set{r(F(y))}{y \in \IR \text{ and } x-1 \le y \le x}
\]
and if $F$ is right-combed at all $y \in \IR$ such that $y \ge x$,
then
\[
 F_l(x) =  \min \set{r(F(y))}{y \in \IR \text{ and } x+1 \ge y \ge x}.
\]
\end{rema}

The next result studies the basic properties of the maps $F_l$ and
$F_u$.

\begin{lemm}\label{lem:FlFu}
For each $F \in \Li$ the maps $F_l$ and $F_u$ are non-decreasing
liftings of (non necessarily continuous) degree one circle maps that
satisfy:
\begin{enumerate}
\item $F_l(x) \preccurlyeq F(y) \preccurlyeq F_u(x)$ for each $x \in
\IR$ and $y \in r^{-1}(x)$.

\item If $G \in \Li$ verifies $F\preccurlyeq G$, then $F_l \le G_l$
and $F_u \le G_u$.

\item If $F$ is non-decreasing, then $ F_u = F_l = \Fcr = r \circ
F\evalat{\IR}$. Moreover,
\[
 \set{x \in \IR}{ r(F(x)) \ne F(x)} \subset \Const(F_u) = \Const(F_l).
\]

\item The map $F_u$ is continuous from the right whereas $F_l$ is
continuous from the left.

\item If $F$ is left-combed (respectively right-combed) at $x \in \IR$
then $F_u$ (respectively $F_l$) is continuous at $x$. In particular,
$F_u$ and $F_l$ are continuous in $\IR\setminus\B(T)$.

\item If $F_u$ (respectively $F_l$) is discontinuous at some $x \in
\IR$, then $x \in \B(T)$ and there exists $\varepsilon > 0$ such that
$[x,x+\varepsilon] \subset \Const(F_u)$ (respectively
$[x-\varepsilon,x] \subset \Const(F_l)$).
\end{enumerate}
\end{lemm}

\begin{proof}
As in the previous lemma, we will only consider the map $F_u$. The
proof for $F_l$ is analogous.

Let $x, z \in \IR$ be such that $x \le z$. We have
\[
  r\circ F (\set{y}{y \preccurlyeq x}) \subset
  r\circ F (\set{y}{y \preccurlyeq z}).
\]
So, $F_u(x) \le F_u(z)$. On the other hand, by
Lemma~\ref{lem:degree1-Fn}(d),
\begin{align*}
F_u(x+1) &= \sup \set{r(F(y))}{y \preccurlyeq x+1}
          = \sup \set{r(F(z+1))}{z \preccurlyeq x} \\
         &= \sup \set{r(F(z))+1}{z \preccurlyeq x} = F_u(x) + 1.
\end{align*}
Thus, $F_u$ is non-decreasing and has degree one.

To prove (a) observe that $F(y) \preccurlyeq F_u(x)$ is equivalent to
$r(F(y)) \le F_u(x)$ which, in turn, is equivalent to $r(F(y)) \in
\set{r(F(z))}{z \preccurlyeq x}$. On the other hand,  $y \in
r^{-1}(x)$ implies that $y \preccurlyeq x$ and this last statement
implies $r(F(y)) \in \set{r(F(z))}{z \preccurlyeq x}$. So, (a) holds.
Statements (b) and (c) follow immediately from the definitions,
Remark~\ref{non-decreasing} and Lemma~\ref{lem:Const}.

To prove (d) take $x \in \IR$ and $\delta > 0$. We have
\[
  F_u(x+\delta) =\max \left\{ F_u(x),
    \sup \set{r(F(y))}{x\preccurlyeq y \preccurlyeq x+\delta}
  \right\}.
\]
Notice that,
\begin{align*}
\lim_{\delta\searrow 0} \bigl( \sup \set{r(F(y))}{x \preccurlyeq
y \preccurlyeq x+\delta} \bigr)
 & = \sup \set{r(F(y))}{r(y) = x} \\
 & \le \sup \set{r(F(y))}{y \preccurlyeq x} = F_u(x).
\end{align*}
Consequently, $\lim_{\delta\searrow 0} F_u(x+\delta) = F_u(x)$.

To prove (e) and (f) notice that, since $r \circ F$ is continuous
and $r^{-1}(x)$ is compact,
\begin{align*}
 F_u(x)  & = \sup \set{r(F(y))}{y \preccurlyeq x} \\
         & = \max \left\{
               \sup \set{r(F(y))}{y \prec x},
               \max \Bigl\{r \circ F\bigl(r^{-1}(x)\bigr)\Bigr\}
           \right\}.
\end{align*}
Now observe that since the points from $B(T)$ are isolated, if
$\delta > 0$ is small enough then
$[x-\delta,x) \cap B(T) \neq \emptyset$, and thus
$
 \sup \set{r(F(y))}{y \preccurlyeq x-\delta}
$
varies continuously with $\delta$. Consequently,
\[
  \lim_{\delta\searrow 0} \bigl(
      \sup \set{r(F(y))}{y \preccurlyeq x-\delta}
  \bigr) = \lim_{\delta\searrow 0} F_u(x-\delta)
\]
exists and coincides with $\sup \set{r(F(y))}{y \prec x}$. In summary,
\[
F_u(x)= \max \left\{
  \lim_{\delta\searrow 0} F_u(x-\delta),
  \max \Bigl\{r \circ F\bigl(r^{-1}(x)\bigr)\Bigr\}
\right\}
\]
and hence, in view of (d), the continuity of $F_u$ at $x$ is
equivalent to
\begin{equation}\label{cont-Fu}
\max \Bigl\{r \circ F\bigl(r^{-1}(x)\bigr)\Bigr\}
  \le \lim_{\delta\searrow 0} F_u(x-\delta)
  = \sup \set{r(F(y))}{y \prec x}.
\end{equation}

Since $F$ is left-combed at $x$ we have,
\[
 r \circ F (r^{-1}(x)) \subset r \circ F (\set{y\in\IR}{y \le x})
 \subset r \circ F (\set{y}{y \prec x} \cup \{x\}),
\]
which gives \eqref{cont-Fu} by the continuity of $r \circ F$. This
ends the proof of (e).

To prove (f) assume that the map $F_u$ is discontinuous at $x \in
\IR$. Then, from \eqref{cont-Fu} it follows that $r^{-1}(x) \ne \{x\}$
and there exists $z \in r^{-1}(x) \setminus \{x\}$ such that
\[
 r(F(z)) > \sup \set{r(F(y))}{y \prec x} \ge r(F(x)).
\]
In particular, $x \in \B(T)$ and, by continuity, there exists
$\varepsilon > 0$ such that $\B(T) \cap (x,x+\varepsilon] = \emptyset$
and $r(F(y)) < r(F(z))$ for all $y \in (x,x+\varepsilon]$. For all
such points $y$ we have
\[
F_u(y) = \sup \set{r(F(y'))}{y' \preccurlyeq y}
       = \sup \set{r(F(y'))}{y' \preccurlyeq x}
       = F_u(x).
\]
This ends the proof of the lemma
\end{proof}

\begin{rema}
According to Lemma~\ref{lem:FlFu}(e), if $F$ is left-combed at
$x\in\IR$ then $F_u$ is continuous at $x$. The converse is not true.
From the proof of statements (e) and (f) of this lemma it easily
follows that if $F_u$ is continuous at some $x \in \IR$ but $F$ is
\emph{not} left-combed at $x$ (and, hence, $x \in \B(T)$), then there
exists a point $z \in \B(T)$, $z < x$ such that $F_u$ is also not
left-combed at $z$ and
\[
 \max \Bigl\{r \circ F\bigl(r^{-1}(x)\bigr)\Bigr\} \le
 \max \Bigl\{r \circ F\bigl(r^{-1}(z)\bigr)\Bigr\}.
\]
Iterating this process if necessary, one can find a point $z'\in
\B(T)$, $z'<x$, such that $F_u$ is not continuous at $z'$. Therefore,
$F_u$ is continuous if and only if $F$ is left-combed at all
$x\in\IR$.

Similar statements with reverse inequalities hold for right-combed and
$F_l$.
\end{rema}

\begin{defi}\label{def:rotFlFu}
The fact that the maps $F_l$ and $F_u$ are non-decreas\-ing implies
\cite[Theorem~1]{RT} that $\rhos{F_l}(x)$ and $\rhos{F_u}(x)$ exist
for each $x \in \IR$ and are independent of the choice of the point
$x$. These two numbers will be denoted by $\rho(F_l)$ and $\rho(F_u)$
respectively.
\end{defi}

\subsection{Rotation sets and water functions for combed maps}
The main goal of this subsection (Theorem~\ref{theo:rotset}) is to
show that, as in the case of circle maps, for combed maps the rotation
set is a closed interval of the real line. This is achieved with the
help of the so called \emph{water functions} that we extend from the
circle maps to the setting of combed maps from $\Li$.

As a consequence of Definition~\ref{def:rotFlFu} and
Lemma~\ref{lem:FlFu} one obtains:

\begin{coro}\label{lem:RUpperBound}
For each $F \in \Li$ it follows that $\rho(F_l) \le \rho(F_u)$,
$\Rot^-(F) \subset [\rho(F_l), \rho(F_u)]$, $\Rot^+(F) \subset
[\rho(F_l), \rho(F_u)]$  and, consequently, $\Rot(F) \subset
[\rho(F_l), \rho(F_u)]$.
\end{coro}

\begin{proof}
By a rewriting of Lemma~\ref{lem:FlFu}(a) we have $F_l\circ r (y)
\preccurlyeq F(y) \preccurlyeq F_u\circ r (y)$ for each $y \in T$.
From Remark~\ref{non-decreasing} and Lemma~\ref{lem:FlFu} it follows
that $F_l \circ r$ and $F_u \circ r$ are non-decreasing. Hence, since
$F_l$ and $F_u$ are self maps of $\IR$, by
Lemma~\ref{rem:order-iterated},
\[
 (F_l)^n \circ r (y) = (F_l \circ r)^n (y) \preccurlyeq F^n(y)
 \preccurlyeq (F_u \circ r)^n (y) = (F_u)^n \circ r (y)
\]
for each $n \in \IN.$ Consequently,
\[
    \frac{(F_l)^n(r(y)) - r(y)}{n}
\le \frac{r(F^n(y)) - r(y)}{n}
\le \frac{(F_u)^n(r(y)) - r(y)}{n}
\]
for each $y \in T$ and $n \in \IN.$ Then the corollary follows from
the fact that  $\rho(F_l) = \rhos{F_l}(x)$ and $\rho(F_u) =
\rhos{F_u}(x)$ for all $x \in \IR$.
\end{proof}

\medskip
In what follows we need to introduce a distance in $\Li$. We will use
the usual one, namely the $\sup$ distance, which gives the topology of
the uniform convergence. But to do this we need to specify before the
distance that we will use in $T$.

\begin{defi}\label{def:almosttaxicab}
Assume that the metric space $T$ is endowed with a $\tau$-invariant
distance $\delta_T$ (that is, for all $x,y\in T, \delta_T(x+1,y+1)=
\delta_T(x,y)$). In this paper, instead of this distance we will use
the distance $\nu$ defined as follows in the spirit of the taxicab
metric (although such a metric, in general, cannot be defined in
lifted spaces).
Given $x,y \in T$ we set\newline\hspace*{3em}
$\nu(x,y) := \delta_T(x,y)$\newline
if $x$ and $y$ lie in the same connected component
of $T\setminus \IR$, and\newline\hspace*{3em}
$\nu(x,y) := \delta_T(x,r(x)) + |r(x)-r(y)| + \delta_T(r(y),y)$
\newline
otherwise.

Note that $\nu$ coincides on $\IR$ with the natural distance. Observe
also that when $T$ is uniquely arcwise connected (in particular, when
$T$ is a lifted tree) then the distance $\nu$ gives the length of the
shortest path (in $T$) joining $x$ and $y$ and, thus, it is indeed the
taxicab metric.
\end{defi}

Now we endow the space $\Li$ with the $\sup$ distance
with respect to the distance $\nu$. Given two maps $F,G
\in \Li$, we set
\[
d(F,G) := \sup_{x \in T} \nu(F(x),G(x))
        = \sup_{x \in r^{-1}([0,1])} \nu(F(x),G(x)).
\]

Observe that the space of (not necessarily continuous) maps from $\IR$
to itself of degree one is also endowed with the sup distance:
\[
d(F,G) := \sup_{x \in \IR} |F(x) - G(x)|
        = \sup_{x \in [0,1]} |F(x) - G(x)|.
\]

\begin{lemm}\label{lem:FuFrLips}
The maps $r$, $F \mapsto r \circ F$, $F \mapsto F_l$ and $F \mapsto
F_u$ are Lipschitz continuous with constant 1.
\end{lemm}

\begin{proof}
The fact that $r$ is Lipschitz continuous with constant 1 follows
easily from the above definitions. Then, this trivially implies that
$F \mapsto r \circ F$ is Lipschitz continuous with constant 1. The
other two statements follow in a similar way to
\cite[Proposition~3.7.7(e)]{ALM}.
\end{proof}

\medskip
Now we are ready to extend to this setting the so called ``water
functions'', that play a key role in the study of the rotation
intervals of circle maps (see \cite{ALM}). Before defining these maps
we notice that, if $F\in\Li[\IR]$, then the definition of $F_u$ is
simply given by $F_u(x)= \sup \set{F(y)}{y\leq x}$. We recall that
$\Fcr$ denotes the map $\map{r \circ F\evalat{\IR}}{\IR}$. Given a map
$F \in \Li$ we define the family $\map{F_{\mu}}{\IR}$ by
\begin{equation}\label{Fmufam}
 F_{\mu} = \left( \min\{\Fcr, F_l + \mu\}\right)_u
  \quad \text{for}\quad
  0 \le \mu \le
   \mu_1 = \sup_{x \in \IR} \bigl\{\Fcr(x) - F_l(x)\bigr\}.
\end{equation}

The next lemma studies the basic properties of the family $F_\mu$. Its
proof basically follows that of \cite[Proposition~3.7.17]{ALM} by
using Lemma~\ref{lem:FlFu} in addition to
\cite[Proposition~3.7.7]{ALM}. However, in sake of completeness and
clarity, we will outline the proof.

\begin{prop}\label{prop:Fmu}
Let $F \in \Li$ be combed. Then, the maps $F_\mu$ are non-decreas\-ing
continuous liftings of degree one circle maps that satisfy:
\begin{enumerate}
\item $F_0 = F_l$ and $F_{\mu_1} = F_u$.

\item If $0 \le \lambda \le \mu \le \mu_1$, then $F_\lambda \le
F_\mu$.

\item $\Const(\Fcr) \subset \Const(F_\mu)$ for each $\mu$.

\item Each $F_\mu$ coincides with $\Fcr$ outside
$\Const(F_\mu)$.

\item The function $\mu \mapsto F_\mu$ is Lipschitz continuous with
constant 1.
\end{enumerate}
\end{prop}

\begin{proof}
To simplify the notation we denote by $G_\mu$ the map
\[
\map{\min\{\Fcr, F_l + \mu\}}{\IR}.
\]
Then, $F_\mu = (G_\mu)_u$.

Since $F$ is combed, Lemma~\ref{lem:FlFu}(e) implies that $F_l$, and
hence $G_\mu$, are continuous liftings of degree one circle maps for
each $\mu$. Then, in view of \cite[Proposition~3.7.7(d)]{ALM}, the
maps $F_\mu$ are non-decreasing continuous liftings of degree one
circle maps.

Lemma~\ref{lem:FlFu}(a) and Remark~\ref{non-decreasing} tell us that
$F_l \le \Fcr$. So, $G_0 = F_l$ and, since $F_l$ is a self-map of
$\IR$, $F_0 = (F_l)_u = F_l$ by \cite[Lemma~3.7.7(c)]{ALM}. On the
other hand, $G_{\mu_1} = \Fcr$. Consequently, for every $x \in \IR$,
\[
F_{\mu_1} (x) = (\Fcr)_u (x)
 = \sup \set{r(F(y))}{y \in \IR \text{ and } y \le x}
 = F_u(x)
\]
by Lemma~\ref{lem:combed}. This ends the proof of (a). Statement (b)
follows from \cite[Proposition~3.7.7(b)]{ALM} and the simple
observation that $G_\lambda \le G_\mu$.

Again by Lemma~\ref{lem:combed} we see that
\[
F_l(x) = \inf \set{r(F(y))}{y \in \IR \text{ and } y \ge x} = (\Fcr)_l
(x).
\]
Thus, $\Const(F_l + \mu) = \Const(F_l) \supset \Const(\Fcr)$ by
\cite[Lemma~3.7.9(b)]{ALM} and, hence, $\Const(G_\mu) \supset
\Const(\Fcr)$. By \cite[Lemma~3.7.9(a)]{ALM} we see that
\begin{equation}
 \Const(F_\mu) \supset \Const(G_\mu) \supset
 \Const(\Fcr);\label{eq:const}
\end{equation}
and (c) holds.

To prove (d) suppose that $\Fcr(x) \ne F_\mu(x) = (G_\mu)_u(x)$. If
$F_l(x) + \mu \ge \Fcr(x)$ then
\[
 G_\mu(x) = \Fcr(x) \ne (G_\mu)_u(x).
\]
So, $x \in \Const((G_\mu)_u) = \Const(F_\mu)$ by
\cite[Lemma~3.7.8(a)]{ALM}. Now suppose that $F_l(x) + \mu < \Fcr(x)$.
This implies that $(\Fcr)_l(x) = F_l(x) < \Fcr(x)$. Then,
\cite[Lemma~3.7.8(b)]{ALM} implies that $x \in \Const(F_l) =
\Const(F_l + \mu).$ Hence, there exists a neighbourhood $U \subset
\Const(F_l + \mu)$ of $x$ in $\IR$ such that $ \Fcr(y) > F_l(y) + \mu
= G_\mu(y)$ for every $y \in U$. Thus, by \eqref{eq:const}, $x \in
\Const(G_\mu) \subset \Const(F_\mu)$.

Finally, one can show that $\mu \mapsto G_\mu$ is Lipschitz continuous
with constant 1. So, (e) follows from
\cite[Proposition~3.7.7(e)]{ALM}.
\end{proof}

\medskip
The next theorem is the main result of this section. It shows that for
maps which are combed, the rotation set has properties similar to the
ones displayed by the rotation interval of continuous degree one
circle maps.

\begin{theo}\label{theo:rotset}
For each map $F \in \Li$ which is combed the following statements hold
\begin{enumerate}
\item $\Rot(F) = \Rot(\Fcr) =\RotR(F) = \Rot^+(F) = \Rot^-(F)$.
Moreover, $\Rot(F) = [\rho(F_l), \rho(F_u)]$.

\item For every $\alpha \in \Rot(F)$, there exists a twist orbit
{\modi} of $F$ contained in $\IR$, disjoint from $\Const
(F\evalat{\IR})$ and having rotation number $\alpha$.

\item For every $\alpha \in \IQ \cap \Rot(F)$, the orbit {\modi} given
by (b) can be taken periodic {\modi}.

\item The endpoints of the rotation interval, $\rho(F_l)$ and
$\rho(F_u)$ depend continuously on $F$.
\end{enumerate}
\end{theo}

\begin{proof}
It follows along the lines of the proof of \cite[Theorem~3.7.20]{ALM}
but using the previous results for combed maps and the family
$F_{\mu}$ with $0 \le \mu \le \mu_1$ defined by \eqref{Fmufam}. By
Proposition~\ref{prop:Fmu} every $F_\mu$ is a continuous
non-decreasing lifting of a degree one circle map. Hence,
\cite[Lemma~3.7.11]{ALM} implies that $\rho(F_\mu) = \rho(F_\mu(x))$
exists and is independent on $x$. Also, from
Proposition~\ref{prop:Fmu}(a,b) it follows easily that $\rho(F_l) \le
\rho(F_\mu) \le \rho(F_\lambda) \le \rho(F_u)$ whenever $0 \le \mu \le
\lambda \le \mu_1$. Notice also that the function $\mu \mapsto
\rho(F_\mu)$ is continuous and Statement~(d) holds by
Proposition~\ref{prop:Fmu}(e), Lemma~\ref{lem:FuFrLips} and
\cite[Lemma~3.7.12]{ALM}.

From Corollary~\ref{lem:RUpperBound} and Theorem~\ref{theo:Fn} we
obtain that the rotation sets $\Rot^+(F)$, $\Rot^-(F)$ and $\Rot(\Fcr)
\subset \RotR \subset \Rot(F)$ are contained in $[\rho(F_l),
\rho(F_u)]$.

From above we see that for all $\alpha \in [\rho(F_l), \rho(F_u)]$
there exists an $a \in [0, \mu_1]$ such that $\rho(F_a) = \alpha$.
Since $F_a$ is the lifting of a continuous degree one circle map, by
\cite[Lemmas~3.7.15 and 3.7.16]{ALM}, $F_a$ has an orbit {\modi} $P
\subset \IR,$ disjoint from $\Const(F_a)$ and whose $F_a$-rotation
number is $\alpha$. Moreover, if $\alpha \in \IQ$, then $P$ can be
taken periodic {\modi}. Since $F_a$ is non-decreasing, $P$ is twist.

Proposition~\ref{prop:Fmu}(c,d) tell us that $P$ is disjoint from
\[
\Const(F_a) \supset \Const(\Fcr) \supset \Const(F\evalat{\IR})
\]
and $F_a\evalat{P} = r \circ F\evalat{P}$. Then, since $P \subset
\IR$, $F_a\evalat{P} = F\evalat{P}$. Consequently, $P$ is a twist
{\modi} orbit of $F$ with $F$-rotation number $\alpha$ and, if $\alpha
\in \IQ$, then $P$ is periodic {\modi}. This ends the proof of the
theorem.
\end{proof}

\subsection{The set of periods for combed maps}
This subsection is devoted to characterising the set of periods
{\modi} for combed maps. Its main result (Theorem~\ref{theo:setper})
is the analogue of \cite[Theorem~3.9.6]{ALM} for circle maps. To state
it we need to introduce some notation.

Given two real numbers $a \le b$ we denote by $M(a,b)$ the set 
\[\set{n\in \IN}{a < k/n < b \text{ for some integer $k$}}.
\]
Clearly $M(a,b) =
\emptyset$ whenever $a = b$ and, if $a \ne b$, $M(a,b) \supset
\set{n\in\IN}{n > \tfrac{1}{b-a}}$.

\begin{theo}\label{theo:setper}
If $F \in \Li$ is combed and $\Rot(F) = [a,b]$, then the following
statements hold:
\begin{enumerate}
\item If $p,q$ are coprime and $p/q\in (a,b)$, then
$\Per(p/q,F)=q\IN$.

\item $\Per(F) = \Per(a,F) \cup M(a,b) \cup \Per(b,F)$.
\end{enumerate}
\end{theo}

\begin{proof}
If $a=b$ there is nothing to prove. So, in the rest of the proof we
assume that $a \ne b$.

Assume that $p,q$ are coprime and $a < p/q < b$, and let $n \in \IN$.
We have to show that $qn \in \Per(p/q,F)$. By
Theorem~\ref{theo:rotset}(a) we see that $p/q \in \Rot(\Fcr)$ and
observe that $\Fcr$ is a degree one circle map. To simplify the
notation, let us denote by $G$ the map $(\Fcr)^q - p$. By
\cite[Lemma~3.7.1]{ALM}, $\Rot(G) = [a-p/q, b- p/q]$ which contains
$0$ in its interior. Then, from the proof of \cite[Lemma~3.9.1]{ALM},
there exist points $t', z, t, z' \in \IR$ such that $t' < z < t < z'$,
$G(t') < t'$, $G(z) \ge (z')$, $G(t) \le t'$ and $G(z') > z'$.

Let us denote the interval $[t',z]$ by $I$ and the interval $[t,z']$
by $J$. Then
\[
 I \pluscover{G} I,J \qquad \text{and} \qquad J \pluscover{G} I,J .
\]

For $n = 1$ take the loop $I \pluscover{G} I$ of length 1 and for $n
\ge 2$ let us consider the following loop of length
$n$:
\[
 \left(I \pluscover{G} J\right) \left(J \pluscover{G} J\right)^{n-2}
 \left(J \pluscover{G} I\right) .
\]
Then, in view of Proposition~\ref{prop:+cover-periodic}, for each $n
\in \IN$, there exists $x \in I$ such that $F^{nq}(x) = x + np$ and
$F^{qi}(x) \in J + ip$ for all $i = 1,2, \dots, n-1$. By setting
$\widetilde{G} := F^q -p$ this can be rewritten as $\widetilde{G}^q(x)
= x$ and $\widetilde{G}^i(x) \in J$ for all $i = 1,2, \dots, n-1$.
Consequently, $x$ is a periodic point of $\widetilde{G}$ of period $n$
because $I \cap J \ne \emptyset$ or, in other words, $x$ is a periodic
{\modi} point of $F^q$ of period $n$ such that $\rhos{F^q}(x) = p$.
Then, from the proof of \cite[Lemma~3.9.3]{ALM} it follows that $x$ is
a periodic {\modi} point of $F$ of period $qn$ such that $\rhos{F}(x)
= p/q$. Since $\Per(p/q,F)\subset q\IN$ by
Proposition~\ref{prop:Per-subset-qN}, this ends the proof of (a).

According to Proposition~\ref{prop:Per-subset-qN},
\[
\Per(F)=\Per(a,F) \cup\Per(b,F)\cup\bigcup_{\alpha\in(a,b)\cap\IQ}
\Per(\alpha,F).
\]
On the other hand, $M(a,b)$ can be written as the union of $q\IN$ for
all pairs $p,q$ such that $a < p/q < b$ and $(p,q) = 1$. Consequently,
$M(a,b)=\bigcup_{\alpha\in(a,b)\cap\IQ} \Per(\alpha,F)$ by (a), which
proves (b).
\end{proof}

\begin{rema}
In this situation, contrary to the case of circle maps, the
characterisation of the sets $\Per(a,F)$ and $\Per(b,F)$ (where $a$
and $b$ are the endpoints of $\Rot(F)$) is not possible without
completely knowing the lifted space $T$.
\end{rema}

\section{Additional results for infinite graphs}\label{sec:graph}

This section is devoted to improving the study of the rotation set and
the set of periods {\modi} for the subclass of $\Li$ consisting of
continuous maps on infinite graph maps defined as follows.

We recall that a \emph{(topological) finite graph} is a compact
connected set $G$ containing a finite subset $V$ such that each
connected component of $G\setminus V$ is homeomorphic to an open
interval. A finite tree is a finite graph with no loops, i.e. with no
subset homeomorphic to a circle.

When we unwind a finite graph $G$ with respect to a loop, we obtain an
infinite graph $T$ that may or may not be in $\InfX$ (see
Figure~\ref{fig:hatG} for an infinite graph not in $\InfX$ and
Figure~\ref{fig:inftrees} for an infinite tree that belongs
to $\InfX$). Notice that if $G$ has exactly one loop, then $T$ is an
infinite tree and $T\in\InfX$.

\begin{defi}
Let $\InfG$ denote the subfamily of spaces $T\in\InfX$ such that
\[
 r^{-1}([0,1]) = \set{x \in T}{0 \leq r(x) \leq 1}
\]
is a finite graph. The elements of $\InfG$ will be informally called
\emph{infinite graphs}.

A point $x\in T$ is called a \emph{vertex} if there exists a
neighbourhood $U$ of $x$ such that $U\setminus\{x\}$ has at least 3
connected components. Note that all branching points of $T$ are
vertices. Also, a point $x\in T$ is called an \emph{endpoint} if
$T\setminus\{x\}$ has a unique connected component.
\end{defi}

\subsection{$\RotR(F)=\Rot(F)$ for transitive {\modi} infinite graph
maps}

A map $F\in\Li$ is said \emph{transitive {\modi}} if it is the lifting
of a transitive map, that is, for every non empty open sets $U,V$ in
$T$, there exists $n\geq 0$ such that $(F^n(U) + \IZ) \cap V \neq
\emptyset$. In other words, for every non empty open set $U \subset
T$, $\bigl(\bigcup_{n\geq 0} F^n(U)\bigr) + \IZ$ is dense in $T$. In
particular, $\bigcup_{n\geq 0} F^n(\IR)$ is dense in $T$ if $F$ is
transitive {\modi}.

Theorem~\ref{theo:RR(F)=R(F)} gives a sufficient condition, which
includes the case when $F$ is transitive {\modi}, to have
$\RotR(F)=\Rot(F)$ when $T\in\InfG$. In this situation, the study of
$\RotR(F)$ done in the rest of the paper gives indeed information on
the whole rotation set. We start with some preliminary results.

\newcommand{\TR}{\ensuremath{T\subsubR}}
In what follows we will set
\[
 \TR := \bigcup_{n\geq 0} F^n(\IR).
\]

\begin{lemm}\label{lem:T'}
Let $T\in\InfG$  and $F\in\Li$. Then
\begin{enumerate}
\item For all $n\geq 0$, $F^n(\IR)$ is a closed set.

\item For all $n\geq 0$, $F^{n+1}(\IR)\supset F^n(\IR)$ and
$\TR$ is connected. Consequently, $\Clos{\TR} \in \InfG$.

\item We have $F(\TR) = \TR$ and consequently, $F(\Clos{\TR}) =
\Clos{\TR}$.
\end{enumerate}
\end{lemm}

\begin{proof}
For $k \in \IZ$ and $G \in \Li$ set $R_k = G([k-1,k])$.
Recall that for a map $G \in \Li$, $G(\IZ) \subset \IR$. Consequently,
by the continuity of $G$ and Definition~\ref{def:liftedspace}(ii),
$R_k \cap \IR \supset [G(k-1), G(k)]$. Moreover, $R_k$ is compact and
\[
 G(\IR) = \bigcup_{k\in\IZ} R_k \supset
          \bigcup_{k\in\IZ} [G(k-1),G(k)].
\]
Since $R_{k+1} \cap R_k \supset \{G(k)\}$, the set $G(\IR)$ contains
$\IR$.

To prove that $G(\IR)$ is closed, we proceed as follows. Let
$\{x_n\}_{n\in\IN}\subset G(\IR)$ be a sequence converging to a
point $x\in T$. We will prove that $x\in G(\IR)$. The fact that it is
convergent implies that it is bounded. The sets $R_k$ are also bounded
and
$R_{k+1} = R_k+ 1$ because $G$ has degree one. This implies that
$\{x_n\}_{n\in\IN}\subset \bigcup_{k\in E} R_k$ where $E\subset \IZ$
is a finite set. Since $\bigcup_{k\in E} R_k$ is compact, we see that
$x\in \bigcup_{k\in E} R_k\subset G(\IR)$.

Now, Statement~(a) follows from above
by taking $G = F^n$. Also, by taking $G = F$ above we obtain $F(\IR)
\supset \IR$. Therefore, $F^{n+1}(\IR)\supset F^n(\IR)$ for all
$n\geq 0$. Since, $F^n(\IR)$ is connected by continuity this implies
that $\TR$ is connected. Hence $\Clos{\TR}\in\InfG$. This proves (b).

To end the proof of the lemma we only have to show that $F(\TR) =
\TR$. The inclusion $F(\TR) \subset \TR$ is obvious. Now we prove the
other inclusion. That is, for each $x \in \TR$ there exists $y \in
\TR$ such that $F(y) = x$. Since $x \in \TR$ there exists $l\ge 0$
such that $x \in F^l(\IR)$ but $x\notin F^j(\IR)$ for $j = 0,1,\dots,
l-1$. If $l > 0$ then, clearly, we can take $y \in F^{l-1}(\IR)$ and
we are done. Otherwise, $x \in \IR = \bigcup_{m\in\IZ} [F(-m),F(m)]$.
Hence, there exists $m \in \IZ$ such that $x \in [F(-m),F(m)]$. So,
$F(y) = x$ for some $y \in [-m,m]$.
\end{proof}

\begin{lemm}\label{lem:periodic-endpoints}
Let $T\in\InfG$ and $F\in\Li$ and assume that $\Clos{\TR}=T$. Then
there exists a finite set $A$ such that $T\setminus\TR = A + \IZ$,
the sets $\{A+n\}_{n\in\IZ}$ are pairwise disjoint and
every point of $T\setminus\TR$ is periodic {\modi}.
\end{lemm}

\begin{proof}
By Remark~\ref{rem:tauplusone} we may assume that $0$ is not a
branching point. Let $X =  r^{-1}([0,1]) = \set{x\in
T}{r(x)\in[0,1]}$. By definition, $X$ is a finite graph. Set $A :=
X\setminus\TR$. Since $0$ is not a branching point,
$X=r^{-1}((0,1))\cup\{0,1\}$, and thus the sets $\{A+n\}_{n\in\IZ}$
are pairwise disjoint. Clearly, $T\setminus\TR = A + \IZ$. By
Lemma~\ref{lem:T'}(b), the set $\TR \supset \IR$ is connected, and by
assumption it is dense in $T$. Thus $A$ is a finite subset of $X$.

By Lemma~\ref{lem:T'}(c), we have $F(\TR) = \TR$. This implies
$F(T\setminus\TR) = T\setminus\TR$ and, since $T\setminus\TR = A +
\IZ$ with $A$ finite it follows that for each $a\in A$ there exist
integers $n\geq 1$ and $k\in\IZ$ such that $F^n(a)=a+k$. This means
that all points in $A$ are periodic {\modi}.
\end{proof}

\begin{rema}
 While Lemma~\ref{lem:T'} holds for any lifted space except for the
statement that $\Clos{\TR} \in \InfG$,
Lemma~\ref{lem:periodic-endpoints} is only true for infinite graphs
from $\InfG$. The fact that $T\setminus\TR = A + \IZ$ \emph{being $A$
finite} is not true in general, when we remove the assumption that
$T\in\InfG$. If $T$ is an infinite tree, then $A$ is a subset of the
endpoints of $T$, but this may not be the case for any infinite graph.
\end{rema}

Now we are ready to state the main result of this subsection.

\begin{theo}\label{theo:RR(F)=R(F)}
Let $T\in\InfG$ and $F\in\Li$. If $\Clos{\TR}=T$ then
$\RotR(F)=\Rot(F)=\Rot^+(F)=\Rot^-(F)$.
\end{theo}

\begin{proof}
By Lemma~\ref{lem:periodic-endpoints}, we can write $T\setminus\TR$ as
$A + \IZ$ where $A$ is finite, the sets $\{A+n\}_{n\in\IZ}$ are
pairwise disjoint, and there exist $k \in \IN$ which is common to all
elements of $A$ and integers $\{i_a\}_{a\in A}$ such that for all
$a\in A$, $F^k(a) = a + i_a$. Hence, the rotation number of every
$a\in A$ exists and we have $\rhos{F}(a) = i_a/k$.

Clearly,
\[
\Rot(F)= \set{\rhos{F}(x)}{x\in \TR} \cup \set{\rhos{F}(x)}{x\in A},
\]
and the same holds for the upper and lower rotation numbers. If
$y\notin A$, then there exist $x\in\IR$ and $n\geq 0$ such that
$F^n(x)=y$. Thus $\urhos(x)=\urhos(y)$ and $\orhos(x)=\orhos(y)$. This
implies that $\RotR(F)=\set{\rhos{F}(x)}{x\in \TR}$, and the same
holds  for the upper and lower rotation numbers. According to
Theorem~\ref{theo:RR-closed-interval},
$\RotR(F)=\Rot^+_{\IR}(F)=\Rot^-_{\IR}(F)$. Hence
\[
\Rot(F)=\Rot^+(F)=\Rot^-(F)=\RotR(F)\cup\set{\rhos{F}(x)}{x\in A}.
\]
Therefore it remains to prove that for every $a\in A$ there exists
$x\in\IR$ such that $\rhos{F}(x)=\rhos{F}(a)$. We are going to find a
point $x\in \IR$ whose orbit is attracted by $a$.

In the rest of the proof the map $F^k$ will be denoted by $G$ so that
$G(a) = a + i_a$ for all $a \in A$ (in particular $\rhos{G}(a) =
i_a$). For each $a \in A$ choose neighbourhoods $V_a \subset W_a$ of
$a$ such that $(W_a+\IZ) \cap (W_{a'}+\IZ) = \emptyset$ whenever $a
\ne a'$ (this is possible because the sets
$\{A+n\}_{n\in\IZ}$ are pairwise disjoint)
and $G(V_a)\subset W_a+i_a$ for all $a \in A$.

Since {\TR} is an increasing union by Lemma~\ref{lem:T'}(b), we also
have
\[
\Clos{\bigcup_{n\geq 0} G^n(\IR)}=T.
\]
Thus, there exists a positive integer $N$ such that
$T\setminus G^n(\IR)\subset \bigcup_{a\in A} (V_a+\IZ)$ for all
$n \geq N$. Let $V_a^n = V_a\setminus G^n(\IR)$. Again by
Lemma~\ref{lem:T'}(b), $\{V_a^n\}_{n\geq N}$ is a decreasing sequence
of sets containing $a$ and $\bigcap_{n\geq N} V_a^n=\{a\}$ because $a$
is in
$\Clos{\bigcup_{n\geq 0}G^n(\IR)}$
but not in
$\bigcup_{n\geq 0} G^n(\IR)$.
For all $n\geq N$, we have
\begin{equation}\label{eq:Gn}
G^n(\IR) = G^{n-1}(\IR) \amalg \left(\coprod_{a\in A}
      \bigl(V_a^{n-1}\setminus V_a^n\bigr) + \IZ
  \right),
\end{equation}
where $\amalg$ denotes disjoint union. If we apply $G$ once more to
Equation~\eqref{eq:Gn} we get
\begin{equation}\label{eq:GGn}
G^{n+1}(\IR)=G^n(\IR) \cup G\left(\coprod_{a\in A}
        \bigl(V_a^{n-1}\setminus V_a^n\bigr) + \IZ
  \right).
\end{equation}
From Equation~\eqref{eq:Gn} for $n+1$ and Equation~\eqref{eq:GGn}, we
deduce:
\[
G\left(\bigcup_{a\in A}
  \bigl(V_a^{n-1}\setminus V_a^n\bigr) +\IZ
\right)\supset \coprod_{a\in A}
  \bigl(V_a^n\setminus V_a^{n+1}\bigr)+\IZ.
\]
Since $G(V_a)\subset W_a+i_a$, the images by $G$ of the sets
$\{V_a^{n-1}\setminus V_a^n\}_{a\in A}$ are pairwise disjoint and
$G(V_a^{n-1}\setminus V_a^n)$ is the only one that intersects
$W_a+\IZ$. Moreover $V_a^n\setminus V_a^{n+1}\subset W_a$. Hence
$G(V_a^{n-1}\setminus V_a^n)\supset \bigl(V_a^n\setminus
V_a^{n+1}\bigr) + i_a$ for every $a\in A$. By compactness we have
\begin{equation}\label{eq:GVa}
 G\left(\Clos{V_a^{n-1}\setminus V_a^n}\right)
\supset \Clos{V_a^n\setminus V_a^{n+1}}+i_a
\end{equation}
for all $n\geq N$ and $a\in A$.

Let $a\in A$. If $V_a^{k-1}\setminus V_a^k=\emptyset$ for some
$k\geq N$, then $V_a^{n-1}\setminus V_a^n=\emptyset$ for all $n\geq
k$, by Equation~\eqref{eq:GVa}. Thus $V_a^n=V_a^k$ for all $k=n\geq k$
and $V_a^k=\bigcap_{n\geq k}V_a^n=\{a\}$. This contradicts the fact
that $G^k(\IR)$ is closed by Lemma~\ref{lem:T'}(a). Consequently,
$V_a^{n-1} \setminus V_a^n \neq \emptyset$ for all $n\geq N$ and, by
Equation~\eqref{eq:GVa}, there exists $x\in\IR$ such that
\[
  (G-i_a)^n(x) \in \Clos{V_a^{n-1}\setminus V_a^n}
  \quad\text{for all}\quad
  n\geq N+1.
\]
This implies that $\rhos{G}(x)=i_a=\rhos{G}(a)$, that is,
$\rhos{F}(x)=\rhos{F}(a)$ in view of
Lemma~\ref{lem:first-properties}(c). This shows that
$\set{\rhos{F}(a)}{a\in A} \subset \RotR(F)$ which concludes the
proof.
\end{proof}

\begin{rema}
If $\bigcup_{n\geq 1}F^{-n}(\IR) \cup \Clos{\TR} = T$ then the
conclusion of Theorem~\ref{theo:RR(F)=R(F)} remains valid. However,
the theorem does not hold with the assumption that
$\Clos{\bigcup_{n\in\IZ} F^n(\IR)} = T$ (see
Example~\ref{ex:R-not-connected}).
\end{rema}

We deduce from Theorem~\ref{theo:RR(F)=R(F)} that for infinite graphs,
$\RotR(F)$ is the rotation set of an $F$-invariant infinite graph
contained in $T$.

\begin{coro}
Let $T\in\InfG$ and $F\in\Li$. Then $\RotR(F)=\Rot(F|_{\Clos{\TR}})$.
\end{coro}

\begin{proof}
Set $\overline{T} := \Clos{\TR}$. By Lemma~\ref{lem:T'}(c)
$\overline{T} \in \InfG$ and the map
$\map{F|_{\overline{T}}}{\overline{T}}$
belongs to $\Li[\overline{T}]$. Then, Theorem~\ref{theo:RR(F)=R(F)}
implies that the sets $\Rot(F|_{\overline{T}})$ and
$\RotR(F|_{\overline{T}})$ coincide. Also,
$\RotR(F|_{\overline{T}})=\RotR(F)$ and the corollary follows.
\end{proof}

\subsection{Periodic {\modi} points associated to the endpoints of
\boldmath $\RotR(F)$ for infinite graph maps}\label{subsec:endpoints}

In Subsection~\ref{sec:periods-X}, we dealt with the rational rotation
numbers in the interior of $\RotR(F)$. In this subsection we are going
to show that for an infinite graph map there exist periodic {\modi}
points whose rotation numbers are equal to $\min \RotR(F)$ (resp.
$\max \RotR(F)$) provided that it is a rational number. This will be
proved in the main result of this subsection
(Theorem~\ref{theo:endpoints}). However, before stating and proving
this result in detail, we will introduce the necessary machinery. It
will consist in the notion of a \emph{direct path to $+\infty$}. One
of the crucial points of the notation that we will introduce is the
construction of a direct version of a given (non direct) path going to
$+\infty$. Then we will devote to three technical lemmas to study the
properties of this kind of paths and to prepare the proof of the basic
technical result of this subsection (Lemma~\ref{lem:R>0}) that gives
sufficient conditions to assure that all points in the rotation set
are positive. This is the key tool in proving
Theorem~\ref{theo:endpoints}.

In the rest of this subsection, we fix an infinite graph $T\in\InfG$
and we let
\[
 X=\Clos{\set{x\in T}{0\leq r(x)<1}\setminus \IR}.
\]
Since $X$ is a finite union of finite graphs, we can write
$X=\bigcup_{\lambda\in\Lambda} I_{\lambda}$ where $\Lambda$ is a
finite set of indices, $I_{\lambda}$ is a set homeomorphic to a closed
non degenerate interval of the real line, $I_{\lambda}$ contains no
vertex except maybe its endpoints and the intersection of two
different sets $I_{\lambda}$, $I_{\lambda'}$ contains at most one
point. Each interval of the form $I_{\lambda}+n$ with
$\lambda\in\Lambda$ and $n\in\IZ$ will be called a \emph{basic
interval}.

We first formalise the idea that in $T$ there are only finitely many
``direct ways'' to go from a basic interval towards $+\infty$. A
\emph{direct path} is a path without loops or returns backwards. For
technical reasons, a direct path is allowed to remain constant on an
interval.

\begin{defi}
A \emph{path from $x_0\in T$ to $+\infty$} is a continuous map
$\map{\gamma}{[0,+\infty)}[T]$ such that $\gamma(0)=x_0$ and
$\lim_{t\to +\infty}r\circ \gamma(t)=+\infty$. Such a path is called
\emph{direct} if, in addition, it verifies the following condition
\[
\text{if $\gamma(t)=\gamma(t')$ for some $t \in [0,t']$, then
$\gamma\evalat{[t,t']}$ is constant.}\tag{DP}
\]
\end{defi}

\begin{rema}
If $\map{\gamma}{[0,+\infty)}[T]$ is a direct path to $+\infty$, then
there exists $t$ such that $\gamma(t')\in\IR$ for all $t'\geq t$. This
is due to the fact that if a path leaves $\IR$ at some point $z$ then
$z \in \B(T)$ and, by Definition~\ref{def:liftedspace}(ii), the path
must return to $\IR$ through the same point $z$. Then, clearly, such a
path does not verify Condition~(DP) and, hence, it is not direct.

Note also that $\gamma([0,+\infty))$ is homeomorphic to $[0,+\infty)$.
\end{rema}

In view of the previous remark, when $\gamma$ is a direct path we can
define an ordering $<_{\gamma}$ on the path $\gamma([0,+\infty))$ such
that it coincides with the order of $\IR$ on the half-line
$\gamma([0,+\infty))\cap \IR$ as follows. If
$x,x'\in \gamma([0,+\infty))$, $x \neq x'$,
then we write $x <_{\gamma} x'$ if and only if $x=\gamma(t)$ and
$x'=\gamma(t')$ with $t \in [0, t').$ The symbols
$\leq_{\gamma},\ >_{\gamma},$ and $\geq_{\gamma}$
are then defined in the obvious way.

\begin{rema}\label{rem:dirisnonicr}
If $\gamma$ is a direct path then
$\map{\gamma}{[0,+\infty)}[\gamma([0,+\infty))]$ is a non decreasing
map with respect to the ordering $\le_{\gamma}$ in the image
$\gamma([0,+\infty))$.
\end{rema}

Let $x_0,x_0'\in T$ and let $\gamma$ and $\gamma'$ be two direct paths
from $x_0,x_0'$ to $+\infty$. We say that $\gamma,\gamma'$ are
\emph{comparable} if, either $\gamma'([0,+\infty))\subset
\gamma([0,+\infty))$, or $\gamma([0,+\infty))\subset
\gamma'([0,+\infty))$. In the first situation, $x_0\leq_{\gamma}x_0'$,
that is, $x_0'$ is ``on the way'' between $x_0$ and $+\infty$. The
second situation is symmetric.

We will be interested in comparing direct paths starting in the same
basic interval.

\begin{lemm}\label{lem:nbclasses}
The relation of comparability is an equivalence relation among all
direct paths to $+\infty$ starting in the same basic interval.
Moreover, the set of equivalence classes of such paths for the
comparability relation is finite.
\end{lemm}

\begin{proof}
Let $I_{\lambda}+n$ be a
basic interval and assume that $\gamma$ is a direct path from some
$x_0\in I_{\lambda}+n$ to $+\infty$.

Set $\lambda_0(\gamma)=\lambda$, $t_0=0$ and
$t_1 = \max\set{t\geq t_0}{\gamma([t_0,t])
          \subset I_{\lambda_0(\gamma)} + n}$.
Clearly, $\gamma(t_1)$ is an endpoint of $I_{\lambda}+n$.
Now we define inductively two finite sequences
$\{\lambda_i(\gamma)\}_{0\leq i\leq k}$
and
$\{t_i\}_{0\leq i\leq k+1}$
in the following way. If
$\gamma(t_i)=r(x_0)$ then $\gamma([t_i,+\infty))\subset \IR$
and we stop the construction. Otherwise, there exists
$\lambda_{i}(\gamma) \neq \lambda_{i-1}(\gamma)$ and $\eps>0$
such that
$\gamma(t_i) \in
   (I_{\lambda_{i-1}(\gamma)} \cap I_{\lambda_i(\gamma)}) +n$
and
$\gamma([t_i,t_i+\eps]) \subset I_{\lambda_i(\gamma)}+n$.
Then we can define
$t_{i+1} = \max\set{t\geq t_i}{
   \gamma([t_i,t]) \subset I_{\lambda_i(\gamma)} + n}$.

Since $\gamma$ is a direct path, each $\lambda\in\Lambda$ appears at
most once in the sequence of $\lambda_i(\gamma)$'s. Therefore the
construction ends and the number of possible sequences
$\{\lambda_i(\gamma)\}_i$ is finite.

The set $\gamma([0,t_1])$ is a subinterval of $I_{\lambda}+n$ with
endpoints $x_0$ and $(I_{\lambda} \cap I_{\lambda_1(\gamma)})+n$.
Moreover, for every
$i=1,\dots,k$, $\gamma([t_i,t_{i+1}]) = I_{\lambda_i(\gamma)}+n$,
and
$\gamma([t_{k+1},+\infty))=[r(x_0),+\infty)$.
If $\gamma'$ is another direct path from some point in $I_{\lambda}+n$
to $+\infty$, then $\gamma'$ is comparable with $\gamma$ if and only
if
$
\{\lambda_i(\gamma)\}_{1 \leq i \leq k} =
\{\lambda_i(\gamma')\}_{1 \leq i \leq k'}.
$
Therefore, comparability is an equivalence relation among the direct
paths starting in $I_{\lambda}+n$, and the number of equivalence
classes of direct paths to $+\infty$ starting at $I_{\lambda}+n$ by
the comparability relation is finite.
\end{proof}

\medskip
In the next definition, we associate to a path $\gamma$ to $+\infty$ a
direct path $\widetilde{\gamma}$ to $+\infty$ by cutting all loops and
returns backwards of $\gamma$. In some sense, $\widetilde{\gamma}$
``globally follows'' the path $\gamma$ but goes directly towards
$+\infty$. Figure~\ref{fig:gammatilde} illustrates this definition.

\begin{figure}[htb]
\centerline{\includegraphics[width=25pc]{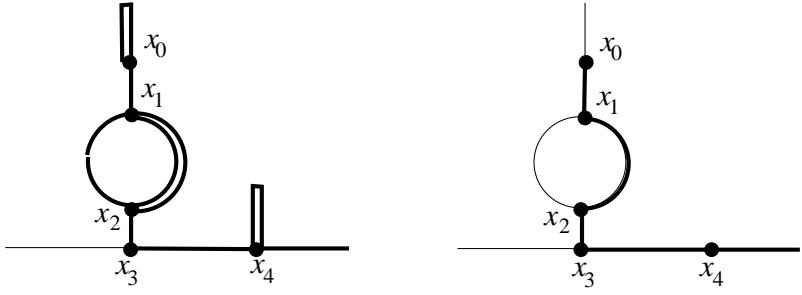}}
\caption{A path $\gamma$ on the left, and the associated direct path
$\widetilde{\gamma}$ on the right. With the notation of
Definition~\ref{def:gammatilde}, one has
$x_0=\gamma(t_0)=\gamma(t_0')=\widetilde{\gamma}\evalat{[t_0,t_0']}$,
$x_1=\gamma(t_1)=\gamma(t_1')=\widetilde{\gamma}\evalat{[t_1,t_1']}$,
$x_2=\gamma(t_2)=\widetilde{\gamma}(t_2)$,
$x_3=r(x_0)=\gamma(t^*)=\widetilde{\gamma}(t^*)$,
with
$0=t_0<t_0'<t_1<t_1'<t_2=t_2'<t_3=t^*$.\label{fig:gammatilde}}
\end{figure}

\newcommand{\ordint}[1]{\,\unlhd_{#1}}
\begin{defi}\label{def:gammatilde}
Let $x_0 \in X+n$, $n\in \IZ$ and let $\gamma$ be a path from $x_0$ to
$+\infty$. We define a path $\widetilde{\gamma}$ as follows.

First we set
\[
t^* = \min \set{t\in [0,+\infty)}{r\circ \gamma(t) = r(x_0)}
\]
and we define
\[
\widetilde{\gamma}(t) = \max r\circ \gamma([t^*,t])
\text{ for all $t\in [t^*,+\infty)$.}
\]
Observe that $\widetilde{\gamma}\evalat{[t^*,+\infty)}$ is a non
decreasing map from the interval $[t^*,+\infty)$ onto $[r(x_0),
+\infty) \subset \IR$. If $x_0 \in \IR$ then $t^* = 0$ and
$\widetilde{\gamma}$ is already defined.
Otherwise, $t^* > 0$ and $\gamma([0,t^*])$ is a path
contained in the connected component of $X+n$ containing $x_0$. Now we
inductively define $\widetilde{\gamma}\evalat{[0,t^*]}$.
\begin{itemize}
\item[\textbf{Step 0.}] Set $t_0=0$ and let $\lambda_0 \in\Lambda$ be
such that
$x_0\in I_{\lambda_0}+n$. We define
\[
 t_0' := \max\set{t\in[t_0,t^*]}{\gamma(t)=\gamma(t_0)}
\]
and
\[
\widetilde{\gamma}(t) := \gamma(t_0) = \gamma(t_0')
\text{ for all $t\in [t_0,t_0']$.}
\]
Observe that $t_0' < t^*$ (otherwise $x_0 = \gamma(t_0') = r(x_0) \in
\IR$).
Let
\[
 t_1 := \max\set{t\in[t_0',t^*]}{\gamma([t_0',t])
             \subset I_{\lambda_0}+n}.
\]
In this situation $\gamma(t_1)$ is an endpoint of
$I_{\lambda_0}+n$. Since $\gamma([t_0',t_1]) \subset I_{\lambda_0}+n$
we can define a linear ordering $\ordint{\lambda_0}$ in
$I_{\lambda_0}+n$ such that $\gamma(t_0') \ordint{\lambda_0}
\gamma(t_1)$. Now, for all $t \in [t_0',t_1]$, we define
\[
\widetilde{\gamma}(t) := \max \gamma([t_0',t]),
\]
where the maximum is taken with respect to the ordering
$\ordint{\lambda_0}$. The map
\[
\map{\widetilde{\gamma}\evalat{[t_0',t_1]}}{I_{\lambda_0}+n}
\]
is non decreasing for $\ordint{\lambda_0}$. Moreover, by the choice of
$t_1$, $\widetilde{\gamma}(t_1) = \gamma(t_1)$.

If $t_1=t^*$ then $\widetilde{\gamma}$ is already defined for all
$t\geq 0$. Otherwise we proceed to the step $k=1$.

\item[\textbf{Step k.}]
Suppose that we have already defined
$t_0\leq t_0' \le t_1 \leq t_1' < t_2\leq t_2'\leq \cdots
t_{k-1}'<t_k<t^*$ and
$\lambda_0,\lambda_1, \dots,\lambda_{k-1} \in \Lambda$
verifying the following properties:
\begin{enumerate}[(i)]
\item $t_i'=\max\set{t\in[t_i,t^*]}{\gamma(t)=\gamma(t_i)}$
for all $0\leq i\leq k-1$,
\item $\gamma(t_{i+1})$ is an endpoint of $I_{\lambda_i}+n$ for all
$0\leq i\leq k-1$,
\item $\gamma([t_i',t_{i+1}])=I_{\lambda_i}+n$ for all $1\leq i\leq
k-1$,
\item $\lambda_0,\lambda_1, \dots,\lambda_{k-1}$ are all
different.
\end{enumerate}
First, let $t_k'$ be the real number given by (i) for $k$. We define
\[
\widetilde{\gamma}(t) := \gamma(t_k) = \gamma(t_k')
\text{ for all $t\in [t_k,t_k']$.}
\]
The definition of $t_k'$ implies that there exists a unique
$\lambda_k\in\Lambda$ such that $\gamma([t_k',t_k'+\eps])\subset
I_{\lambda_k}+n$ for $\eps>0$ small enough. Since $\gamma(t_k)$ is an
endpoint of $I_{\lambda_{k-1}}+n$ by (ii), we get in addition that
$\lambda_k\not=\lambda_{k-1}$ and $\gamma(t_k')=\gamma(t_k)$ is an
endpoint of $I_{\lambda_k}+n$.
Let
\[
 t_{k+1} := \max\set{t\in[t_k',t^*]}{\gamma([t_k',t])
             \subset I_{\lambda_k}+n}.
\]
The choice of $\lambda_k$ implies that $t_{k+1}>t_k'$. This definition
implies that $\gamma(t_{k+1})$ is an endpoint of $I_{\lambda_k}+n$,
which gives (ii) for $k$. In addition, $\gamma(t_{k+1})$ is not equal
to the other endpoint $\gamma(t_k')$, and thus we get (iii) for $k$.

In this situation, we can define, as in step (0), a linear ordering
$\ordint{\lambda_k}$ in $I_{\lambda_k}+n$ such that $\gamma(t_k')
\,\lhd_{\lambda_k} \gamma(t_{k+1})$. Now, for all $t \in
[t_k',t_{k+1}]$, we define
\[
\widetilde{\gamma}(t) := \max \gamma([t_k',t]),
\]
where the maximum is taken with respect to the ordering
$\ordint{\lambda_k}$. As in Step~0,
$\widetilde{\gamma}\evalat{[t_k',t_{k+1}]}$ is non decreasing for
$\ordint{\lambda_k}$. Also, the the choice of $t_{k+1}$
implies that
$\widetilde{\gamma}(t_{k+1}) = \gamma(t_{k+1})$.

Suppose that $\lambda_k=\lambda_i$ for some $0\leq i\leq k-2$. Since
$\gamma(t_k')$ and $\gamma(t_{k+1})$ are the two endpoints of
$I_{\lambda_k}+n$, one of them is equal to
$\gamma(t_{i+1})=\gamma(t_{i+1}')$ by (i-ii). Then the definition of
$t_{i+1}'$ gives a contradiction because $t_{i+1}'<t_k'\leq t_{k+1}$.
Since we have shown above that $\lambda_k\not=\lambda_{k-1}$, this
gives (iv) for $k$ and ends the step $k$. If $t_{k+1}=t^*$ then
$\widetilde{\gamma}$ is already defined for all $t\geq 0$. Otherwise
we proceed to the step $k+1$.
\end{itemize}
According to the property (iv), this construction comes to an end
because $\Lambda$ is finite.
\end{defi}

\begin{rema}
A construction related with the one performed in
Definition~\ref{def:gammatilde} but in a topological framework can be
found in \cite{AMM}.
\end{rema}

The next lemma can be easily deduced from the above construction.

\begin{lemm}
Given a path $\gamma$ to $+\infty$, it follows that the path
$\widetilde{\gamma}$ constructed in Definition~\ref{def:gammatilde} is
a direct path to $+\infty$.
\end{lemm}

\begin{lemm}\label{lem:const(gammatilde)}
Let $\gamma$ be a path  to $+\infty$. If
$\widetilde{\gamma}(a) \neq \gamma(a)$,
then there exist $s_1,s_2$ such that $s_1 < a < s_2$ and
$\gamma(s_1)=\gamma(s_2)=\widetilde{\gamma}(a)$
for all $t\in [s_1,s_2]$.
\end{lemm}

\begin{proof}
We use the same notation as in the definition of $\widetilde{\gamma}$
above. There are three cases when $\widetilde{\gamma}(t)$ and
$\gamma(t)$ can differ:
\begin{case}{1}$t_i < a < t_i'$ for some integer $i \geq 0$.\end{case}
Then
$\widetilde{\gamma}(t) = \gamma(t_i) = \gamma(t_i')$
for all $t\in [t_i,t_i']$.
In this case we take $s_1 = t_i$ and $s_2 = t_i'$.

\begin{case}{2}$t_{i-1}' < a < t_i$ for some integer $k\geq 1$.\end{case}
There exists $z \in [t_{i-1}',a)$ such that
$\gamma(a) <_{\widetilde{\gamma}} \gamma(z) = \widetilde{\gamma}(a)$,
and thus $\widetilde{\gamma}\evalat{[z,a]}$ is constant.
Recall that, by Definition~\ref{def:gammatilde},
$\gamma(t_{i-1}') = \widetilde{\gamma}(t_{i-1}')$,
$\gamma(t_{i}) = \widetilde{\gamma}(t_{i})$ and
$\widetilde{\gamma}(t) = \max_{\ordint{\lambda_i}} \gamma([t_i',t])$
for all $t \in [t_{i-1}',t_i]$. Taking all this and the continuity of
$\gamma$ and $\widetilde{\gamma}$ into account, it follows that there
exists a maximal interval $[s_1,s_2] \subset [t_{i-1}',t_i]$
containing $a$ such that
$\widetilde{\gamma}\evalat{[s_1,s_2]}$ is constant, $\gamma(s_1) =
\widetilde{\gamma}(s_1)$ and $\gamma(s_2) = \widetilde{\gamma}(s_2)$.
Since $\gamma(a)\neq\widetilde{\gamma}(a)$, we have $s_1<a<s_2$.
This ends the proof of the lemma in this case.

\begin{case}{3}$a > t^*$.\end{case}
By Definition~\ref{def:gammatilde},
$\widetilde{\gamma}(t^*)=\gamma(t^*)\in\IR$ and
\[
\widetilde{\gamma}(t) =\max r\circ \gamma([t^*, t])
\]
for all $t\geq t^*$. Since $\lim_{t\to+\infty} r\circ \gamma(t) =
+\infty$, there exists $t'>a$ such that $r\circ \gamma(t')\geq
\widetilde{\gamma}(a)$. A similar argument as before shows that there
exist a maximal interval $[s_1,s_2] \subset [t^*,+\infty)$ such that
$s_1<a<s_2$, $\widetilde{\gamma}\evalat{[s_1,s_2]}$ is constant,
$r\circ \gamma(s_1) = \widetilde{\gamma}(s_1)$ and $r\circ \gamma(s_2)
= \widetilde{\gamma}(s_2)$. The maximality of the interval $[s_1,s_2]$
implies that $r\circ\gamma(s_i)=\gamma(s_i)\in\IR$ for $i=1,2$. This
concludes the proof of the lemma.
\end{proof}

\medskip
Suppose that $\gamma_0$ is a direct path such that
$\gamma_1=\widetilde{F\circ \gamma_0 }$ is comparable with $\gamma_0$
and $\gamma_1(\IR^+)\supset \gamma_0(\IR^+)$. If there exist $x\in
\gamma_0(\IR^+)$ and $y\in \IR$, $y> r(x)$, such that
$F(x)\leq_{\gamma_1} x$ and $y\leq r\circ F(y)$, then this looks much
like a positive covering of $[x,y]$ by itself ($[x,y]$ being seen as a
subinterval inside the half-line $\gamma_0(\IR^+)$). This situation
does indeed imply the existence of a fix point: the next lemma states
this result in the more general setting of successive iterations of
$F$.

Recall that $\Fcr$ denotes the map $r\circ F\evalat{\IR}$.

\begin{lemm}\label{lem:path-periodicity}
Let $F\in\Li$ and let $\gamma$ be a path from $x_0$ to $+\infty$.
Define $\gamma_0=\widetilde{\gamma}$ and
$\gamma_{n+1}=\widetilde{F\circ \gamma_n}$ for all $n \ge 0$.
Suppose that for some $n\geq 1$ the paths $\gamma_n$ and $\gamma_0$
are comparable and $F^n(x_0)\leq_{\gamma_n} x_0$, and suppose that
there exists $y\in\IR$ such that $(\Fcr)^j(y)\geq y$ for all $1\leq
j\leq n$. Then there exists $z$ such that  $F^n(z)=z$.
\end{lemm}

\begin{proof}
Since $F$ has degree one, by taking $y+k$ with $k \in \IZ$
sufficiently large instead of $y$, we may assume that $r\circ
F^j(x_0)< y$ for all $0\leq j\leq n$. There exists $t\in (0,+\infty)$
such that $\gamma_0(t)=y$.

We show by induction that
\begin{equation}\label{eq:*}
  \gamma_j(t) \geq (\Fcr)^j(y)\text{ for all $0\leq j\leq n$}.
\end{equation}
This is clearly true for $j=0$ by the choice of $t$. Suppose now that
$\gamma_j(t)\geq (\Fcr)^j(y)$ for some $j \in \{0,\dots, n-1\}$ and
prove it for $j+1$. By assumption, $(\Fcr)^j(y)\geq y> r\circ
F^j(x_0)$, and thus there exists $t'\leq t$ such that
$\gamma_j(t')=(\Fcr)^j(y)$. Since $\gamma_{j+1}$ is a direct path, we
have
$\gamma_{j+1}(t)\geq\gamma_{j+1}(t'):=\widetilde{F\circ\gamma_j}(t')$
by Remark~\ref{rem:dirisnonicr}. Also,
\[
r\circ F(\gamma_j(t')) = \Fcr((\Fcr)^j(y)) = (\Fcr)^{j+1}(y).
\]
Moreover, observe that $F \circ \gamma_j$ is a path starting at
$F^{j+1}(x_0)$ and, by the assumptions,
\[
r\circ F(\gamma_j(t')) = (\Fcr)^{j+1}(y) \ge y > r\circ F^{j+1}(x_0).
\]
Thus, we are in the part $[t^*,+\infty)$ of
Definition~\ref{def:gammatilde},
and hence
$\widetilde{F\circ\gamma_j}(t') \geq r\circ F(\gamma_j(t'))$.
Summarising we have shown that
$\gamma_{j+1}(t) \geq (\Fcr)^{j+1}(y)$; which ends the proof of the
induction step.

Set $I=\gamma_0([0,t])$ and $J=\gamma_n([0,t])$. Clearly both sets
are homeomorphic to closed intervals of the real line, $I$ has
endpoints $x_0$ and $y$ while the endpoints of $J$ are
$F^n(x_0)$ and $\gamma_n(t) \geq (\Fcr)^n (y)\geq y$
(Equation \eqref{eq:*} for $i=n$). By assumption,
$F^n(x_0)\leq_{\gamma_n} x_0$, and thus $I\subset J$. We define a map
{\map{G}{I}[J]} as follows. Given a point $x \in I$ take $t \in
\IR^+$ such that $\gamma_0(t) = x$ and then set $G(x) = \gamma_n(t)$.
We have to show that the map $G$ is well defined. Let $U_x$
denote $\set{t \in \IR^+}{\gamma_0(t) = x}$. If $\Card(U_x) > 1$
then, since $\gamma_0 = \widetilde{\gamma}$ is a direct path, $U_x$
is an interval where it is constant. Consequently, one can easily
prove inductively that $\gamma_{i+1}=\widetilde{F\circ\gamma_i}$ is
also constant on $U_x$ for $0\leq i\leq n-1$.
In particular $\gamma_n$ is constant on $U_x$.

The map $G$ is continuous. Then, by identifying $J$ with an interval
on $\IR$ we can use Lemma~\ref{lem:fixed-point-non-constant} to prove
that there exists $z \in I$ such that $G(z)=z$ and $z \notin
\Const(G)$.

It remains to show that $G(z)=F^n(z)$. Let $a$ be such that $z =
\gamma_0(a)$. Then, we have to show that $\gamma_n(a) =
F^n\circ\gamma_0(a)$. Suppose that, for some $j< n$,
$\gamma_{j+1}(a) \neq F\circ\gamma_j(a)$
and
$\gamma_j(a)=F^j\circ \gamma_0(a)$.
Applying Lemma~\ref{lem:const(gammatilde)} to the path
$F\circ\gamma_j$, we find that there exist $s_1<a<s_2$ such that
$F\circ\gamma_j(s_1)=F\circ\gamma_j(s_2)=\gamma_{j+1}(s)$
for all $s\in [s_1,s_2]$. Then, $\gamma_0(s_1) \neq \gamma_0(a)$
because
$F\circ\gamma_j(s_1) = \gamma_{j+1}(a) \neq F\circ\gamma_j(a)$, and
for the same reason $\gamma_0(a) \neq \gamma_0(s_2)$. Since the map
$\gamma_0$ is non decreasing, $z=\gamma_0(a)$ is in the interior of
$[\gamma_0(s_1),\gamma_0(s_2)]$. Moreover, $G$ is constant on this
interval, and thus $z\in\Const(G)$, which is a contradiction. We
conclude that $\gamma_{j+1}(a)= F\circ\gamma_j(a)$ for all $0\leq
j<n$, and thus $\gamma_n(a) = F^n\circ\gamma_0(a)$.
\end{proof}

\medskip
The next lemma is the key tool in the proof of
Theorem~\ref{theo:endpoints}.

\begin{lemm}\label{lem:R>0}
Let $F\in\Li$ be such that $\Clos{\TR}=T$ and $0 \le \min
\Rot(F)$. Suppose that there exists $y\in\IR$ such that
$(\Fcr)^n(y)\geq y$ for all $n\geq 1$, and $F^n(x) \neq x$ for all
$x\in T$ and $n\in\IN$. Then $\min\Rot(F)>0$.
\end{lemm}

\begin{proof}
By Lemma~\ref{lem:periodic-endpoints}, there exists a finite subset
$A$ such that $T\setminus\TR=A+\IZ$ and there exist a $s \in \IN$
which is common to all elements of $A$ and integers $\{i_a'\}_{a\in
A}$ such that for all $a\in A$, $F^s(a) = a + i_a'$. Hence, the
rotation number of every $a\in A$ exists and we have $\rhos{F}(a) =
i_a'/s$. Moreover, for each $a \in A$, $i_a' \ge 0$ because
$\min\Rot(F) \ge 0$. Moreover, $i_a'\geq 1$ because otherwise $a =
F^{s}(a)$, which contradicts our assumptions.

For every $a\in A$, let $V_a$ be a neighbourhood of $a$ such that
$F^{s}(V_a) \subset X+i_a'$. By Lemma~\ref{lem:T'}(b),
$\{F^n(\IR)\}_{n\geq 0}$ is an increasing sequence of connected sets
whose union is $\TR$. By assumption, $T= \Clos{\TR}$. Thus, there
exists an integer $\ell \geq 0$ such that $F^{\ell s}(\IR)$ contains
$X \setminus \bigcup_{a\in A} V_a$. Set $G = F^{\ell s}$. We have $T
\setminus G(\IR) \subset \bigl(\bigcup_{a\in A} V_a \bigr) + \IZ$ and
$G(a) = a + i_a$ for all $a\in A$, where $i_a=\ell i_a'\geq 1$.

In the rest of the proof we will use the distance $\nu$
on $T$ introduced in Definition~\ref{def:almosttaxicab}.
With this notation we set
\[
\delta_n := \min\set{\nu(G^n(x),x)}{x\in T}
          = \min\set{\nu(G^n(x),x)}{0\leq r(x)\leq 1}.
\]
Notice that $\delta_n > 0$ for all $n\geq 1$ because $G^n(x) \neq x$
for all $x \in T$.

Given a point $x \in T$ and a path $\gamma$ from $x$ to $+\infty$ we
iteratively define the following directed
paths: $\gamma_0 = \widetilde{\gamma}$ and
$\gamma_{i+1}=\widetilde{G\circ\gamma_i}$ for all $i \ge 0$.
Then, by Lemma~\ref{lem:path-periodicity}, the following property
holds for each $n \in \IN$.
\begin{equation}\label{eq:gammatilde}\text{\begin{minipage}{22em}
Either the paths $\gamma_0$ and $\gamma_n$ are not comparable or the
inequality $G^n(x) \leq_{\gamma_n} x$ does not hold.
\end{minipage}}\end{equation}
Suppose in addition that  $x\in\IR$ and $x\geq r\circ G^n(x)$ for some
$n\geq 1$. The path $\gamma_n$ goes from $G^n(x)$ to $+\infty$, and
thus its image contains $[r(G^n(x)),+\infty)\supset [x,+\infty)$.
Therefore, $\gamma_n$ and $\gamma_0$ are comparable and
$G^n(x)\leq_{\gamma_n}
x$, which contradicts \eqref{eq:gammatilde}. Consequently, we have:
\begin{equation}\label{eq:gammatildeR}
\forall x\in \IR,\ \forall n\geq 1,\ x<r\circ G^n(x).
\end{equation}

Now we prove the following claim, that means that if the orbit of some
point $y$ remains in $X+\IZ$ then it cannot go too much to the left
and has to go to the right of $y$ in bounded time.

\medskip \paragraph{\bf Claim:}
There exists $\widetilde{N} \in \IN$ such that if $y\in X$ verifies
that
\[
 \text{there exists $k_n\geq 0$ such that $G^n(y)\in X-k_n$,}
\]
for all $n = 1,\dots,N-1$, then $N< \widetilde{N}$.

\medskip\paragraph{\emph{Proof of the claim.}}
We know that
$G(\IR) \supset X \setminus \bigcup_{a\in A} V_a$.
Hence, since $G \in \Li$, there exists a point $x_0\in\IR$, $x_0<1$,
such that
$G([x_0,+\infty)) \supset X \setminus \bigcup_{a\in A} V_a$.
We denote the number $1 - \lfloor x_0 \rfloor$ by $p$, where
$\lfloor x_0 \rfloor$ denotes the integer part of $x_0$. Clearly,
$p$ bounds the number of copies of $X$ of the form $X-n$
fitting between $x_0$ and $1$.

Fix a path $\xi$ from $y$ to $+\infty$ and define
$\xi_0 = \widetilde{\xi}$ and
$\xi_{n+1} = \widetilde{G\circ\xi_n}$ for all $n\geq 0$.
We are going to show that a large proportion of
$\{G^n(y)\}_{0 \leq n < N}$ lie in the same basic interval and then a
large proportion of paths
$\{\xi_n\}_{0 \leq n < N}$
start in this particular interval and are comparable.

The point $y$ does not belong to $\bigcup_{a\in A} V_a$ because
$G(V_a) \subset X+i_a$ with $i_a\geq 1$. Hence
$y \in G(\IR) = G([x_0,+\infty))$.
Let $x\in\IR$, $x\geq x_0$ be such that $G(x)=y$.
Equation~\eqref{eq:gammatildeR} implies that
$x<r(G^{n+1}(x))=r(G^n(y))$
for all $n\geq 0$, and thus
$x_0 < r\circ G^n(y)$ for all $0 \leq n \leq N-1$,
and hence $0 \leq k_n \leq p-1$. In other words, the number of basic
intervals that contain one of the points $(G^n(y))_{0 \leq n < N}$ is
at most $p\Card(\Lambda)$. By the drawers principle, there exists a
basic interval $Y$ such that
\[
\Card(\set{0\leq n \leq N-1}{G^n(y)\in Y})
\geq
\frac{N}{p\Card(\Lambda)}.
\]

By Lemma~\ref{lem:nbclasses}, the number of equivalence classes of
comparable paths starting from $Y$ is finite. Let $q$ be this number.
Let $[\sigma]_Y$ denote the equivalence class of a path $\sigma$
starting in $Y$. Again by the drawers principle, there  exists a path
$\sigma$ such that the number of elements of the set
\[
{\mathcal N} = \set{0 \leq n \leq N-1}{G^n(y)\in Y \text{ and }
[\xi_n]_Y = [\sigma]_Y}
\]
is at least $\tfrac{N}{m}$, where $m$ denotes $pq\Card(\Lambda)$.
Write
\[
{\mathcal N}=\{n_1,n_2,\dots,n_k\}\text{ with } k\geq\frac{N}{m}
\text{ and } n_1<n_2<\cdots<n_k.
\]
We will show that there cannot be too many elements in ${\mathcal N}$
because of \eqref{eq:gammatilde}. If we choose $\sigma$ to be maximal
with respect to the inclusion relation of images then, for every $n
\in {\mathcal N}$, the ordering $\le_{\xi_n}$ is a restriction of
$\le_{\sigma}$. Since $\xi_n(0) = G^n(y)$, all the points
$\{G^n(y)\}_{n\in{\mathcal N}}$ belong to $Y$ and are ordered by
$\le_{\sigma}$. If $n_i,n_j \in {\mathcal N}$ with $n_i > n_j$ then
$G^{n_i}(y) \leq_{\sigma} G^{n_j}(y)$ contradicts
\eqref{eq:gammatilde} applied to the path $\xi_{n_i}$ and initial path
$\xi_{n_j}$ (notice that, by definition of ${\mathcal N}$, $\xi_{n_i}$
and $\xi_{n_j}$ are equivalent and hence comparable). Therefore
\begin{equation}\label{eq:Gny-ordered}
G^{n_1}(y)<_{\sigma}G^{n_2}(y)<_{\sigma}\cdots<_{\sigma}
G^{n_k}(y).
\end{equation}
This implies that $\nu(G^{n_k}(y),G^{n_1}(y))=\sum_{i=1}^{k-1}
\nu(G^{n_i}(y),G^{n_{i+1}}(y))$. This observation will be used to
find a lower bound of $\nu(G^{n_k}(y),G^{n_1}(y))$.

For $i=1,\dots,k-1$ set $j_i=n_{i+1}-n_i$. We have
$j_i \geq 1$ and
\[
j_1+\cdots+j_{k-1} = n_k-n_1 \leq N.
\]
Define also
$M:=\Card(\set{1\leq i\leq k-1}{ j_i \leq 2m})$.
Then, there are $k-1-M$ integers $i$ such that $j_i \geq 2m+1$ and
for the rest we have $j_i\geq 1$. Thus,
\[
N \geq j_1+\cdots+j_{k-1} \geq M + (2m+1)(k-1-M).
\]
Hence
\[
M \geq \frac{(2m+1)(k-1)-N}{2m} \geq k - 1 - \frac{N}{2m}.
\]
Since
$k \geq \tfrac{N}{m}$ it follows that $M\geq\tfrac{N}{2m}-1$.
According to the definition of $\delta_j$, we have
$\nu(G^{n_{i+1}}(y),G^{n_i}(y)) \geq \delta_{j_i}$.
There are $M$ integers $i$ such that $j_i \leq 2m$. Thus, in
view of Equation~\eqref{eq:Gny-ordered}, we get that
$\nu(G^{n_k}(y),G^{n_1}(y)) \geq M\kappa$, where
$\kappa = \min\{\delta_1,\ldots,\delta_{2m}\} > 0$.
Let $L$ denote the maximal length of all the basic intervals. It
follows that
$
L \geq \nu(G^{n_k}(y)),G^{n_1}(y)) \geq M\kappa,
$
and thus $N \leq 2m\left(\tfrac{L}{\kappa} + 1\right)$.
This concludes the proof of the claim by setting
$\widetilde{N} > 2m\left(\tfrac{L}{\delta}+1\right)$
(recall that $m = pq\Card(\Lambda)$ and that $L$ and $\kappa$
depend only on $T$ and $G$).

To end the proof of the lemma it is enough to show that
$\orhos[G](x) > 0$ for all $x\in T$. Indeed, by
Lemma~\ref{lem:first-properties} in that case we will have
$0 < \orhos[G](x) = \ell s\orhos(x)$
which implies $\orhos(x) > 0$ for all $x\in T$. Since
$\Rot^+(F) \supset \Rot(F)$ the lemma holds.

To prove that $\orhos[G](x) > 0$ for all $x\in T$ we consider three
cases.

\begin{case}{1}$G^n(x)\in\IR$ for all $n\geq N$.\end{case}
Then $G^{n+1}(x) \geq G^n(x) + \delta_1$ for all $n\geq N$ and
$\orhos[G](x) \geq \delta_1 > 0$.

\begin{case}{2}$G^n(x)\in X+\IZ$ for all $n\geq N$.\end{case}
By the Claim there exist two sequences $\{n_i\}_{i\geq 0}$ and
$(k_i)_{i\geq 0}$ such that
$n_{i+1}-n_i\leq \widetilde{N}$,
$k_{i+1}\geq k_i+1$ and
$G^{n_i}(x)\in X+k_i$ for all $i\geq 0$. Hence
$\orhos[G](x)\geq \tfrac{1}{\widetilde{N}}>0$.

\begin{case}{3}Assume that we are not in the first two
cases.\end{case}
Then, there exists an increasing sequence $\{n_i\}_{i\geq 0}$ such
that, for all $i\geq 0$, $G^{n_i}(x)\in\IR$ and $G^j(x)\not\in \IR$
for all $j \in\{ n_i+1,\dots,n_{i+1}-1\}$. For these $j$, let
$q_j\in\IZ$ be
such that $G^j(x)\in X+q_j$. Let $\widetilde{N}$ be the integer given
by the Claim, and define:
\begin{align*}
C := & \max\set{|r\circ G^n(x)-r(x)|}{x\in T,\ n\leq \widetilde{N}}\\
N_1 := & \lceil 2\widetilde{N}(C+2) + 2\rceil \\
\delta := & \min\{\delta_1,\ldots,\delta_{N_1}\} > 0
\end{align*}
(recall that $\lceil \cdot \rceil$ denotes the ceiling
function). Observe that all direct paths going to $+\infty$ and
stating at some point $G^{n_i}(x)$ are comparable because $G^{n_i}(x)
\in \IR$ for all $i$. Consequently, by \eqref{eq:gammatilde},
$G^{n_{i+1}}(x) > G^{n_i}(x)$.

If $n_{i+1}-n_i\leq N_1$ then
\begin{equation}\label{eqcas3-1}
 \frac{G^{n_{i+1}}(x)-G^{n_i}(x)}{n_{i+1}-n_i} \geq
 \frac{\delta}{N_1}.
\end{equation}

If $n_{i+1}-n_i> N_1$, by the Claim, there exist
$n_i+1 = j_1 < j_2 < \cdots < j_k <n_{i+1}$
such that
$j_{i+1} - j_i \leq \widetilde{N}$,
$n_{i+1} - j_k \leq \widetilde{N}$ and
$q_{j_{t+1}} - q_{j_t} \geq 1$
for all $1 \leq t\leq k-1$.
Hence $k \geq \tfrac{n_{i+1}-n_i-1}{\widetilde{N}}$.
Since $q_{j_{t+1}} - q_{j_t} \geq 1$,
the point $G^{j_k}(x)$ belongs to $X+q_{j_1} + m$ for some
$m \geq k-1$, and thus
$r\circ G^{j_k}(x) - r\circ G^{j_1}(x) \geq k-2$.
Moreover, $G^{n_i}(x)<r\circ G^{j_1}(x)$ because of
\eqref{eq:gammatildeR}.
Therefore,
\begin{align*}
G^{n_{i+1}}(x) - G^{n_i}(x)
=  & (G^{n_{i+1}}(x) - r\circ G^{j_k}(x))
       + (r\circ G^{j_k}(x) - r\circ G^{j_1}(x))\\
   &   + (r\circ G^{j_1}(x) - G^{n_i}(x)) \geq  -C+(k-2)+0\\
\geq & \frac{n_{i+1}-n_i-1}{\widetilde{N}}-C-2 \\
  =  &  \frac{n_{i+1}-n_i}{\widetilde{N}}-(C+2+1/\widetilde{N}).
\end{align*}
The choice of $N_1$ implies that $N_1 \ge 2\widetilde{N}(C+2) + 2$
which is equivalent to $\tfrac{N_1}{2\widetilde{N}} \ge
C+2+1/\widetilde{N}$. Consequently, since $n_{i+1}-n_i> N_1$
\begin{align*}
 G^{n_{i+1}}(x) - G^{n_i}(x)
\ge &
   \frac{n_{i+1}-n_i}{\widetilde{N}}-(C+2+1/\widetilde{N}) \ge
   \frac{2(n_{i+1}-n_i)}{2\widetilde{N}} -
   \tfrac{N_1}{2\widetilde{N}} \\
\ge & \frac{n_{i+1}-n_i}{2\widetilde{N}},
\end{align*}
which is equivalent to
\begin{equation}\label{eqcas3-2}
\frac{G^{n_{i+1}}(x)-G^{n_i}(x)}{n_{i+1}-n_i} \geq
\frac{1}{2\widetilde{N}}.
\end{equation}
Summarising, Equations~\eqref{eqcas3-1} and \eqref{eqcas3-2} imply
that
\[
 \orhos[G](x) \geq
\min\left\{\frac{\delta}{N_1},\frac{1}{2\widetilde{N}}\right\} > 0.
\]
This ends the proof of the lemma.
\end{proof}

\medskip
Now we are ready to prove the main result of this section.

\begin{theo}\label{theo:endpoints}
Let $T\in\InfG$ and let $F\in\Li$. If $\min\RotR(F)=p/q$ (resp.
$\max\RotR(F)=p/q$), then there exists a periodic {\modi} point $x\in
T$ such that $\rhos{F}(x)=p/q$.
\end{theo}

\begin{proof}
We deal only with the case $p/q = \min \RotR(F)$. The other one
follows similarly.

If $\RotR(F) = \{p/q\}$ then $\Rot(\Fcr) = \{p/q\}$ by
Corollary~\ref{cory:incl} and, in view of Theorem~\ref{theo:Fn}, there
exists a periodic {\modi} point of rotation number $p/q$. In the
rest of the proof we suppose that $\max \RotR(F) > p/q$.

Set $\overline{T} = \Clos{\TR}$ and $G =
(F^q-p)\evalat{\overline{T}}$. By Lemma~\ref{lem:T'}(b),
$\overline{T}\in\InfG$ and $G\in\Li[\overline{T}]$. By
Lemma~\ref{lem:first-properties} we have $0=\min \Rot(G)$ and
$\max\Rot(G)>0$. Also, $\overline{T}=\Clos{\bigcup_{n\geq 0}
G^n(\IR)}$ by Lemma~\ref{lem:T'}(b). By
Theorem~\ref{theo:RR-closed-interval}, there exists a positive integer
$N$ such that $\max\tfrac{1}{N}\Rot(G_N^r)>0$, and, by
Theorem~\ref{theo:Fn}, there exists $y\in\IR$ such that
$\rhos{G_N^r}(y)>0$ and the orbit of $y$ for $G^N$ is twist. In
particular, $(G_N^r)^n(y)\geq y$ for all $n\geq 0$. Let $H=G^N$ (hence
$H_1^r=G_N^r$). If $H^n(x) \neq x$ for all $x\in T$ and $n\in\IN$,
then $\min\Rot(H)>0$ by Lemma~\ref{lem:R>0}, which is a contradiction.
Therefore, there exist $x\in T$ and $n\in\IN$ such that $H^n(x)=x$,
and thus $x$ is periodic {\modi} for $F$ and  $\rhos{F}(x)=p/q$.
\end{proof}

\begin{rema}
Unfortunately, the periodic {\modi} point given by
Theorem~\ref{theo:endpoints} may be in $T\setminus \IR$ and there may
not exist a periodic {\modi} point $x$ in $\IR$ with $\rhos{F}(x)$
being an endpoint of the rotation interval (see
Example~\ref{ex:gaps-in-Per(p/q)}).
\end{rema}

\section{Examples}\label{sec:examples}

This section is devoted to showing some examples to help understanding
the theoretical results of the previous sections. Attention is payed
to the differences between this case and the circle one. For easiness
the first two examples will be Markov. To be able to compute the
periods {\modi} and rotation numbers of these examples we need to
develop the appropriate machinery. So we will divide this section into
two subsections. In the first one we will introduce the theoretical
results to make the computations in the examples whereas in the second
one we provide the examples themselves. Some of the properties of
rotation sets for symbolic systems used here already appear in
\cite{Zie}.

\subsection{Preliminary results on Markov lifted graph maps}
We say that a subset of $T$ is an \emph{interval} if it is
homeomorphic to an interval of the real line and does not contain
vertices (except maybe at its endpoints). In other words, a subset of
$T$ is an interval if it is still homeomorphic to an interval after
removing the vertices of $T$.

An interval of $T$ can be endowed with two opposite linear orderings
compatible with its structure of interval. If $I, J$ are two intervals
of $T$, we choose arbitrarily one of these two orderings for each
interval, and we say that a map $\map{f}{I}[J]$ is \emph{monotone} if
it is monotone with respect to these orderings. Notice that this is
independent of the choice of the orderings.

Let $T\in\InfG$ and let $\nu$ be the distance on $T$ introduced in
Definition~\ref{def:almosttaxicab}. When $T$ is an infinite tree
(i.e., it is uniquely arcwise connected), then $\nu(x,y)$ coincides
with the \emph{taxicab metric} which gives the length of the shortest
path in $T$ from $x$ to $y$. We say that $\map{f}{I}[J]$ is
\emph{affine} if there exists $\lambda\in \IR$ such that for all
$x,y\in I$, $\nu(f(x),f(y)) = \lambda\nu(x,y)$. Observe that if $f$ is
affine then it is also monotone.

Now we adapt the well known notion of Markov map to the context of
lifting graphs.

\begin{defi}
Let $T\in \InfG$ and let $F\in \Li$. We say that $F$ is a
\emph{Markov map} if there exist compact intervals $P_1,\ldots,P_k$
such that
\begin{enumerate}[(i)]
\item the vertices of $T$ are included in $\bigcup_{i=1}^k \partial
P_i+\IZ$,

\item $(P_1\cup\cdots\cup P_k) +\IZ=T$,

\item if $i \neq j$ then $P_i\cap P_j$ contains at most one point,

\item for all $1\leq i\leq k$, $F(P_i)$ is an interval,
$\map{F\evalat{P_i}}{P_i}[F(P_i)]$ is monotone, and $F(P_i)$ is a
finite union of sets $\{P_j+n\}_{1\leq j\leq k, n\in\IZ}$.
\end{enumerate}
When we will need to specify it, we will say that \emph{$F$ is a
Markov map with respect to the partition $(P_1,\ldots,P_k)$}, or that
\emph{$(P_1,\ldots,P_k)$ is the Markov partition of $F$}.

The Markov map $F$ is called \emph{affine} if $F\evalat{P_i}$ is
affine for all $1\leq i\leq k$.

If $F(P_i)\supset P_j+n$, we write $P_i\labelarrow{n} P_j$. This gives
a finite labelled oriented graph, which is called the \emph{Markov
graph of $F$} and denoted by $\CG(F)$. If $B=\{B_1,\ldots,B_p\}$
and $A\labelarrow{n} B_i$ for all $1\leq i\leq p$, we also write (or
picture) $A\labelarrow{n} B$ for short.
\end{defi}

We now give some notations about paths in graphs that we will need
later.

Let $\CG$ be a finite labelled oriented graph. A \emph{(finite) path}
is a
sequence of labelled arrows in $\CG$ of the form
\[
\CA :=A_0\labelarrow{n_0} A_1\labelarrow{n_1} \cdots
      A_{p-1}\labelarrow{n_{p-1}} A_p.
\]
The \emph{length} of $\CA$ is $L(\CA)=p$ and its \emph{weight} is
$W(\CA)=n_0+\cdots+n_{p-1}$.

If $\CB := B_0\labelarrow{m_0}
B_1\labelarrow{m_1}\cdots B_{q-1}\labelarrow{m_{q-1}} B_q$ is another
path with $B_0=A_p$, we define the concatenated path as
\[
A_0\labelarrow{n_0} A_1\labelarrow{n_1}\cdots
A_{p-1}\labelarrow{n_{p-1}}
A_p \labelarrow{m_0} B_1\labelarrow{m_1}\cdots B_{q-1}
\labelarrow{m_{q-1}} A_q.
\]
Such a path will be denoted by $\CA \CB$. A path $\CA$ is called a
\emph{loop} if $A_0=A_p$. In such a case, $\CA^0$ denotes the empty
path and, for $n\geq 1$, $\CA^n$ denotes the path
\[
\overbrace{\CA\CA\cdots\CA}^{\text{$n$ times}}.
\]
Also, $\CA^\infty$ denotes the loop $\CA$ concatenated with itself
infinitely many times, which gives an infinite path.

A loop $\CA$ is called \emph{simple} if it is not of form $\CB^n$,
$\CB$ being a shorter loop and $n\geq 2$.
A loop is \emph{elementary} if it cannot
be formed by concatenating two loops, up to a circular permutation.
Equivalently,
$
\CA := A_0\labelarrow{n_0} A_1\labelarrow{n_1} \cdots
       A_{p-1}\labelarrow{n_{p-1}} A_0
$
is elementary if $A_0,\ldots,A_{p-1}$
are all pairwise different. Observe that the number of distinct
elementary loops in $\CG$ is finite.

If
$
\CA :=A_0\labelarrow{n_0} A_1\labelarrow{n_1} \cdots
      A_{p-1}\labelarrow{n_{p-1}} A_p\labelarrow{n_p}\cdots
$
is an infinite path, let $\CA_i^j$ denote the truncated path
$A_i\labelarrow{n_i}  \cdots \labelarrow{n_{j-1}}A_j$,
where $0\leq i<j$.

Suppose that $\CG$ is the Markov graph of a Markov map $F\in\Li$ and
let $x\in T$. We say that an infinite path
\[
\CA:=A_0\labelarrow{n_0} A_1\labelarrow{n_1} \cdots
     A_{p-1}\labelarrow{n_{p-1}} A_p\labelarrow{n_p}\cdots
\]
is an \emph{itinerary} of $x$ if there exists $n(x)\in\IZ$ such that
$F^i(x)\in A_i+n(x)+W(\CA_0^i)$ for all $i\geq 0$. If in addition
there exists a loop $\CB$ such that $\CA=\CB^{\infty}$, then we say
that $\CB$ is a \emph{periodic itinerary} of $x$.

\medskip
The following proposition is a version for lifted graph maps
of folk knowledge properties of Markov maps on finite topological
graphs.

\begin{prop}\label{prop:Fmarkov}
Let $T\in\InfG$ and let $F\in\Li$ be a Markov map with respect to the
partition $(P_1,\ldots,P_k)$.

\begin{enumerate}
\item If $F$ is an affine Markov map such that $\CG(F)$ is connected
and is not reduced to a unique loop, then $F$ is transitive
{\modi}.

\item  For every $x\in T$, there exists  an infinite path in $\CG(F)$
which is an itinerary of $x$.

\item If $x\in T$ is a periodic {\modi} point, then there exists a
simple loop $\CB$ in $\CG(F)$ which is a periodic itinerary of $x$.
Moreover, if the period {\modi} of $x$ is $p$ and
$F^i(x)\notin\bigcup_{j=1}^{k} \partial P_i+\IZ$ for all $0\leq i<p$,
then $p=L(\CB)$ and $F^p(x)=x+W(B)$.

\item Every infinite path in $\CG(F)$ is an itinerary of some point
$x\in T$. Every loop in $\CG(F)$ is a periodic itinerary of some
periodic {\modi} point $x$.
\end{enumerate}
\end{prop}

The next two lemmas show how the rotation numbers and the rotation
set can be deduced from the Markov graph.

\begin{lemm}\label{lem:rho-Markov}
Let $F\in\Li$ be a Markov map with $T\in\InfG$ and let $x\in T$ be
such that $\rhos{F}(x)$ exists. If the infinite path
$\CA:=A_0\labelarrow{n_0}\cdots A_i\labelarrow{n_i}A_{i+1}\cdots$
is an itinerary of $x$ in $\CG(F)$, then
\[
\rhos{F}(x)=\lim_{i\to+\infty}\frac{W(\CA_0^i)}{i}.
\]
If $\CB$ is a loop in $\CG(F)$ which is a periodic itinerary of $x$,
then
$\rhos{F}(x)=\tfrac{W(\CB)}{L(\CB)}$.
\end{lemm}

\begin{proof}
Let $(P_1,\ldots,P_k)$ be the Markov partition of $F$. By
definition of an itinerary, $F^i(x)-n(x)-W(\CA_0^i)\in A_i$
for all $i\geq 0$. The set $\set{r(y)}{y\in P_1\cup\cdots\cup P_k}$ is
bounded, and $A_i\in\{P_1,\ldots,P_k\}$ for all $i\geq 0$. Therefore
$r\circ F^i(x)-W(\CA_0^i)$ is bounded too, and thus
$
\rhos{F}(x)=\lim_{i\to+\infty}\frac{W(\CA_0^i)}{i}.
$

Suppose that the loop $\CB$ is a periodic itinerary of $x$. Then
$\CA=\CB^{\infty}$ is an itinerary of $x$, and
$W(A_0^{iL(\CB)})=iW(\CB)$ for all $i\geq 0$. What precedes implies
that $\rhos{F}(x)=\tfrac{W(\CB)}{L(\CB)}$.
\end{proof}

\begin{lemm}\label{lem:R-Markov}
Let $F\in\Li$ be a transitive {\modi} Markov map with $T\in\InfG$ and
set
\[
m:=\min_{\CE}
\frac{W(\CE)}{L(\CE)}\quad\text{and}\quad
M:=\max_{\CE}\frac{W(\CE)}{L(\CE)},
\]
where $\CE$ ranges over the set of all elementary loops in $\CG(F)$.
Then $\Rot(F)=[m,M]$.
\end{lemm}

\begin{proof}
By Theorem~\ref{theo:RR-closed-interval} $\RotR(F)$ is a compact
interval, and by Theorem~\ref{theo:RR(F)=R(F)} $\Rot(F)=\RotR(F)$
because $F$ is transitive {\modi}. By
Proposition~\ref{prop:Fmarkov}(d) and Lemma~\ref{lem:rho-Markov}, $m$
and $M$ belong to $\Rot(F)$, and hence $[m,M]\subset \Rot(F)$.

Let $x\in T$ such that $\rhos{F}(x)$ exists and let
$
\CA:=A_0\labelarrow{n_0} A_1\labelarrow{n_1} \cdots
     A_k\labelarrow{n_k}\cdots
$
be an itinerary of $x$, which exists by
Proposition~\ref{prop:Fmarkov}(b). Since the number of vertices in
$\CG(F)$ is finite, there exists $P$, an element of the partition, and
an increasing sequence $k_i$ such that $A_{k_i}=P$ for all $i\geq 0$.
By Lemma~\ref{lem:rho-Markov}, $\rhos{F}(x)=\rhos{F}(F^{k_0}(x))$ is
equal to
\[
\lim_{i\to\infty} \frac{W(\CA_{k_0}^{k_i})}{k_i-k_0}.
\]
If we decompose the loop $\CA_{k_0}^{k_i}$
into elementary loops, we see that the above quantity is a
barycentre of
\[
 \left\{\frac{W(\CE)}{L(\CE)}\,\colon
 \text{$\CE$ is an elementary loop of $\CG(F)$.}\right\}
\]
Hence, $\rhos{F}(x)\in[m,M]$.
\end{proof}

\subsection{The examples}
\begin{exem}\label{ex:Per0not=N}
\emph{$\Rot(F)=[-1/2,1/2]$,
$\Per(0,F)=\{1\}\cup\{n\geq 4\}$ and if $p/q\in \Rot(F)$ with $p,q$
coprime, $p \neq 0$ then $\Per(p/q,F)=q\IN$.}\bigskip

Let $F\in\Li$ be the affine Markov map represented in
Figure~\ref{fig:Per0not=N}. By Proposition~\ref{prop:Fmarkov}(a), $F$
is transitive {\modi}.

\begin{figure}[htb]
\centerline{\includegraphics{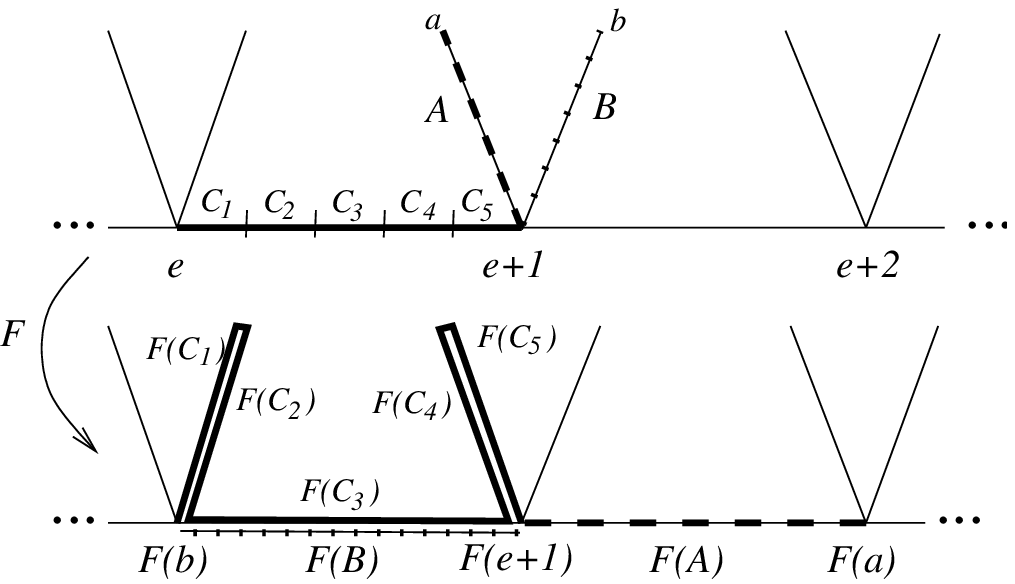}}
\par\medskip
\centerline{\includegraphics{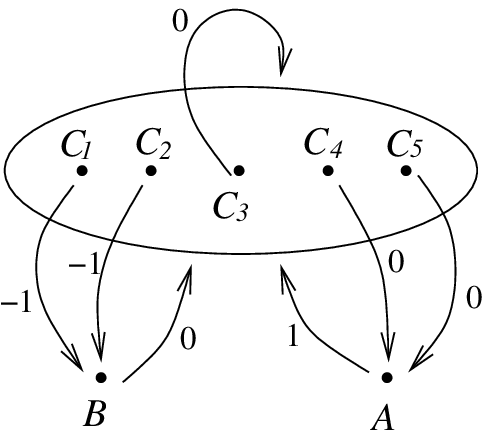}}
\caption{In the picture above the affine Markov map $F$. Below it is
displayed its Markov graph. In this example
$\Rot(F)=[-1/2,1/2]$,
$\Per(0,F)=\{1\}\cup\{n\geq 4\}$ and if $p/q\in \Rot(F)$ with $p,q$
coprime, $p \neq 0$ then $\Per(p/q,F)=q\IN$.\label{fig:Per0not=N}}
\end{figure}

We define:
\begin{align*}
\CA := & C_4\labelarrow{0} A\labelarrow{1} C_4,\\
\CB := & C_2\labelarrow{-1} B\labelarrow{0} C_2 \text{ and}\\
\CC := & C_3\labelarrow{0} C_3.
\end{align*}
The loops $\CB$ and $\CA$ correspond respectively to
$\min \tfrac{W(\CE)}{L(\CE)}$ and
$\max \tfrac{W(\CE)}{L(\CE)}$, where
$\CE$ describes the set of elementary loops.
Thus $\Rot(F)=[-1/2,1/2]$ by Lemma~\ref{lem:R-Markov}. Notice that
\[
\Rot(F)=\tfrac{1}{2}\Rot(F_2^r)
\quad\text{but}\quad
\tfrac{1}{3}\Rot(F_3^r)= [-1/3,1/3] \neq \Rot(F),
\]
and thus $\left\{\tfrac{1}{n}\Rot(\Fncr)\right\}_{n\geq 1}$ is not an
increasing sequence of sets.

We are going to show that $0$ is the unique rational $p/q\in\Rot(F)$
such that $\Per(p/q,F) \neq q\IN$. Moreover there is a ``gap'' in
$\Per(0,F)$: it contains $1$ and $4$ but not $2$ and $3$ (it is also
possible
to construct examples with more than one gap in $\Per(0,F)$).
This shows that $\Per(p/q,F)$ is not necessarily of the form
$\set{nq}{n\geq N}$ when $p/q\in \Int{\Rot(F)}$ and $p,q$ coprime.

We compute $\Per(p/q,F)$ by using Proposition~\ref{prop:Fmarkov} and
Lemma \ref{lem:rho-Markov}. If $x$ is an endpoint of one of the
intervals of the Markov partition then, either $x$ is not periodic
{\modi}, or $x=e\ \modi$ and $F(x)=x$.

\superparagraph{$\Per(0,F)=\{1\}\cup\{n\geq 4\}$}
From one side $1\in\Per(0,F)$ because $F(e)=e$. Also, there are no
simple loops of length $2$ or $3$ and weight $0$. Thus
$2,3\notin\Per(0,F)$. For all $k\geq 0$, the loop
\[
(C_2\labelarrow{-1} B\labelarrow{0} C_4\labelarrow{0} A\labelarrow{1}
C_3)\CC^k
\]
is simple, its length is $k+4$ and its weight is $0$ (if $k=0$, take
$C_2$ instead of $C_3$ to get a loop). Thus there exists a point $x$
such that $F^{k+4}(x)=x$ and $k+4\in\Per(0,F)$.

\superparagraph{$\Per(p/q,F)=q\IN$ if $p/q\in [-1/2,1/2]$, $(p,q) = 1$
and $p \neq 0$}
If $p\geq 1$, $q>2p$ and $n\geq 1$, we consider the loop
\[
\CA^{np-1}(C_4\labelarrow{0} A\labelarrow{1} C_3)\CC^{n(q-2p)-1}
(C_3\labelarrow{0} C_4)^{n(q-2p)-1}.
\]
It is simple, its length is $nq$ and its weight is $np$. Thus there
exists a periodic {\modi} point of period $nq$ and rotation number
$p/q$. For $q=2p$ and $np\geq 2$ consider the simple loop
\[
\CA^{np-2}
(C_4\labelarrow{0} A\labelarrow{1} C_5\labelarrow{0} A\labelarrow{1}
C_4).
\]

For $p=n=1$ and $q=2$ we consider the loop $\CA$.

If $p<0$ then the same arguments hold with $\CB$ instead of $\CA$.

Therefore, if $p/q\in [-1/2,1/2]$, $p \neq 0$ then $\Per(p/q,F)\supset
q\IN$. To conclude, we use that, if $p,q$ are coprime, then
$\Per(p/q,F)\subset q\IN$ by Proposition~\ref{prop:Per-subset-qN}.
\end{exem}

\begin{exem}\label{ex:gaps-in-Per(p/q)}
\emph{$\Rot(F)=[0,1]$, $\bigcup_{n\geq 1}\tfrac{1}{n}
\Rot(\Fncr)=(0,1]$ is not closed and there exist infinitely many
$p/q\in (0,1)$ with $p,q$ coprime such that $\Per(p/q,F) \neq
q\IN$.}\bigskip

Let $F\in\Li$ be the affine Markov map represented in
Figure~\ref{fig:gaps-in-Per(p/q)}. By
Proposition~\ref{prop:Fmarkov}(a), $F$ is transitive {\modi}.

\begin{figure}[htb]
\centerline{\includegraphics[width=25pc]{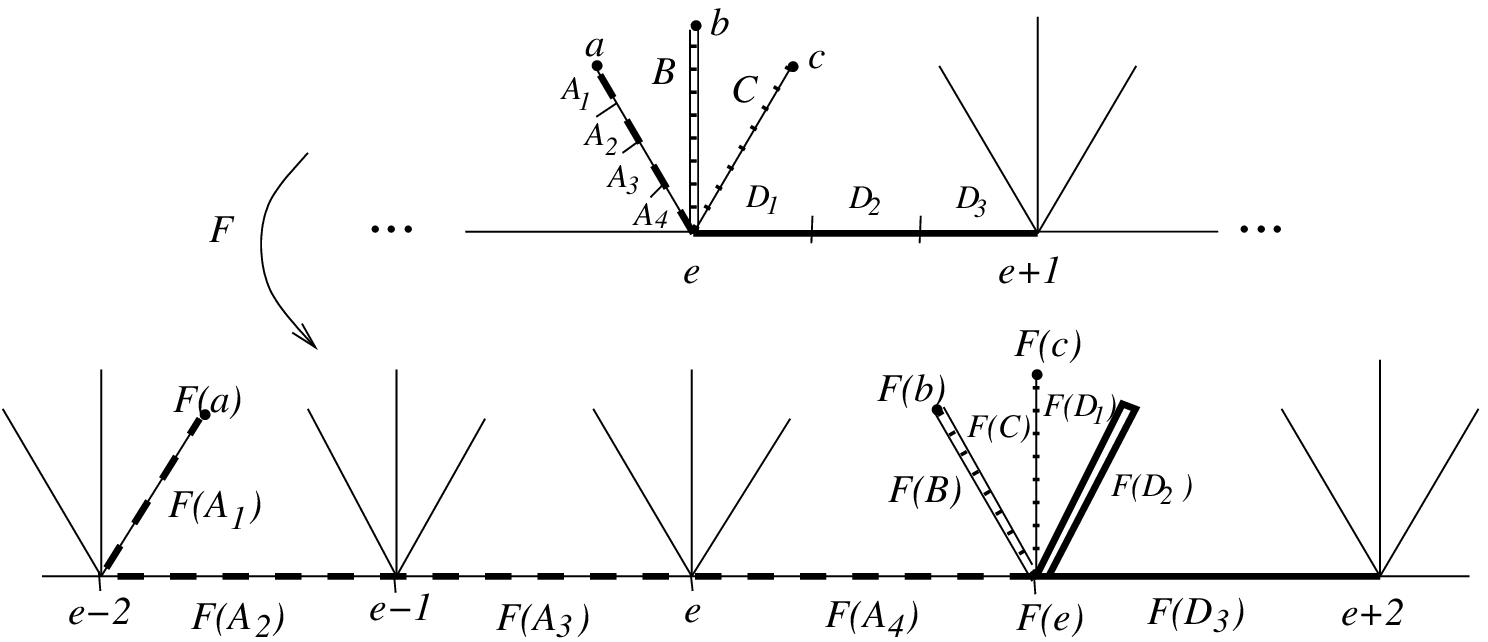}}
\par\medskip
\centerline{\includegraphics{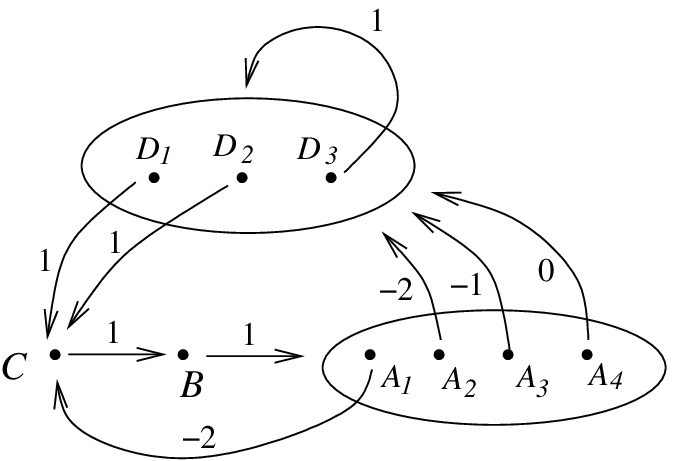}}
\caption{In the picture above the affine Markov map $F$. Below it is
displayed its Markov graph. In this example
$\Rot(F)=[0,1]$, $\bigcup_{n\geq 1}\tfrac{1}{n}
\Rot(\Fncr)=(0,1]$ is not closed and there exist infinitely many
$p/q\in (0,1)$ with $p,q$ coprime such that $\Per(p/q,F) \neq
q\IN$.\label{fig:gaps-in-Per(p/q)}}
\end{figure}

We define the loops
\begin{align*}
\CA := & C\labelarrow{1} B\labelarrow{1} A_1\labelarrow{-2} C, \\
\CB_{\eps} := & C\labelarrow{1} B\labelarrow{1} A
                \labelarrow{\eps} D_2\labelarrow{1} C,\text{ and} \\
\CD := & D_3\labelarrow{1} D_3,
\end{align*}
where  $\eps \in \{0,-1,-2\}$ and $A$ represents either
$A_2, A_3$ or $A_4$ depending on $\eps$.

The weights of $\CA$, $\CD$ and $\CB_{\eps}$ are respectively $0$, $1$
and $3+\eps$. Modifying $\CD$ into $D_2\labelarrow{1} D_3$ and
$\CB_{\eps}$ into $D_3\labelarrow{1} C\labelarrow{1} B\labelarrow{1} A
\labelarrow{\eps} D_2$, we can concatenate them. For short, we will
write $\CD\CB_{\eps}$ as the concatenated loop.

The only periodic {\modi} points which are endpoints of intervals of
the Markov partition are $e\ \modi$, with $F(e)=e+1$,  and $a,b,c$,
that form a periodic {\modi} orbit of period $3$ and rotation number
$0$. They correspond respectively to the loops $\CD$ and $\CA$.
Therefore, by Proposition~\ref{prop:Fmarkov}(c,d), there is a
correspondence between periodic {\modi} points of periods $p$ and
simple loops of length $p$.

The loops $\CA$ and $\CD$ correspond respectively to $\min
\tfrac{W(\CE)}{L(\CE)}$ and $\max \tfrac{W(\CE)}{L(\CE)}$, where $\CE$
describes the set of elementary loops. Thus $\Rot(F)=\RotR(F)=[0,1]$
by Lemma~\ref{lem:R-Markov}. According to
Theorem~\ref{theo:RR-closed-interval},
$(0,1) \subset \bigcup_{n\geq1}\tfrac{1}{n}\Rot(\Fncr)$.
The only simple loop of weight $0$ is $\CA$. It is the periodic
itinerary of $c$, which is of period $3$, and thus $\Per(0,F)=\{3\}$.
Moreover, $F^n(c)\notin\IR$ for all $n\geq 0$. Thus
$0\notin\Rot(\Fncr)$ by Theorem~\ref{theo:Fn}. The only simple loop
$\CL$ with $W(\CL)/L(\CL)=1$ is $\CD$. It is the periodic itinerary of
$e+1\in\IR$, and $F(e+1)=(e+2)$. Thus $\Per(1,F)=\{1\}$ and
$\bigcup_{n\geq 1}\tfrac{1}{n}\Rot(\Fncr)=(0,1]$.

We are going to compute $\Per(p/q,F)$ for all $p/q\in (0,1)$, $p,q$
coprime. The final results are given in Table~\ref{tab:Per(p/q)}.

\begin{table}[htb]
\[
\begin{array}{ccc}
 & & \Per(p/q,F)\\
\hline
p=1 & q\equiv 0\bmod\ 3 &\set{nq}{n\geq 3}\\
    & q\equiv 1\bmod\ 3 & q\IN\\
    & q\equiv 2\bmod\ 3 & \set{nq}{n\geq 2}\\
\hline
p=2 & q\equiv 0\bmod\ 3 & \set{nq}{n\geq 2}\\
    & q\equiv 1,2 \bmod\ 3 & q\IN\\
\hline
p\geq 3 & & q\IN\\
\hline
\end{array}
\]
\caption{Values of $\Per(p/q,F)$ for $p/q\in (0,1)$ and $p,q$
coprime.\label{tab:Per(p/q)}}
\end{table}

\smallskip\noindent$\bullet$
The only loops of weight $1$ are $\CD$ (length $1$) and
$\CB_{-2}\CA^k$ (which is of length $4+3k$), for all $k\geq 0$. Thus
there exists a periodic {\modi} point of period $n$ and rotation
number $1/n$ if and only if $n\equiv 1\bmod\
3$.

\smallskip\noindent$\bullet$
The only simple loops of weight $2$ are $\CB_{-2}^2\CA^k$ (length
$8+3k$), $\CB_{-1}\CA^k$ (length $4+3k$), and $\CB_{-2}\CD\CA^k$
(length $5+3k$), for all $k\geq 0$. Thus there exists a periodic
{\modi} point of period $n$ and rotation number $2/n$ if and only if
$n\equiv 1$ or $2$ $\bmod\ 3$ and $n\geq 4$.

\smallskip\noindent$\bullet$
Considering the simple loops $\CB_0\CA^k$, $\CD\CB_{-1}\CA^k$ and
$\CD^2\CB_{-2}\CA^k$ of weight $3$, we see that, for all $n\geq 4$,
there exists a periodic {\modi} point of period $n$ and rotation
number $3/n$. For all $n\geq 4$, we call $\CL_n$ the loop of length
$n$ among the above loops. We notice that $\CL_n$ passes through
$D_2$.

\smallskip\noindent$\bullet$
If $m\geq 4$ and $n>m$, then $n-m+3\geq 4$. The loop
$\CL_{(n-m+3)}\CD^{m-3}$ is of length $n$ and weight $m$, and thus it
gives a periodic {\modi} point of period $n$ and rotation number
$m/n$. This completes Table~\ref{tab:Per(p/q)}.

This example shows that there may exist infinitely many rationals
$p/q$, with $p$ and $q$ coprime, in the interior of the rotation
interval such that $\Per(p/q,F) \neq q\IN$, and the integer $N$
of Theorem~\ref{theo:set-of-periods} cannot be taken the same for the
whole interval $\RotR(F)$. Moreover, the interval
\[
\bigcup_{n \geq 1}\tfrac{1}{n}\Rot(\Fncr)
\]
may not be closed and, if $0 \in \partial\RotR(F)$, there may not
exist a periodic {\modi} point $x\in\IR$ with $\rhos{F}(x)=0$
(although there exists $x\in T$ with this property).

Compare also this situation with the one for combed maps. In view of
Theorem~\ref{theo:rotset}, the rotation interval of a combed map is a
closed interval and coincides with $\Rot(\Fcr)$. Moreover, in view of
Theorem~\ref{theo:setper}, $\Per(p/q,F) = q\IN$ for every $p/q\in
\Int(\Rot(F))$ with $p,q$ coprime. This example shows that both
statements can fail for a non combed map lifted graph map.
\end{exem}

\begin{exem}\label{ex:pb-endpoint}
\emph{$\RotR(F)=[0,1]$ but there is no periodic {\modi} point $x\in T$
such that $\rhos{F}(x)=0$.}\bigskip

We define $T$ as the following subset of $\IR^3$:
\[
 T = \set{(x,0,0)}{x\in \IR} \cup \set{((n,y,z)}{n\in \IZ
\text{ and } y^2 + z^2 \le 1}.
\]
Clearly, $T \in \InfX$. To be able to define a map $F$ on $T$ we
will identify the $z,y$-plane with $\IC$ taking the $y$ axis as the
real axis in $\IC$. Then we define the sets
\begin{align*}
 D = & \set{z\in\IC}{|z|\leq 1} \text{ and}\\
 \CC = & \set{z\in\IC}{|z|=1}.
\end{align*}
We identify the $x$-axis with $\IR$ and we denote $D+(n,0,0)$ by $D+n$
when $n\in\IZ$.
Note that to define a map $F \in \Li$ it is enough to define it on $D
\cup (x,0,0)$ with $x \in [0,1]$ and extend the definition to the
whole $T$ by using that $F(z+1)$ must be $F(z) +1$. Thus, we construct
our map by choosing $\alpha \in \IR\setminus\IQ$ defining
(see Figure~\ref{fig:pb-endpoint} for a representation of $T$ and
$F$):
\begin{enumerate}[(1)]
\item If $z \in D$ with $\tfrac{1}{2} \leq |z|\leq 1$ then $F(z)=z'\in
D$ with $|z'| = 2|z|-1 \in [0,1]$ and $\arg(z')=\arg(z)+2\pi\alpha$.

\item If $z\in D$ with $0 \leq |z| \leq\tfrac{1}{2}$ then
$F(z)=(1-2|z|, 0, 0) \in \IR$.

\item If $x\in [0,1/2]$, $F(x, 0, 0)=(1,1-4|x-1/4|, 0)$.

\item if $x \in [1/2,1]$, $F(x, 0, 0)=(2x, 0, 0)$.
\end{enumerate}
The map $F\evalat{\CC}$ is the rotation of angle $2\pi\alpha$.
\begin{figure}
\centerline{\includegraphics[width=25pc]{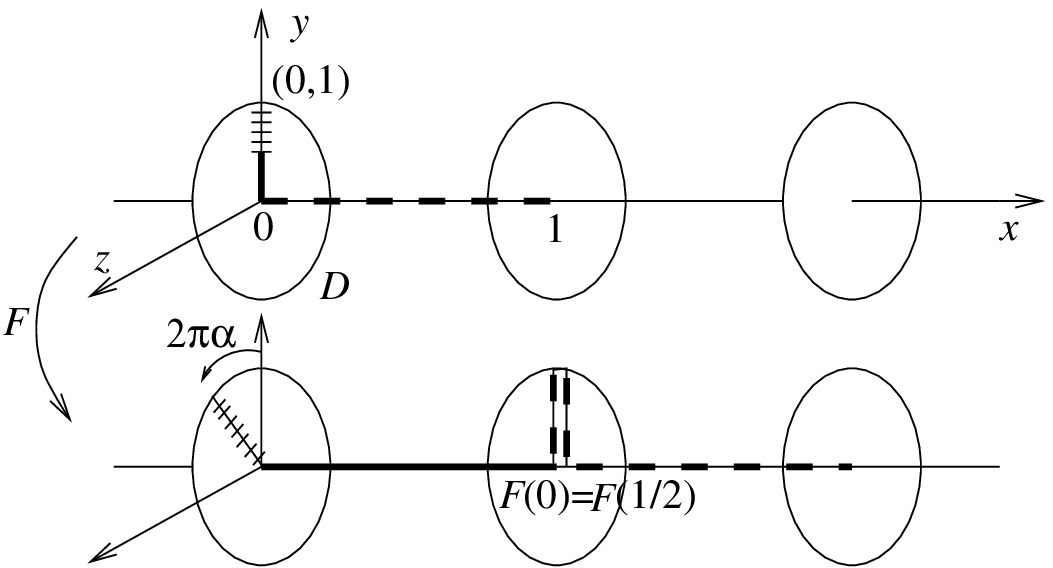}}
\caption{$T\in\InfX$, $F\in \Li$, $\RotR(F)=[0,1]$ but there
is no periodic {\modi} point $x\in T$ such that
$\rhos{F}(x)=0$.\label{fig:pb-endpoint}}
\end{figure}

We are going to show that $\RotR(F)=[0,1]$ but there is no periodic
{\modi} point $x\in T$ such that $\rhos{F}(x)=0$.

It is clear that, if $x\in D+\IZ$, then $r(x)\leq r(F(x))\leq r(x)+1$.
And, if $x\in \IR$, then $r(x)+\tfrac{1}{2}\leq r(F(x))\leq
r(x)+1$. Hence $\Rot(F)\subset [0,1]$. Moreover $F(0,0,0)=(1,0,0)$ and
$F^{n+1}(1/4,0,0)=F^n(1,1,0)= (1,e^{i2\pi n\alpha},0)$. Thus
$\rhos{F}(0,0,0)=1$, $\rhos{F}(1/4,0,0)=0$ and
$\RotR(F)=\Rot(F)=[0,1]$ by
Theorem~\ref{theo:RR-closed-interval}.

Suppose that $x\in T$ is a periodic {\modi} point such that
$\rhos{F}(x)=0$. Because of the properties stated above, $x$ cannot
belong to $\IR$, and there exists $k\in\IZ$ such that $F^n(x)\in D+k$
for all $n\geq 0$. By definition of $F|_D$, the point $x=(k,z)$ with
$k\in \IZ$ and $z \in D$ must belong to $\CC+k$. Thus
$F^n(x)=(k,ze^{i2\pi n\alpha}) \neq x$ for all $n\geq 1$. This is a
contradiction and, hence, $\Per(0,F)=\emptyset$.
\end{exem}

\bibliographystyle{plain}

\begin{thebibliography}{10}

\bibitem{AJM1}
Ll. Alsed{\`a}, D.~Juher, and P.~Mumbr{\'u}.
\newblock Sets of periods for piecewise monotone tree maps.
\newblock {\em Internat. J. Bifur. Chaos Appl. Sci. Engrg.}, 13(2):311--341,
  2003.

\bibitem{AJM3}
Ll. Alsed{\`a}, D.~Juher, and P.~Mumbr{\'u}.
\newblock On the preservation of combinatorial types for maps on trees.
\newblock {\em Ann. Inst. Fourier (Grenoble)}, 55(7):2375--2398, 2005.

\bibitem{AJM2}
Ll. Alsed{\`a}, D.~Juher, and P.~Mumbr{\'u}.
\newblock Periodic behavior on trees.
\newblock {\em Ergodic Theory Dynam. Systems}, 25(5):1373--1400, 2005.

\bibitem{AJM4}
Ll. Alsed{\`a}, D.~Juher, and P.~Mumbr{\'u}.
\newblock Minimal dynamics for tree maps.
\newblock {\em Discrete Contin. Dyn. Syst. Ser. A}, to appear.

\bibitem{ALM2}
Ll. Alsed{\`a}, J.~Llibre, and M.~Misiurewicz.
\newblock Periodic orbits of maps of {$Y$}.
\newblock {\em Trans. Amer. Math. Soc.}, 313(2):475--538, 1989.

\bibitem{ALM}
Ll. Alsed{\`a}, J.~Llibre, and M.~Misiurewicz.
\newblock {\em Combinatorial dynamics and entropy in dimension one}.
\newblock Advanced Series in Nonlinear Dynamics, 5. World Scientific Publishing
  Co. Inc., River Edge, NJ, 1993.

\bibitem{AMM}
Ll. Alsed{\`a}, F.~Ma{\~n}osas, and P.~Mumbr{\'u}.
\newblock Minimizing topological entropy for continuous maps on graphs.
\newblock {\em Ergodic Theory Dynam. Systems}, 20(6):1559--1576, 2000.

\bibitem{Bal}
S.~Baldwin.
\newblock An extension of sharkovski\u{\i}'s theorem to the {$n\text{-od}$}.
\newblock {\em Ergodic Theory Dynam. Systems}, 11(2):249--271, 1991.

\bibitem{BL}
S.~Baldwin and J.~Llibre.
\newblock Periods of maps on trees with all branching points fixed.
\newblock {\em Ergodic Theory Dynam. Systems}, 15(2):239--246, 1995.

\bibitem{Bern}
C.~Bernhardt.
\newblock Vertex maps for trees: algebra and periods of periodic orbits.
\newblock {\em Discrete Contin. Dyn. Syst.}, 14(3):399--408, 2006.

\bibitem{Bloc3}
L.~Block.
\newblock Homoclinic points of mappings of the interval.
\newblock {\em Proc. Amer. Math. Soc.}, 72(3):576--580, 1978.

\bibitem{BGMY}
L.~Block, J.~Guckenheimer, M.~Misiurewicz, and L.~S. Young.
\newblock Periodic points and topological entropy of one dimensional maps.
\newblock In {\em Global Theory of Dynamical Systems}, Lecture Notes in
  Mathematics, no. 819, pages 18--34. Springer-Verlag, 1980.

\bibitem{Ito}
R.~Ito.
\newblock Rotation sets are closed.
\newblock {\em Math. Proc. Cambridge Philos. Soc.}, 89(1):107--111, 1981.

\bibitem{LLl}
M.~C. Leseduarte and J.~Llibre.
\newblock On the set of periods for {$\sigma$} maps.
\newblock {\em Trans. Amer. Math. Soc.}, 347(12):4899--4942, 1995.

\bibitem{LPR}
J.~Llibre, J.~Para{\~n}os, and J.~A. Rodr{\'{\i}}guez.
\newblock Periods for continuous self-maps of the figure-eight space.
\newblock {\em Internat. J. Bifur. Chaos Appl. Sci. Engrg.}, 13(7):1743--1754,
  2003.
\newblock Dynamical systems and functional equations (Murcia, 2000).

\bibitem{Mis}
M.~Misiurewicz.
\newblock Periodic points of maps of degree one of a circle.
\newblock {\em Ergodic Theory Dynamical Systems}, 2(2):221--227 (1983), 1982.

\bibitem{RT}
F.~Rhodes and C.~L. Thompson.
\newblock Rotation numbers for monotone functions on the circle.
\newblock {\em J. London Math. Soc. (2)}, 34(2):360--368, 1986.

\bibitem{Sar}
A.~N. Sharkovs{$'$}ki{\u\i}.
\newblock Co-existence of cycles of a continuous mapping of the line into
  itself.
\newblock {\em Ukrain. Mat. \u Z.}, 16:61--71, 1964.
\newblock (in Russian).

\bibitem{Shartrans}
A.~N. Sharkovs{$'$}ki{\u\i}.
\newblock Coexistence of cycles of a continuous map of the line into itself.
\newblock In {\em Thirty years after Sharkovski\u\i's theorem: new perspectives
  (Murcia, 1994)}, volume~8 of {\em World Sci. Ser. Nonlinear Sci. Ser. B Spec.
  Theme Issues Proc.}, pages 1--11. World Sci. Publ., River Edge, NJ, 1995.
\newblock Translated by J. Tolosa, Reprint of the paper reviewed in MR1361914
  (96j:58058).

\bibitem{Wa}
C.~T.~C. Wall.
\newblock {\em A geometric introduction to topology}.
\newblock Addison-Wesley Publishing Co., Reading, Mass.-London-Don Mills, Ont.,
  1972.

\bibitem{ZMGG}
F.~Zeng, H.~Mo, W.~Guo, and Q.~Gao.
\newblock {$\omega$}-limit set of a tree map.
\newblock {\em Northeast. Math. J.}, 17(3):333--339, 2001.

\bibitem{Zie}
K.~Ziemian.
\newblock Rotation sets for subshifts of finite type.
\newblock {\em Fund. Math.}, 146(2):189--201, 1995.

\end{thebibliography}

\newpage

\noindent
{\scshape Llu\'{\i}s Alsed\`a}\footnote{Partially supported by the by MEC 
grant number MTM2005-021329.} -- Departament de Matem\`{a}tiques,
Edifici Cc, Universitat Aut\`{o}noma de Barcelona,
08913 Cerdanyola del Vall\`es, Barcelona,
Spain\\ {\tt alseda@mat.uab.cat}\\
{\tt http://www.mat.uab.cat/{\tiny$\sim$}alseda/}

\medskip\noindent
{\scshape Sylvie Ruette}\footnote{Partially supported by the Marie Curie 
Fellowship number HPMF-CT-2002-02026 of the European Community programme Human
Potential.} -- Laboratoire de Math\'ematiques,
CNRS UMR 8628, B\^atiment 425,
Universit\'e Paris-Sud 11,
91405 Orsay cedex,
France\\
{\tt sylvie.ruette@math.u-psud.fr}\\
{\tt http://www.math.u-psud.fr/{\tiny$\sim$}ruette/}

\end{document}